\documentclass[a4paper,10pt]{article}
\usepackage[frenchb]{babel}
\usepackage[T1]{fontenc}
\usepackage[utf8]{inputenc}
\usepackage{amsmath}
\usepackage{array}
\usepackage{amsthm}
\usepackage{amssymb}
\input xy
\xyoption{all}

\usepackage{hyperref}

\newcommand{\A}{{\mathcal{A}}}
\newcommand{\B}{{\mathcal{B}}}
\newcommand{\C}{{\mathcal{C}}}
\newcommand{\D}{{\mathcal{D}}}

\newcommand{\G}{{\mathcal{G}}}
\newcommand{\fct}{\mathbf{Fct}}
\newcommand{\st}{\mathbf{St}}

\newcommand{\E}{{\mathcal{E}}}

\newcommand{\Pol}{{\mathcal{P}}ol}
\newcommand{\s}{{\mathcal{S}}}

\newcommand{\K}{{\mathcal{K}}}

\newcommand{\M}{{\mathcal{M}}}
\newcommand{\N}{{\mathcal{N}}}

\newcommand{\col}{{\rm colim}\,}
\newcommand{\T}{{\mathcal{T}}}

\newcommand{\mn}{\M on_{\rm nul}}
\newcommand{\mi}{\M on_{\rm ini}}

\title{Des propriétés de finitude des foncteurs polynomiaux}
\author{Aur\'elien DJAMENT\thanks{CNRS, laboratoire de mathématiques Jean Leray (UMR 6629), 2 rue de la Houssinière, BP 92208, 44322 NANTES CEDEX 3, FRANCE ; aurelien.djament@univ-nantes.fr.}}

\newtheorem{thi}{Th\'eor\`eme}
\newtheorem{pin}{Proposition}

\theoremstyle{definition}
\newtheorem{noti}{Notation}

\newtheorem{thm}{Th\'eor\`eme}[section]
\newtheorem{pr}[thm]{Proposition}
\newtheorem{cor}[thm]{Corollaire}
\newtheorem{lm}[thm]{Lemme}

\newtheorem{prdef}[thm]{Proposition et d\'efinition}

\theoremstyle{definition}
\newtheorem{defi}[thm]{D\'efinition}

\newtheorem{conv}[thm]{Convention}

\theoremstyle{remark}
\newtheorem{rem}[thm]{Remarque}
\newtheorem{ex}[thm]{Exemple}

\begin{document}

\maketitle

\begin{abstract}
On étudie différentes propriétés de finitude, notamment la propriété noethérienne, la dimension de Krull et une variante de la présentation finie, dans des catégories de foncteurs polynomiaux d'une petite catégorie monoïdale symétrique dont l'unité est objet initial vers une catégorie abélienne (notion introduite dans \cite{DV3}). On montre notamment que les foncteurs polynomiaux depuis la catégorie des groupes abéliens libres $\mathbb{Z}^n$ avec monomorphismes scindés vers les groupes abéliens forment <<~presque~>> une catégorie localement noethérienne. On donne également une application à des foncteurs liés aux automorphismes des groupes libres.
\end{abstract}

\begin{small}
\begin{center}
 \textbf{Abstract}
\end{center}

We study finiteness properties, especially the noetherian property, the Krull dimension and a variation of finite presentation, in categories of polynomial functors from a small symmetric monoidal category whose unit is an initial object to an abelian category (notion introduced in \cite{DV3}). We prove in particular that the category of polynomial functors from the category of free abelian groups $\mathbb{Z}^n$ with split monomorphisms to abelian groups is ``almost'' locally noetherian. We give also an application to functors related to automorphisms of free groups.
\end{small}

\medskip

\noindent
{\em Mots clefs} : foncteurs polynomiaux ; objets noethériens ; catégories abéliennes ; catégories quotients ; dimension de Krull.

\smallskip

\noindent
{\em Classification MSC 2010} : 18A25 ; 18D10 ; 18E15 ; 18E35.

\section*{Introduction}

Cet article est consacré à l'étude des propriétés de finitude des foncteurs polynomiaux depuis une petite catégorie monoïdale symétrique $\M$ dont l'unité est objet initial vers une catégorie abélienne raisonnable $\A$. Les foncteurs polynomiaux dans ce contexte sont définis et étudiés dans \cite{DV3} ; deux exemples s'avèrent particulièrement intéressants : le cas où la catégorie source $\M$ est la catégorie $\Theta$ des ensembles finis avec injections (notée FI dans \cite{CEF} ou \cite{CEFN}) et le cas, plus délicat, où la catégorie source est la catégorie des objets hermitiens (non dégénérés) sur une petite catégorie additive à dualité --- notamment le cas particulier de la catégorie $\mathbf{S}(\mathbb{Z})$ (définie dans le théorème~\ref{thfi} ci-dessous). Le cas de la catégorie $\Theta$ a déjà été étudié par Church, Ellenberg, Farb et Nagpal dans \cite{CEFN}, où les auteurs montrent que la catégorie des foncteurs depuis cette catégorie vers la catégorie des modules sur un anneau noethérien est localement noethérienne (il n'y a pas de condition polynomiale ici car cette catégorie est engendrée par des foncteurs projectifs de type fini polynomiaux, ce qui en rend l'étude particulièrement favorable). De plus, dans \cite{CEF}, les auteurs montrent l'ubiquité des foncteurs de ce type. D'un autre côté, les foncteurs depuis une catégorie d'objets hermitiens vers une catégorie abélienne apparaissent très naturellement quand on s'intéresse à l'homologie des groupes de congruences ou des sous-groupes $IA$ des automorphismes des groupes libres induisant l'identité sur l'abélianisation ou les sous-quotients des filtrations centrales usuelles sur les automorphismes des groupes libres. Comme le montre l'article \cite{DV3}, l'étude des foncteurs polynomiaux sur de telles catégories est plus difficile que dans le cas de $\Theta$, mais accessible --- on obtient notamment une classification des foncteurs polynomiaux de degré au plus $d$ modulo les foncteurs polynomiaux de degré au plus $d-1$ se ramenant à une catégorie quotient analogue dans le cas bien connu d'une catégorie source additive.

L'un des nos résultats principaux est le suivant.

\begin{thi}\label{thfi}
 Soient $A$ un anneau, $\mathbf{S}(A)$ la catégorie des $A$-modules libres de rang fini avec monomorphismes scindés et $\A$ une catégorie de Grothendieck localement noethérienne.
\begin{enumerate}
 \item Si $A$ est fini, alors la catégorie des foncteurs faiblement polynomiaux $\mathbf{S}(A)\to\A$ est localement noethérienne.
\item Pour $A=\mathbb{Z}$, la catégorie des foncteurs faiblement polynomiaux $\mathbf{S}(\mathbb{Z})\to\A$ est localement presque noethérienne.
\end{enumerate}
\end{thi}

({\em Localement presque noethérienne} signifie : engendrée par un ensemble de foncteurs presque noethériens. Un foncteur $F : \mathbf{S}(\mathbb{Z})\to\A$ est dit presque noethérien s'il existe un entier $n$ tel que la restriction de $F$ à la sous-catégorie pleine $\mathbf{S}(\mathbb{Z})_{\geq n}$ des groupes abéliens libres de rang $\geq n$ est noethérienne. Un foncteur presque noethérien qui prend des valeurs noethériennes est noethérien. La catégorie des foncteurs faiblement polynomiaux $\mathbf{S}(\mathbb{Z})\to\mathbf{Ab}$ n'est pas localement noethérienne car, pour $n\geq 2$, l'anneau de groupe $\mathbb{Z}[GL_n(\mathbb{Z})]$ n'est pas noethérien.)

La notion de foncteur faiblement polynomial est rappelée en début d'article. Elle s'oppose à celle de foncteur fortement polynomial, peut-être plus intuitive mais manquant de propriétés de stabilité essentielles (un sous-foncteur d'un foncteur fortement polynomial n'est pas nécessairement fortement polynomial). De fait, la notion d'objet polynomial la plus naturelle n'intervient pas dans la catégorie de foncteurs elle-même, mais dans une catégorie quotient appropriée ; les foncteurs faiblement polynomiaux sont ceux dont l'image dans cette catégorie quotient est polynomiale.

Une première étape importante pour établir le théorème~\ref{thfi} consiste à étudier d'abord les propriétés noethériennes des foncteurs fortement polynomiaux. Cela s'avère nettement plus facile (et nécessite beaucoup moins d'hypothèses sur la catégorie monoïdale source) ; le c\oe ur du raisonnement est déjà présent dans le travail \cite{CEFN} (pour les foncteurs {\em fortement} polynomiaux, on peut se ramener à la catégorie source $\Theta$). Pour aborder les foncteurs faiblement polynomiaux, nous utilisons un raisonnement plus détourné. Un ingrédient essentiel est donné par une propriété du foncteur section de la catégorie quotient susmentionnée vers la catégorie de foncteurs sur laquelle on travaille, ainsi que de ses dérivés. Cette propriété, qui présente un intérêt intrinsèque, se montre en utilisant un résultat important de \cite{DV3} et des récurrences enchevêtrées. Elle utilise également des propriétés voisines de la présentation finie, notions qui se révèlent aussi importantes pour en déduire le théorème~\ref{thfi}. Ces propriétés nécessitent un certain nombre d'hypothèses techniques sur les catégories sources des foncteurs polynomiaux qu'on rencontre : en effet, une difficulté apparaît du fait que les foncteurs de décalage (i.e. de précomposition par un endofoncteur $x+-$ de $\M$) ne préservent pas forcément les foncteurs de type fini (même pour $\M=\mathbf{S}(\mathbb{Z})$). On contourne le problème en raisonnant souvent avec la notion de {\em foncteur à support fini} (voir le §\,\ref{sect-sk}).

La notion de foncteur à support fini conduit naturellement à un affaiblissement de la notion de foncteur de présentation finie, que nous nommons la {\em présentation de support fini}, qui se traduit commodément en termes d'extensions de Kan (voir la proposition-définition~\ref{pres-suppFini}). Les résultats qui mènent au théorème~\ref{thfi} nous permettent également de montrer que, sous des hypothèses raisonnables, les foncteurs polynomiaux sont à présentation de support fini. Nous en donnons deux applications dans la section~\ref{sappl}. Voici l'énoncé informel de l'une d'entre elles (voir la proposition~\ref{th-and} et le paragraphe qui la précède pour les définitions et un énoncé précis) :

\begin{pin}\label{pai}
 Pour tout groupe libre $G$ et tout entier $n$, notons $\A_n(G)$ le $n$-ème étage de la filtration d'Andreadakis du groupe des automorphismes de $G$. Le groupe abélien $\A_n(G)/\A_{n+1}(G)$ ne dépend fonctoriellement que de l'abélianisation $G_{ab}$ de $G$, notons $F_n : G_{ab}\mapsto \A_n(G)/\A_{n+1}(G)$ le foncteur ainsi obtenu. Alors pour tout $n\in\mathbb{N}$, il existe $N\in\mathbb{N}$ tel que le morphisme canonique
 $$\underset{W=(R,S)\in\G r^s_{\leq N}(V)}{\col}F_n(R)\to F_n(V)$$ 
 soit un isomorphisme pour tout groupe abélien libre de rang fini $V$, où $\G r^s_{\leq N}(V)$ désigne l'ensemble ordonné des décompositions $V=R\oplus S$, où $R$ est de rang au plus $N$.
\end{pin}

Dans la section~\ref{seckrul}, on applique la propriété fondamentale susmentionnée du foncteur section pour obtenir des résultats sur la dimension de Krull de catégories d'objets polynomiaux. Cela permet de voir que la catégorie des foncteurs de $\Theta$ vers une catégorie d'espaces vectoriels est de dimension de Krull $1$ (voir le corollaire~\ref{cor-k2}) ; on obtient également le résultat suivant (cf. corollaires~\ref{cor-k1} et~\ref{cor-k2}).

\begin{thi}
\begin{enumerate}
 \item Si $A$ est un anneau fini, tout foncteur faiblement polynomial de $\mathbf{S}(A)$ vers une catégorie d'espaces vectoriels est de dimension de Krull au plus $1$.
\item Si $\A$ est une catégorie de Grothendieck localement noethérienne, la catégorie des objets polynomiaux de $\st(\mathbf{S}(\mathbb{Z}),\A)$ a la même dimension de Krull que $\A$.
\end{enumerate}
\end{thi}
(La définition de la catégorie $\st(\mathbf{S}(\mathbb{Z}),\A)$, quotient de la catégorie des foncteurs $\mathbf{S}(\mathbb{Z})\to\A$, est rappelée dans la première section.)

Ce résultat est frappant dans la mesure où les problèmes de dimension de Krull dans les catégories de foncteurs sont souvent très difficiles (cf. \cite{Dja}).

Dans l'appendice, on étudie une classe de foncteurs polynomiaux dits {\em parfaits} et on en donne des caractérisations (théorème~\ref{th-nag}) dans le cas de la catégorie source $\Theta$. Celles-ci s'inspirent des travaux de Nagpal \cite{Nag} (ainsi que de \cite{NSS}, \cite{Nag-prive} et \cite{Ram}), motivés par l'étude de phénomènes de périodicité en cohomologie des groupes symétriques et des espaces de configurations, et les recoupent largement. Notre théorème~\ref{th-nag} s'inspire aussi beaucoup de l'article \cite{GP-cow} de Powell.

\smallskip

Après la rédaction de la première version de ce travail, Putman et Sam ont démontré \cite{PSam} un résultat beaucoup plus fort que la première assertion du théorème~\ref{thfi}, lorsque l'anneau fini $A$ est commutatif, à savoir que la catégorie de {\em tous} les foncteurs de $\mathbf{S}(A)$ vers une catégorie de Grothendieck localement noethérienne est localement noethérienne (lorsque la catégorie but est celle des espaces vectoriels sur un corps de caractéristique nulle, ce résultat a été également obtenu par Gan et Li \cite{GL}, d'une manière indépendante à la fois de \cite{PSam} et du présent travail). Leur méthode n'utilise pas d'algèbre homologique (et nulle catégorie quotient n'y apparaît) mais repose sur un astucieux changement de catégorie source et des arguments combinatoires inspirés des bases de Gröbner (voir aussi le travail \cite{SSn} de Sam et Snowden, relié à \cite{PSam}). Toutefois, il paraît très difficile d'adapter les techniques de \cite{PSam} pour établir la deuxième assertion du théorème~\ref{thfi}, et nos résultats sur les dérivés du foncteur section, la présentation de support fini ou la dimension de Krull en semblent également disjoints.

\paragraph*{Remerciements} Cet article s'inscrit dans le prolongement du travail \cite{DV3} avec Christine Vespa, il a bénéficié de ses encouragements et remarques pertinentes. L'auteur sait également gré à Steven Sam de la communication privée d'une version préliminaire de \cite{SSn}
et de nombreux échanges fructueux sur la propriété noethérienne dans les catégories de foncteurs, ainsi qu'à Rohit Nagpal de discussions autour de son article \cite{Nag}, qui sont à l'origine de l'appendice. Il témoigne aussi sa reconnaissance à Jacques Darné, dont la lecture vigilante de versions antérieures de ce travail en a permis plusieurs améliorations, notamment la correction d'une erreur dans la définition d'objet de type fini.  Il remercie enfin Thomas Church et Jordan Ellenberg pour des discussions autour des articles \cite{CEF} et \cite{CEFN}, Yves de Cornulier pour lui avoir indiqué l'ouvrage \cite{MCR}, et le rapporteur anonyme dont la lecture attentive a permis de notables améliorations des premières versions de ce texte.

\begin{noti}
 Dans tout cet article, si $A$ est un objet d'une catégorie additive $\A$ et $E$ un ensemble fini, on note $A[E]$ la somme directe de copies de $A$ indexées par $E$. On étend cette notation au cas où $E$ est infini si $\A$ possède des sommes infinies.
\end{noti}

\begin{noti}\label{nca}
 Soient $\A$ une catégorie et $G$ un monoïde. On note $\A_G$ la catégorie des objets de $\A$ munis d'une action de $G$, appelés aussi représentations de $G$ dans $\A$ (si $\A$ est abélienne), c'est-à-dire la catégorie dont les objets sont les couples $(A,f)$ où $A$ est un objet de $\A$ et $f : G\to {\rm End}_\A(A)$ un morphisme de monoïdes, les morphismes $(A,f)\to (B,g)$ étant les morphismes $\xi : A\to B$ de $\A$ tels que, pour tout $u\in G$, le diagramme
 $$\xymatrix{A\ar[r]^-{f(u)}\ar[d]_\xi & A\ar[d]^\xi \\
 B\ar[r]^-{g(u)} & B
 }$$
 de $\A$ commute.
\end{noti}

Si l'on voit $G$ comme une catégorie à un objet, cette catégorie s'identifie à la catégorie des foncteurs de $G$ dans $\A$.

\tableofcontents

\section{Notations et rappels sur les foncteurs polynomiaux (cf. \cite{DV3})}

(Nos références générales pour les notions catégoriques sont \cite{ML} et, pour les catégories abéliennes, catégories quotients comprises, \cite{Gab}.)

\smallskip

Si $\A$ et $\C$ sont des catégories, avec $\C$ petite, on note $\fct(\C,\A)$ la catégorie des foncteurs de $\C$ vers $\A$. 

On note $\mi$ la catégorie des petites catégories monoïdales symétriques $(\M,+,0)$ dont l'unité $0$ est objet initial, les morphismes étant les foncteurs monoïdaux au sens fort, et $\mn$ la sous-catégorie pleine des catégories telles que $0$ soit objet nul de $\M$. Les petites catégories additives seront vues comme des objets de $\mn$, $+$ étant alors la somme directe. L'exemple qui nous intéressera le plus est celui de la catégorie $\mathbf{P}(A)$, où $A$ est un anneau, des $A$-modules à droite libres de rang fini (ou plutôt un squelette de cette catégorie).

Si $\A$ est une petite catégorie additive, on note $\mathbf{S}(\A)$ la catégorie ayant les mêmes objets et dont les morphismes sont les monomorphismes scindés de $\A$, {\em le scindage étant donné dans la structure}. Munie du foncteur induit par la somme directe de $\A$, $\mathbf{S}(\A)$ est un objet de $\mi$ qui nous intéresse particulièrement. Cette catégorie est un cas particulier de catégorie d'objets hermitiens sur une petite catégorie additive à dualité (en l'occurence, $\A^{op}\times\A$ avec l'échange des deux facteurs) ; la plupart des considérations relatives à $\mathbf{S}(\A)$ pourraient s'étendre à cette situation plus générale (cf. \cite{DV3}). Si $A$ est un anneau, on notera $\mathbf{S}(A)$ pour $\mathbf{S}(\mathbf{P}(A))$. Un autre exemple fondamental d'objet de $\mi$ est la catégorie $\Theta$ des ensembles finis avec injections, ou plutôt son squelette constitué des objets $\mathbf{n}:=\{1,\dots,n\}$ pour $n\in\mathbb{N}$.

 Pour toute catégorie $\A$ et tout objet $x$ d'un objet $\M$ de $\mi$, on note $\tau_x$ l'endofoncteur de $\fct(\M,\A)$ de précomposition par l'endofoncteur $x+-$ de $\M$. On dispose ainsi d'isomorphismes naturels $\tau_x\tau_y\simeq\tau_y\tau_x\simeq\tau_{x+y}$. Si $\A$ est une catégorie abélienne, on note $\delta_x$ (resp. $\kappa_x$) l'endofoncteur de $\fct(\M,\A)$ conoyau (resp. noyau) de la transformation naturelle $i_x : {\rm Id}\simeq\tau_0\to\tau_x$ induite par l'unique morphisme $0\to x$. Ainsi, par le lemme du serpent, toute suite exacte courte $0\to F\to G\to H\to 0$ de $\fct(\M,\A)$ induit une suite exacte
$$0\to\kappa_x(F)\to\kappa_x(G)\to\kappa_x(H)\to\delta_x(F)\to\delta_x(G)\to\delta_x(H)\to 0,$$
ce qu'on utilisera abondamment dans la suite.

On dit qu'un foncteur $F\to\A$ est {\em fortement polynomial} de degré fort au plus $d$ si $\delta_{a_0}\dots\delta_{a_d}(F)=0$ pour tous objets $a_0,\dots,a_d$ de $\M$. On note $\Pol_d^{{\rm fort}}(\M,\A)$ la sous-catégorie pleine de $\fct(\M,\A)$ formée de ces foncteurs (par convention, cette catégorie est réduite à $\{0\}$ pour $d<0$). Cette sous-catégorie est stable par quotients et extensions, mais pas par sous-objets en général.

Supposons que $\A$ est une catégorie de Grothendieck\,\footnote{Pour plusieurs de nos considérations, des hypothèses plus faibles sur cette catégorie abélienne suffiraient, mais cela ne serait guère utile pour les applications.}. On note $\kappa$ le sous-foncteur de l'identité de $\fct(\M,\A)$ somme des $\kappa_x$ pour $x\in {\rm Ob}\,\M$ et l'on note $\s n(\M,\A)$ la sous-catégorie pleine des foncteurs $F$ tels que l'inclusion $\kappa(F)\subset F$ soit une égalité. C'est une sous-catégorie localisante de $\fct(\M,\A)$ ; ses objets sont appelés foncteurs {\em stablement nuls}. La catégorie quotient $\fct(\M,\A)/\s n(\M,\A)$ est notée $\st(\M,\A)$ ; le foncteur canonique $\fct(\M,\A)\to\st(\M,\A)$ est noté $\pi$. Son adjoint à droite, le foncteur section, est noté $s$.

Pour tout objet $x$ de $\M$, les endofoncteurs $\tau_x$ (exact) et $\delta_x$ (exact seulement à droite en général) de $\fct(\M,\A)$ induisent des endofoncteurs exacts encore notés $\tau_x$ et $\delta_x$ de $\st(\M,\A)$. Un objet $X$ de $\st(\M,\A)$ tel que $\delta_{a_0}\dots\delta_{a_d}(X)=0$ pour tous objets $a_0,\dots,a_d$ de $\M$ est dit {\em polynomial} de degré au plus $d$ ; on note $\Pol_d(\M,\A)$ la sous-catégorie pleine de ces objets. Elle est bilocalisante. Un foncteur $F$ de $\fct(\M,\A)$ est dit {\em faiblement polynomial} de degré faible au plus $d$ si son image $\pi(F)$ dans $\st(\M,\A)$ appartient à $\Pol_d(\M,\A)$ ; on note $\Pol^{{\rm faible}}_d(\M,\A)$ la sous-catégorie pleine de ces foncteurs. C'est une sous-catégorie localisante de $\fct(\M,\A)$, et $\Pol^{{\rm faible}}_d(\M,\A)=\s n(\M,\A)$ pour $d<0$. On montre que l'inclusion des foncteurs constants induit une équivalence de catégories entre $\A$ et $\Pol_0(\M,\A)$ (la restriction du foncteur section en est un quasi-inverse). Le foncteur $s$ commute à isomorphisme naturel près aux foncteurs $\tau_x$, tandis qu'on dispose de monomorphismes naturels $\delta_x s\hookrightarrow s\delta_x$.

Bien sûr, si $\M$ est dans $\mn$, alors $\s n(\M,\A)$ est réduite à $0$, de sorte que les notions de foncteur fortement et faiblement polynomial coïncident.

Les foncteurs $\tau_x$ commutent, à isomorphisme naturel près, aux foncteurs $\delta_t$ et $\kappa_t$. En particulier, les différentes catégories de foncteurs polynomiaux sont stables par les $\tau_x$.

Nous utiliserons plus tard la propriété élémentaire suivante du foncteur section (par l'intermédiaire de son corollaire), que le corollaire~\ref{cor-pft} généralisera.

\begin{pr}\label{caract-sect}
 Soient $\M$ un objet de $\mi$, $\A$ une catégorie de Grothendieck et $F : \M\to\A$ un foncteur. Alors l'unité $F\to s\pi(F)$ de l'adjonction entre $s$ et $\pi$ est un isomorphisme si et seulement si $\kappa_x(F)$ et $\kappa_x(\delta_x(F))$ sont nuls pour tout objet $x$ de $\M$.
\end{pr}

\begin{proof}
 La condition $\kappa_x(F)=0$ pour tout objet $x$ de $\M$ équivaut à la nullité de ${\rm Hom}_{\fct(\M,\A)}(N,F)$ pour $N$ stablement nul, tandis que l'unité $F\to s\pi(F)$ est un isomorphisme si et seulement si ${\rm Ext}^i_{\fct(\M,\A)}(N,F)$ est nul pour $i\leq 1$ et $N$ stablement nul. Cette condition équivaut encore à la nullité ${\rm Ext}^i_{\fct(\M,\A)}(N,F)$ pour tout $i\leq 1$, tout $x\in{\rm Ob}\,\M$ et $N$ tel que $i_x(N) : N\to\tau_x(N)$ soit nul (pour $i=0$, cela équivaut à $\kappa_x(F)=0$).
 
 Si $\kappa_x(F)=0$ et $i_x(N)=0$, la suite exacte $0\to F\to\tau_x(F)\to\delta_x(F)\to 0$ induit un isomorphisme ${\rm Hom}(N,\delta_x(F))\xrightarrow{\simeq}{\rm Ext}^1(N,F)$, car ${\rm Hom}(N,\tau_x(F))=0$ (en effet, $\kappa_x(\tau_x(F))\simeq\tau_x(\kappa_x(F))=0$) et le morphisme ${\rm Ext}^1(N,F)\to{\rm Ext}^1(N,\tau_x(F))$ induit par $i_x(F)$ est nul en raison du diagramme commutatif suivant.
 $$\xymatrix{{\rm Ext}^1(N,F)\ar[r]\ar[rd]_-{{\rm Ext}^1(N,i_x(F))} & {\rm Ext}^1(\tau_x(N),\tau_x(F))\ar[d]^-{{\rm Ext}^1(i_x(N),F)=0} \\
 & {\rm Ext}^1(N,\tau_x(F))
 }$$
 Cet isomorphisme ${\rm Hom}(N,\delta_x(F))\xrightarrow{\simeq}{\rm Ext}^1(N,F)$, lorsque $\kappa_x(F)=0$ et $i_x(N)=0$, et le fait que l'unité $F\to s\pi(F)$ est un isomorphisme si et seulement si $\kappa_x(F)=0$ pour tout $x\in {\rm Ob}\,\M$ et ${\rm Ext}^1(N,F)=0$ pour tous objets $x$ de $\M$ et $N$ de $\fct(\M,\A)$ tels que $i_x(N)=0$ montrent que ces conditions équivalent à $\kappa_x(F)=0$ et $\kappa_x(\delta_x(F))=0$ pour tout $x\in {\rm Ob}\,\M$, comme souhaité.
\end{proof}

\begin{cor}\label{corsec}
 Soient $\M$ un objet de $\mi$ et $\A$ une catégorie de Grothendieck. Le foncteur section $s : \st(\M,\A)\to\fct(\M,\A)$ commute aux colimites filtrantes.
\end{cor}

\begin{proof}
 Cela découle de la proposition précédente et de ce que, pour tout $x\in {\rm Ob}\,\M$ :
 \begin{enumerate}
  \item la condition $\kappa_x(F)=0$ équivaut à $i_x(F) : F\to\tau_x(F)$ injectif, condition préservée par colimites filtrantes parce que $\tau_x$ commute aux colimites et que les colimites filtrantes sont exactes dans $\fct(\M,\A)$ ;
  \item le foncteur $\delta_x$ commute aux colimites. 
 \end{enumerate}\end{proof}

À tout objet $(\M,+,0)$ de $\mi$ on associe une catégorie $(\widetilde{\M},+,0)$ dans $\mn$ qui a les mêmes objets et dont les morphismes sont donnés par
$$\widetilde{\M}(a,b):=\underset{\M}{\col}\tau_b\M(a,-)\;;$$
on dispose d'un foncteur monoïdal (au sens fort) $\M\to\widetilde{\M}$ qui est l'identité sur les objets. Si $\A$ est une catégorie abélienne, on notera $\eta : \fct(\widetilde{\M},\A)\to\fct(\M,\A)$ la précomposition par ce foncteur (attention, dans \cite{DV3}, c'est le foncteur $\M\to\widetilde{\M}$ qui est noté $\eta$, de sorte que celui qu'on note maintenant $\eta$ y figure comme $\eta^*$). Le foncteur $\eta$ possède un adjoint à gauche $\alpha$ tel que
\begin{equation}\label{eqalph}
 \alpha(F)(t)=\underset{\M}{\col}\tau_t(F).
\end{equation}
L'unité $\eta(F)\to s\pi\eta(F)$ est un isomorphisme pour tout foncteur $F$ de $\fct(\widetilde{\M},\A)$, tandis que les foncteurs dérivés à droite $\mathbf{R}^i s$ du foncteur section sont nuls, pour $i>0$, sur un objet du type $\pi\eta(F)$. Le foncteur $\alpha$ et ses dérivés à gauche sont nuls sur $\s n(\M,\A)$, de sorte que $\alpha$ induit un foncteur $\st(\M,\A)\to\fct(\widetilde{\M},\A)$ qui est adjoint à gauche à $\pi\eta$.

Les foncteurs $\eta$ et $\alpha$ commutent (à isomorphisme naturel près) aux foncteurs $\tau_x$ et $\delta_x$ ; $\eta$ envoie $\Pol_d(\widetilde{\M},\A)$ dans $\Pol_d^{{\rm fort}}(\M,\A)$ et $\alpha$ envoie $\Pol_d^{{\rm faible}}(\M,\A)$ dans $\Pol_d(\widetilde{\M},\A)$. Si $X$ est un objet de $\Pol_d(\M,\A)$, alors le noyau et le conoyau de l'unité $X\to\eta\alpha(X)$ appartiennent à $\Pol_{d-1}(\M,\A)$ ; en conséquence, les foncteurs adjointes $\alpha$ et $\eta$ induisent des équivalences de catégories quasi-inverses l'une de l'autre entre $\Pol_d(\M,\A)/\Pol_{d-1}(\M,\A)$ et $\Pol_d(\widetilde{\M},\A)/\Pol_{d-1}(\widetilde{\M},\A)$.

\section{Propriétés de finitude dans les catégories de foncteurs vers une catégorie abélienne}

\subsection{Objets de type fini, de présentation finie, noethériens}

Commençons par des rappels généraux sur les propriétés de finitude dans une catégorie abélienne $\A$. Un objet $A$ de $\A$ est dit {\em noethérien} si toute suite croissante de sous-objets de $\A$ stationne, ce qui implique que toute famille filtrante croissante de sous-objets de $\A$ stationne. Il est dit {\em de type fini} si toute famille filtrante croissante de sous-objets de $A$ dont la réunion est $A$ stationne.

Lorsque $\A$ possède des colimites et que les colimites filtrantes y sont exactes (par exemple, si $\A$ est une catégorie de Grothendieck), un objet est noethérien si et seulement si tous ses sous-objets sont de type fini. La catégorie abélienne $\A$ est dite {\em localement noethérienne} (resp. {\em localement de type fini}) si elle possède un ensemble de générateurs noethériens (resp. de type fini).

La proposition suivante, classique, est laissée en exercice.
\begin{pr}
Supposons que $\A$ est une catégorie abélienne avec colimites filtrantes exactes. Soit $A$ un objet de $\A$. Les assertions suivantes sont équivalentes.
\begin{enumerate}
 \item L'objet $A$ de $\A$ est de type fini.
 \item Pour tout foncteur $F : \T\to\A$, où $\T$ est une petite catégorie filtrante, le morphisme canonique de groupes abéliens
 $$\underset{\T}{\col}\A(A,-)\circ F\to\A(A,\underset{\T}{\col}F)$$
 est injectif.
 \item Le foncteur $\A(A,-)$ commute aux colimites filtrantes de monomorphismes, c'est-à-dire que le morphisme précédent est un isomorphisme si $F$ envoie chaque flèche de $\T$ sur un monomorphisme ($\T$ étant toujours une petite catégorie filtrante).
\end{enumerate}
\end{pr}

Cette proposition implique que, dans une catégorie abélienne avec colimites filtrantes exactes, la classe des objets de type fini est stable par quotients et par extensions.

Si $\E$ est une classe de générateurs d'une catégorie abélienne avec colimites filtrantes exactes $\A$, alors tout foncteur de type fini est quotient d'une somme directe {\em finie} d'éléments de $\E$. Si les éléments de $\E$ sont de type fini, la réciproque est vraie, d'après ce qui précède.

On rappelle qu'un objet $A$ d'une catégorie abélienne $\A$ est dit {\em de présentation finie} si le foncteur $\A(A,-) : \A\to\mathbf{Ab}$ commute aux colimites filtrantes. Il est classique (et facile) que si, dans une suite exacte courte $0\to C\to B\to A\to 0$, $A$ est de présentation finie et $B$ de type fini, alors $C$ est de type fini. Dans une catégorie localement noethérienne, tout objet de type fini est noethérien et de présentation finie.

\begin{cor}\label{crtfe}
 Soit $\Phi : \A\to\B$ un foncteur entre catégories abéliennes avec colimites filtrantes exactes. On suppose que $\Phi$ possède un adjoint à droite qui commute aux colimites filtrantes. Alors $\Phi$ préserve les objets de type fini, ainsi que les objets de présentation finie.
\end{cor}

Appliquant ce résultat au foncteur $\alpha$ de la section précédente, dont l'adjoint à droite $\eta$ commute aux colimites, nous obtenons la propriété suivante, qui nous sera utile plus tard.

\begin{pr}\label{lmatf}
Soient $\M$ un objet de $\mi$, $\A$ une catégorie de Grothendieck et $F : \M\to\A$ un foncteur de type fini (resp. de présentation finie). Alors $\alpha(F) : \widetilde{\M}\to\A$ est également de type fini (resp. de présentation finie).
\end{pr}

En combinant les corollaires~\ref{crtfe} et~\ref{corsec}, on obtient également :
\begin{cor}\label{pi-tf}
Soient $\M$ un objet de $\mi$ et $\A$ une catégorie de Grothendieck. Le foncteur canonique $\pi : \fct(\M,\A)\to\st(\M,\A)$ préserve les objets de type fini et les objets de présentation finie.
\end{cor}

\subsection{Propriétés de finitude ponctuelles des foncteurs}

Donnons maintenant une définition générale pour traiter de propriétés de finitude dans des catégories de foncteurs.

\begin{defi}\label{df-finitude}
 Soient $\C$ une petite catégorie, $\A$ une catégorie abélienne et $F : \C\to\A$ un foncteur.
\begin{enumerate}
 \item On dit que $F$ est {\em ponctuellement de type fini} (resp. {\em ponctuellement noethérien}) si $F(c)$ est un objet de type fini (resp. noethérien) dans la catégorie $\A$ pour tout objet $c$ de $\C$.
 \item On dit que $F$ est {\em faiblement ponctuellement de type fini} (resp. {\em faiblement ponctuellement noethérien}) si, pour tout objet $c$ de $\C$, $F(c)$ est un objet de type fini (resp. noethérien) dans la catégorie $\A_{{\rm End}_\C(c)}$ des objets de $\A$ munis d'une action du monoïde ${\rm End}_\C(c)$.
\end{enumerate}
\end{defi}

On utilisera aussi la propriété suivante pour les catégories sources.

\begin{defi}\label{dfm}
 On dit qu'une catégorie $\C$ vérifie la propriété $(FM)$ si pour tous objets $t$ et $x$ de $\M$, le ${\rm End}_\C(x)$-ensemble $\C(t,x)$ est de type fini.
\end{defi}

\begin{pr}\label{pf-gal} Soient $\C$ une petite catégorie, $\A$ une catégorie abélienne avec colimites filtrantes exactes et $F : \C\to\A$ un foncteur.
 \begin{enumerate}
 \item Si $F$ est un quotient d'une somme directe finie de foncteurs du type $A[\C(t,-)]$, où $t$ est un objet de $\C$ et $A$ un objet de type fini de $\A$, alors $F$ est de type fini.
\item Si $\A$ est localement de type fini et $F$ de type fini, alors $F$ est un quotient d'une somme directe finie de foncteurs du type $A[\C(t,-)]$, où $t$ est un objet de $\C$ et $A$ un objet de type fini de $\A$.
\item Si $\A$ est localement de type fini, alors il en est de même pour $\fct(\C,\A)$.
\item Supposons que $\C$ vérifie la propriété $(FM)$ et que $\A$ est localement de type fini. Alors le foncteur $F$ est faiblement ponctuellement de type fini s'il est de type fini.
 \end{enumerate}
\end{pr}

\begin{proof}
Il est classique que le lemme de Yoneda fournit un isomorphisme
$$\fct(\C,\A)(A[\C(t,-)],F)\simeq\A(A,F(t))$$
naturel en les objets $A$ de $\A$, $t$ de $\C$ et $F$ de $\fct(\C,\A)$. Par conséquent, si $A$ est de type fini dans $\A$, $A[\C(t,-)]$ est de type fini dans $\fct(\C,\A)$ (cf. le corollaire~\ref{crtfe}) ; de plus, cet isomorphisme montre que si $A$ parcourt un ensemble de générateurs de $\A$ et $t$ les objets de $\C$, alors les $A[\C(t,-)]$ engendrent $\fct(\C,\A)$. Cela implique les trois premiers points.

Comme tout quotient d'une somme directe finie de foncteurs faiblement ponctuellement de type fini est faiblement ponctuellement de type fini, ce qui précède montre également qu'il suffit de prouver, pour établir la dernière assertion, que $A[\C(t,-)]$ est faiblement ponctuellement de type fini lorsque $A$ est de type fini et que $\C$ vérifie $(FM)$. En effet, si $G$ est un monoïde et $E$ un $G$-ensemble de type fini, $A[E]$ est de type fini dans la catégorie $\A_G$, à cause de l'isomorphisme naturel
$$\A_G(A[G],M)\simeq\A(A,Ou(M))$$
où $Ou : \A_G\to\A$ est le foncteur d'oubli (appliquer encore le corollaire~\ref{crtfe}) et du fait que $A[E]$ est quotient d'une somme directe finie de copies de $A[G]$ si $E$ est un $G$-ensemble de type fini.
\end{proof}

\begin{rem}\label{rq-fingal}
 La propriété $(FM)$ sur $\C$, suffisante pour que les foncteurs de type fini $\C\to\A$ soient faiblement ponctuellement de type fini, est aussi nécessaire, au moins si $\A=\mathbf{Ab}$, comme on le voit en considérant les générateurs projectifs de type fini $\mathbb{Z}[\C(c,-)]$. Même en se restreignant à des foncteurs polynomiaux depuis une catégorie dans $\mn$, les foncteurs de type fini ne sont pas nécessairement faiblement ponctuellement de type fini. Un exemple est donné par le foncteur
$$\mathbf{P}(\mathbb{Z}[t])\to\mathbf{Ab}\qquad V\mapsto V^{\underset{\mathbb{Z}}{\otimes} 2},$$
qui est polynomial de degré $2$, de type fini (en raison de l'épimorphisme
$$\mathbb{Z}[\mathbf{P}(\mathbb{Z}[t])(\mathbb{Z}[t]^2,V)]\simeq\mathbb{Z}[V]^{\underset{\mathbb{Z}}{\otimes} 2}\twoheadrightarrow V^{\underset{\mathbb{Z}}{\otimes} 2}$$
naturel en le $\mathbb{Z}[t]$-module $V$ obtenu en prenant la deuxième puissance tensorielle de la linéarisation $\mathbb{Z}[V]\twoheadrightarrow V$), mais pas faiblement ponctuellement de type fini. Ce dernier point provient de ce que l'action du monoïde multiplicatif sous-jacent à l'anneau $\mathbb{Z}[t]$ sur $\mathbb{Z}[t]^{\underset{\mathbb{Z}}{\otimes} 2}$ donnée par $\xi.(a\otimes b)=(\xi a)\otimes (\xi b)$ (où $\xi$, $a$ et $b$ sont des éléments de $\mathbb{Z}[t]$) n'est pas de type fini. En effet, sinon, il existerait un entier $N$ tel que le groupe abélien $\mathbb{Z}[t]^{\underset{\mathbb{Z}}{\otimes} 2}$ soit engendré par les éléments $(\xi t^i)\otimes(\xi t^j)$ pour $\xi\in\mathbb{Z}[t]$ et $i,j\leq N$, ou, ce qui revient au même, par les $t^{i+r}\otimes t^{j+r}$, où $i,j\leq N$, ce qui n'est pas le cas, puisque les $t^a\otimes t^b$, où $(a,b)$ parcourt $\mathbb{N}^2$, forment une {\em base} de ce groupe abélien.
\end{rem}

\subsection{Support des foncteurs et extensions de Kan}\label{sect-sk}

\begin{defi}\label{df-supp}
 Soient $\C$ une petite catégorie, $\A$ une catégorie abélienne et $F : \C\to\A$ un foncteur.

On appelle {\em support} de $F$ tout ensemble $S$ d'objets de $\C$ tel que tout sous-foncteur $G$ de $F$ tel que l'inclusion $G(t)\subset F(t)$ est une égalité pour $t\in S$ est égal à $F$.

On dit que $F$ est {\em à support fini} s'il admet un support fini.
\end{defi}

Afin d'étudier cette notion, on rappelle quelques faits classiques sur les extensions de Kan (cf. par exemple \cite{ML}, chap.~X, §\,3).

Soient $\varphi : \C\to\D$ un foncteur entre petites catégories et $\A$ une catégorie avec colimites. Le foncteur de précomposition $\varphi^* : \fct(\D,\A)\to\fct(\C,\A)$ possède un adjoint à gauche $\varphi_!$, appelé {\em extension de Kan à gauche} de $\varphi$, qui est donné explicitement sur les objets par
$$\varphi_!(F)(d)=\underset{\C[\varphi,d]}{\col}\iota[\varphi,d]^*F$$
(où $F : \C\to\A$ est un foncteur et $d$ un objet de $\D$), $\C[\varphi,d]$ désignant la catégorie des objets $c$ de $\C$ munis d'un morphisme $\varphi(c)\to d$ (les morphismes étant les morphismes de $\C$ vérifiant la condition de compatibilité évidente) et $\iota[\varphi,d] : \C[\varphi,d]\to\C$ le foncteur d'oubli. Le foncteur $\varphi_!$ peut également être caractérisé par le fait qu'il commute aux colimites et qu'on dispose d'un isomorphisme $\varphi_!(A[\C(c,-)])\simeq A[\D(\varphi(c),-)]$ naturel en les objets $c$ de $\C$ et $A$ de $\A$ respectivement.

Si le foncteur $\varphi$ est pleinement fidèle, l'unité ${\rm Id}\to\varphi^*\varphi_!$ de l'adjonction est un isomorphisme.

\begin{pr}\label{pr-suppFini}
 Soient $\C$ une petite catégorie, $S$ un ensemble d'objets de $\C$ vu comme une sous-catégorie pleine de $\C$, $\varphi : S\to\C$ le foncteur d'inclusion, $\A$ une catégorie abélienne avec colimites et $F : \C\to\A$ un foncteur. Les propriétés suivantes sont équivalentes :
 \begin{enumerate}
  \item\label{psf1} $S$ est un support de $F$ ;
  \item\label{psf2} la coünité $\varphi_!\varphi^*F\to F$ de l'adjonction est un épimorphisme ;
  \item\label{psf3} $F$ est un quotient d'une somme directe de foncteurs du type $A[\C(s,-)]$, où $A$ est un objet de $\A$ et $s$ un élément de $S$ ;
  \item\label{psf4} le morphisme canonique
 $$\underset{s\in S}{\bigoplus}\mathbb{Z}[\C(s,-)]\underset{{\rm End}_\C(s)}{\otimes}F(s)\to F$$
  est un épimorphisme.
 \end{enumerate}
\end{pr}

 \begin{proof} L'implication \ref{psf1}$\Rightarrow$\ref{psf4} provient de ce que, si $G$ désigne l'image du morphisme canonique en question, alors $G(s)=F(s)$ pour $s\in S$, car, évalué sur $s$, le morphisme canonique
 $$\mathbb{Z}[\C(s,-)]\underset{{\rm End}_\C(s)}{\otimes}F(s)\to F$$
 se réduit à l'isomorphisme canonique $\mathbb{Z}[\C(s,s)]\underset{{\rm End}_\C(s)}{\otimes}F(s)\xrightarrow{\simeq} F(s)$.
 
 L'implication \ref{psf4}$\Rightarrow$\ref{psf3} s'obtient en écrivant $F(s)$ comme quotient d'un objet libre de $\A_{{\rm End}_\C(s)}$.
 
 Comme $\varphi$ est pleinement fidèle, la coünité $\varphi_!\varphi^*(A[\C(s,-)])\to A[\C(s,-)]$ est un isomorphisme pour tous $s\in S$ et $A\in {\rm Ob}\,\A$. Le foncteur $\varphi_!\varphi^*$ commutant aux colimites, on en déduit l'implication \ref{psf3}$\Rightarrow$\ref{psf2}.
 
 Sous l'hypothèse \ref{psf2}, considérons un sous-foncteur $G$ de $F$ tel que l'inclusion $G(s)\subset F(s)$ soit une égalité pour $s\in S$. Cela signifie que le foncteur $\varphi^*(F/G)$ est nul. Le diagramme commutatif
 $$\xymatrix{\varphi_!\varphi^*(F)\ar@{->>}[r]\ar[d] & F\ar[d] \\
 \varphi_!\varphi^*(F/G)\ar[r] & F/G
 }$$
 dont les flèches horizontales sont les coünités et les flèches verticales les projections permet d'en déduire la nullité de $F/G$, d'où l'implication \ref{psf2}$\Rightarrow$\ref{psf1}.
 \end{proof}

\begin{cor}\label{pf-gal2} Soient $\C$ une petite catégorie, $\A$ une catégorie abélienne avec colimites filtrantes exactes et $F : \C\to\A$ un foncteur.
 \begin{enumerate}
\item Si $F$ est de type fini, alors $F$ est à support fini.
\item Si $F$ est à support fini et faiblement ponctuellement de type fini, alors $F$ est de type fini.
\item La classe des foncteurs $\C\to\A$ à support fini est stable par quotients et extensions.
 \end{enumerate}
\end{cor}

%\begin{proof}
% Le premier point découle de ce que les générateurs $A[\C(c,-)]$ (où $A$ est un objet de $\A$ et $c$ un objet de $\C$ --- cf. la démonstration de la proposition~\ref{pf-gal}) de la catégorie $\fct(\C,\A)$ sont à support fini ($\{c\}$ est un support de ce foncteur).
%
%Si $S$ est un support de $F$, ce foncteur est isomorphe à un quotient de la somme directe sur $s\in S$ de
%$$\mathbb{Z}[\C(s,-)]\underset{{\rm End}_\C(s)}{\otimes}F(s)\;;$$
%ce foncteur est de type fini si $F(s)$ est une représentation de type fini du monoïde ${\rm End}_\C(s)$, car on dispose, pour toute représentation $M$ de ce monoïde et tout foncteur $G : \C\to\A$, d'un isomorphisme naturel
%$$\fct(\C,\A)(\mathbb{Z}[\C(s,-)]\underset{{\rm End}_\C(s)}{\otimes}M,G)\simeq\A_{{\rm End}_\C(s)}(M,G(s))$$
%qui permet d'appliquer le corollaire~\ref{crtfe}. Cela montre le deuxième point.
%
%Pour la dernière assertion, on note que si $G$ est un quotient de $F$, alors tout support de $F$ est aussi un quotient de $G$. Si
%$$0\to F\to G\to H\to 0$$
%est une suite exacte courte de $\fct(\C,\A)$, $S$ un support de $F$ et $S'$ un support de $H$, alors $S\cup S'$ est un support de $G$. Cela achève la démonstration.
%\end{proof}

\begin{rem}\label{sup0}
 Soient $\M$ un objet de $\mi$ et $\A$ une catégorie abélienne. Du fait que $0$ est objet initial de $\M$, un foncteur $F : \M\to\A$ appartient à $\Pol_0^{{\rm fort}}(\M,\A)$ si et seulement si $\{0\}$ est un support de $F$.
\end{rem}

Au-delà du degré $0$, le lien entre propriété polynomiale forte et support ne fonctionne généralement bien que dans un sens (cf. remarque~\ref{rqsp} ci-après), cela fera l'objet du paragraphe~\ref{prp}.

\smallskip

La variante suivante de la proposition~\ref{lmatf} utilisant la notion de support fini nous sera utile dans la section~\ref{spffp}.

\begin{pr}\label{lmalph}
 Soient $\M$ un objet de $\mi$, $\A$ une catégorie de Grothendieck et $F : \M\to\A$ un foncteur ponctuellement noethérien. On suppose que, pour tout $t\in {\rm Ob}\,\M$, le foncteur $\tau_t(F)$ est à support fini. Alors $\alpha(F) : \widetilde{\M}\to\A$ est ponctuellement noethérien.
\end{pr}

\begin{proof}
Si $F$ est à support fini et ponctuellement noethérien, alors $\underset{\M}{\col}F$ est un objet noethérien de $\A$. En effet, $F$ est alors quotient d'une somme directe finie d'objets du type $A[\M(x,-)]$ avec $x\in  {\rm Ob}\,\M$ et  $A\in {\rm Ob}\,\A$ noethérien (on peut prendre $A=F(x)$), dont la colimite est $A$.

Sous les hypothèses de l'énoncé, pour tout $t\in {\rm Ob}\,\M$, $\tau_t(F)$ est à support fini et ponctuellement noethérien, donc $\alpha(F)(t)=\underset{\M}{\col}\tau_t(F)$ est noethérien.
\end{proof}

La notion suivante nous servira également dans la section~\ref{spffp}.

\begin{prdef}\label{pres-suppFini}
 Soient $\C$ une petite catégorie, $S$ un ensemble d'objets de $\C$ vu comme une sous-catégorie pleine de $\C$, $\varphi : S\to\C$ le foncteur d'inclusion, $\A$ une catégorie abélienne avec colimites et $F : \C\to\A$ un foncteur. Les propriétés suivantes sont équivalentes :
 \begin{enumerate}
  \item\label{pspf1} $S$ est un support de $F$ et, pour toute suite exacte courte $0\to N\to G\to F\to 0$, si $S$ est un support de $G$, alors $S$ est un support de $N$ ;
  \item\label{pspf2} la coünité $\varphi_!\varphi^*F\to F$ de l'adjonction est un isomorphisme ;
  \item\label{pspf3} il existe une famille $(A_s)_{s\in S}$ d'objets de $\A$ et une suite exacte courte 
$$0\to N\to\underset{s\in S}{\bigoplus}A_s[\C(s,-)]\to F\to 0$$
telle que $S$ est un support de $N$ ;
  \item\label{pspf4} il existe des familles $(A_s)_{s\in S}$ et $(B_s)_{s\in S}$ d'objets de $\A$ et une suite exacte
$$\underset{s\in S}{\bigoplus}B_s[\C(s,-)]\to\underset{s\in S}{\bigoplus}A_s[\C(s,-)]\to F\to 0.$$
 \end{enumerate}
 Si ces propriétés sont satisfaites, nous dirons que $S$ est un {\em support de présentation} de $F$.
 
 On dira que $F$ est {\em à présentation de support fini} (en abrégé $PSF$) s'il possède un support de présentation fini.
\end{prdef}

\begin{proof}
 L'équivalence entre \ref{pspf3} et \ref{pspf4} et l'implication \ref{pspf1}$\Rightarrow$\ref{pspf3} résultent de l'équivalence entre les assertions \ref{psf1} et \ref{psf3} de la proposition~\ref{pr-suppFini}.
 
 Pour l'implication \ref{pspf3}$\Rightarrow$\ref{pspf2}, si l'on pose $P= \underset{s\in S}{\bigoplus}A_s[\C(s,-)]$, on a déjà vu dans la démonstration de la proposition~\ref{pr-suppFini} que la coünité $\varphi_!\varphi^*(P)\to P$ est un isomorphisme. De plus, $\varphi_!\varphi^*(N)\to N$ est un épimorphisme, puisque $S$ est un support de $N$ (utiliser encore la proposition~\ref{pr-suppFini}), de sorte que le diagramme commutatif
 $$\xymatrix{ & \varphi_!\varphi^*(N)\ar[r]\ar@{->>}[d] & \varphi_!\varphi^*(P)\ar[r]\ar[d]^\simeq & \varphi_!\varphi^*(F)\ar[r]\ar[d] & 0\\
 0\ar[r] & N\ar[r] & P\ar[r] & F\ar[r] & 0
 }$$
 dont les flèches verticales sont les coünités et les lignes sont exactes montre la propriété~\ref{pspf2}.
 
 L'implication \ref{pspf2}$\Rightarrow$\ref{pspf1} s'établit de façon analogue.
\end{proof}

\begin{rem}\label{rq-spf}
 \begin{enumerate}
  \item On vérifie facilement qu'un foncteur $F : \C\to\A$ est $PSF$ si et seulement si le foncteur ${\rm Hom}(F,-) : \fct(\C,\A)\to\mathbf{Ab}$ commute aux colimites filtrantes {\em ponctuellement stationnaires}, c'est-à-dire aux colimites de foncteurs $\phi : \T\to\fct(\C,\A)$ où $\T$ est une petite catégorie filtrante et où, pour tout objet $c$ de $\C$, il existe un objet $t$ de $\T$ tel que l'évaluation en $c$ de l'image par $\phi$ d'une flèche $a\to b$ de $\T$ soit un isomorphisme si $\T(t,a)\neq\emptyset$. Pour le faire, on vérifie d'abord qu'un foncteur du type $A[\C(c,-)]$ commute aux colimites filtrantes ponctuellement stationnaires.
  
  Ce nouveau critère (dont nous n'aurons pas usage) permet de voir qu'un foncteur $\C\to\A$ de présentation finie est toujours $PSF$. Si $\A$ est localement de type fini, cette propriété découle des propositions~\ref{pf-gal} et~\ref{pres-suppFini}.
  \item On peut renforcer la notion introduite à la proposition~\ref{pres-suppFini} en considérant, pour tout entier $n>0$, la classe des foncteurs $F$ tels qu'il existe une suite exacte du type
  $$\underset{s\in S}{\bigoplus}A^n_s[\C(s,-)]\to\underset{s\in S}{\bigoplus}A^{n-1}_s[\C(s,-)]\to\dots\to\underset{s\in S}{\bigoplus}A^0_s[\C(s,-)]\to F\to 0.$$
  On peut facilement en donner des caractérisations analogues à celles de la proposition~\ref{pres-suppFini} ; par exemple, en termes de l'extension de Kan $\varphi_!$, cette propriété est équivalente à la conjonction de l'assertion~\ref{pspf2} de la proposition et de l'annulation de $\mathbf{L}_i(\varphi_!)(F)$ pour $0<i<n$, où les $\mathbf{L}_i(\varphi_!)$ désignent les dérivés à gauche de $\varphi_!$.
  
  Cette notion, très naturelle dans certains cadres, ne se comporte généralement pas bien, contrairement à la propriété $PSF$, relativement aux foncteurs de translation dans le contexte du présent article (sauf dans le cas de la catégorie source $\Theta$), c'est pourquoi nous ne l'utiliserons pas.
 \end{enumerate}
\end{rem}

En utilisant le critère \ref{pspf2} dans la proposition-définition~\ref{pres-suppFini} et le lemme des cinq, on obtient le résultat suivant.

\begin{cor}\label{cr-psf}
 Soient $\C$ une petite catégorie, $S$ un ensemble d'objets de $\C$, $\A$ une catégorie abélienne avec colimites et $0\to G\to F\to H\to 0$ une suite exacte de $\fct(\C,\A)$.
 \begin{enumerate}
  \item Si $S$ est un support de présentation de $G$ et $H$, alors c'est aussi un support de présentation de $F$.
  \item Supposons que $S$ est un support de $G$ et un support de présentation de $F$. Alors $S$ est un support de présentation de $H$.
 \end{enumerate}
\end{cor}

\begin{pr}\label{prelim-spf}
 Soient $\C$ une petite catégorie, $\A$ une catégorie abélienne avec colimites, $S$ et $T$ deux ensembles d'objets de $\C$ et $F : \C\to\A$ un foncteur. On suppose que :
 \begin{enumerate}
  \item $F$ est nul hors des objets isomorphes à un élément de $S$ ;
  \item $F$ est nul sur $T$ ;
  \item tout morphisme $s\to x$ de $\C$, où $s\in S$ et $F(x)=0$, se factorise par une flèche $s\to t$ où $t\in T$.
 \end{enumerate}

 Alors $S$ est un support de $F$ et $S\cup T$ est un support de présentation de $F$.
\end{pr}

\begin{proof}
 La première assertion est évidente. Pour la deuxième, notons $\varphi$ l'inclusion dans $\C$ de $S\cup T$, vu comme sous-catégorie pleine de $\C$. La coünité $\varphi_!\varphi^*(F)\to F$ est un isomorphisme sur les objets de $\C$ isomorphes à un élément de $S$, puisque $\varphi$ est pleinement fidèle, il suffit donc de vérifier que $\varphi_!\varphi^*(F)$ est nul sur les objets $x$ de $\C$ tels que $F(x)=0$. Cela provient de ce que, pour tout objet $(s,s=\varphi(s)\xrightarrow{\xi} x)$ de $(S\cup T)[\varphi,x]$ (cf. les notations du début de ce paragraphe), où $s\in S$ (si $s\in T$, $F(s)=0$), il existe $t\in T$ tel que $\xi$ se factorise par une flèche $s\to t$, et de la nullité de $F(t)$, observations qui impliquent que l'image de $F(s)$ dans $\underset{(S\cup T)[\varphi,x]}{\col}\iota[\varphi,x]^*\varphi^* F=\varphi_!\varphi^*(F)(x)$ est nulle.
\end{proof}

\subsection{Trois notions de finitude auxiliaires sur les foncteurs}

Dans la suite, nous nous concentrerons le plus souvent sur des foncteurs ponctuellement noethériens. La raison en est donnée par la définition et la proposition suivantes.

\begin{defi}
 Soient $\M$ un objet de $\mi$, $\A$ une catégorie abélienne et $F : \M\to\A$ un foncteur.  On dira que $F$ est {\em fortement noethérien} si, pour tout objet $x$ de $\M$, le foncteur $\tau_x(F)$ est noethérien.
\end{defi}

Nous ne traiterons essentiellement que de foncteurs fortement noethériens, dans la suite, pour la raison suivante. Nos démonstrations de propriétés de finitude de foncteurs polynomiaux reposent presque toutes sur des récurrences sur leur degré. En général, il s'agit donc de montrer que $F$ est noethérien en s'appuyant sur le fait que les $\delta_x(F)$, qui sont de degré strictement inférieur (si $F$ est non nul), le sont. Mais les suites exactes $F\to\tau_x(F)\to\delta_x(F)\to 0$ montrent que si $F$ et les $\delta_x(F)$ sont noethériens, alors $F$ est fortement noethérien. On pourra donc difficilement montrer qu'un foncteur est noethérien lorsqu'il n'est pas fortement noethérien. Il existe néanmoins des foncteurs polynomiaux noethériens non fortement noethériens, même dans les situations les plus simples.

\begin{ex}
 Soient $A$ un anneau noethérien à gauche. Le foncteur d'inclusion de $\mathbf{P}(A)$ dans $\mathbf{Ab}$ est polynomial de degré $1$ et noethérien, mais son image par le foncteur $\delta_A$ est le foncteur constant en $A$, qui n'est pas nécessairement un groupe abélien de type fini, donc pas forcément un foncteur de type fini.
\end{ex}

\begin{pr}\label{pr-pni}
 Si $F$ est fortement noethérien, alors $F$ est ponctuellement noethérien.
\end{pr}

\begin{proof}
 Le fait que $0$ est objet initial de $\M$ implique que, pour tout foncteur noethérien $G : \M\to\A$, $G(0)$ est un objet noethérien de $\A$ (si $V$ est un sous-objet de $G(0)$, l'évaluation en $0$ du sous-foncteur de $G$ engendré par $V$ est égale à $V$). Comme $\tau_x(F)(0)=F(x)$, on en déduit la proposition.
\end{proof}

Introduisons maintenant une définition pour une situation particulièrement simple pour traiter de finitude sur des foncteurs.

\begin{defi}\label{df-psqn}
 Soient $\C$ une petite catégorie, $\A$ une catégorie abélienne et $F : \C\to\A$ un foncteur. On dira que $F$ est {\em presque nul} si l'ensemble des classes d'isomorphisme d'objets $c$ de $\C$ tels que $F(c)\neq 0$ est fini.
\end{defi}

On notera que les foncteurs presque nuls forment une sous-catégorie épaisse de $\fct(\C,\A)$.

Les deux propositions qui suivent sont immédiates.

\begin{pr}\label{ex-psqn}
Tout foncteur presque nul est à support fini.
\end{pr}

\begin{pr}\label{pr-idiote}
 Tout foncteur presque nul et faiblement ponctuellement noethérien est noethérien.
\end{pr}

On introduit enfin une définition dont on ne fera usage qu'à la fin de cet article.

\begin{defi}\label{presqnoet}
 Soient $\C$ une petite catégorie, $\A$ une catégorie abélienne et $F : \C\to\A$ un foncteur. On dit que $F$ est {\em presque noethérien} s'il existe un sous-foncteur $G$ presque nul de $F$ tel que $F/G$ soit noethérien. 
\end{defi}

\begin{cor}\label{cor-idiot}
 Tout foncteur presque noethérien et faiblement ponctuellement noethérien est noethérien.
\end{cor}

\section{Hypothèses de finitude utiles sur les objets de $\mi$}\label{sec-sn}

On commence par donner deux définitions qu'on utilisera, dans la section~\ref{spffp}, pour des objets de $\mi$ mais qui ne font pas intervenir de structure monoïdale.

\begin{defi}\label{df-fei}
 On dit qu'une catégorie $\C$ vérifie :
\begin{enumerate}
 \item la propriété $(FE)$ si pour tout objet $x$ de $\M$, le monoïde ${\rm End}_\C(x)$ est de type fini ;
\item la propriété $(EI)$ si tout endomorphisme de $\C$ est un isomorphisme.
\end{enumerate}
\end{defi}

Nous donnons maintenant la liste des propriétés de finitude sur les catégories monoïdales symétriques qui nous seront utiles dans la suite de cet article, à des degrés divers : la propriété $(P_0)$ sera presque toujours nécessaire, tandis que la propriété $(P_4)$ n'interviendra que dans la section~\ref{spffp}.

\subsection{Définitions}

\begin{defi}\label{df-fin2}
 Soit $(\M,+,0)$ une catégorie monoïdale symétrique.
\begin{enumerate}
\item On dira qu'un ensemble $T$ d'objets de $\M$ est {\em générateur monoïdal} (resp. {\em générateur monoïdal faible}) de $(\M,+,0)$ (ou simplement, par abus, de $\M$) si tout objet de $\M$ est isomorphe à une somme finie (resp. à un facteur direct d'une somme finie), au sens de $+$, d'éléments de $T$.
\item On notera $(P_0)$ la propriété suivante de $(\M,+,0)$ : il existe un ensemble générateur monoïdal faible fini.
\item On dira que $(\M,+,0)$ est une {\em pseudo-théorie} si elle possède un ensemble générateur monoïdal réduit à un élément, i.e. s'il existe $a\in {\rm Ob}\,\M$ tel que tout objet de $\M$ soit isomorphe à $a^{+n}$ pour un $n\in\mathbb{N}$. Si l'on peut faire en sorte que cet entier soit toujours unique, nous dirons que $\M$ est une pseudo-théorie {\em régulière}.
 \item Nous dirons que $\M$ vérifie la propriété $(P_1)$ si pour tout objet $x$ de $\M$, l'ensemble des classes d'isomorphisme d'objets de $\M$ n'appartenant pas à l'image essentielle de l'endofoncteur $x+ -$ est fini.
\item Nous dirons que $\M$ vérifie la propriété $(P_2)$ si pour tout ensemble fini $E$ d'objets de $\M$, il existe un objet $x$ de $\M$ tel qu'aucun des éléments de $E$ n'appartienne à l'image essentielle de l'endofoncteur $x+ -$ de $\M$.
\item Nous dirons que $\M$ vérifie la propriété $(P_3)$ si pour tous objets $x$ et $t$ de $\M$, l'ensemble des classes d'isomorphisme d'objets $y$ de $\M$ tel qu'existe un morphisme $x\to y$ ne se factorisant pas par le morphisme canonique $x\to x+t$ est fini.
\item On dira que $\M$ vérifie la propriété $(P_4)$ si pour tous objets $t$ et $a$ de $\M$, il existe un ensemble fini d'objets $U(a,t)$ tel que, pour tout objet $x$ et tout morphisme $f : a\to t+x$ de $\M$, il existe un élément $b$ de $U(a,t)$ et des morphismes $u\in\M(a,t+b)$ et $g\in\M(b,x)$ tels que $f$ coïncide avec la composée                                                                                                                                                                                                                                                                                                                                                             $$a\xrightarrow{u}t+b\xrightarrow{t+g}t+x.$$
\end{enumerate}

\end{defi}

On voit que si $(P_0)$ est vérifiée et que $\M$ appartient à $\mn$, alors il existe un ensemble générateur monoïdal faible réduit à un élément (prendre la somme des éléments d'un ensemble générateur fini).

\begin{rem}\label{rq-diff}
 Si $T$ est un ensemble générateur monoïdal faible d'un objet $\M$ de $\mi$, alors un foncteur $F$ de $\M$ dans une catégorie abélienne est fortement polynomial de degré au plus $d\in\mathbb{N}$ si $\delta_{a_0}\dots\delta_{a_d}(F)=0$ pour tout $(a_0,\dots,a_d)\in T^{d+1}$ ; cela découle des propriétés élémentaires des foncteurs $\delta_x$ (cf. \cite{DV3}). Un critère analogue vaut pour les foncteurs faiblement polynomiaux.
\end{rem}

Le cadre catégorique formel introduit dans \cite{DV} (et légèrement revisité dans \cite{Dja-JKT} et \cite{Dja-cong}), ainsi que celui de \cite{W-stab} (qui s'inspire du précédent) pourra aider le lecteur à se convaincre que, malgré l'apparence technique de certaines des définitions ci-dessus, celles-ci ne sont pas si étranges.

\smallskip

La proposition suivante est immédiate.

\begin{pr}\label{pr-ev}
 \begin{enumerate}
  \item Toute pseudo-théorie vérifie les propriétés $(P_0)$ et $(P_1)$.
\item Une pseudo-théorie régulière vérifie également la propriété $(P_2)$.
 \end{enumerate}
\end{pr}

\subsection{Exemples fondamentaux}

\begin{ex}
 La catégorie $\Theta$ est une pseudo-théorie régulière vérifiant toutes les propriétés introduites précédemment. Toutefois, pour étudier les propriétés de finitude des foncteurs depuis la catégorie $\Theta$, nous n'aurons même pas besoin d'utiliser explicitement toutes ces propriétés.
\end{ex}

Dans des situations issues de catégories additives, les propriétés de finitude précédentes ne sont pas toujours satisfaites, et leur vérification ou réfutation peut s'avérer délicate. Nous nous contenterons de donner quelques cas favorables où l'on peut les établir sans trop de peine.

\begin{pr}\label{pfpa}
 Soit $A$ un anneau non nul.
\begin{enumerate}
 \item L'objet $\mathbf{S}(A)$ de $\mi$ est une pseudo-théorie ; elle est régulière si $A$ est commutatif ou artinien.
\item La catégorie $\mathbf{S}(A)$ vérifie $(P_3)$ si $A$ a un rang stable (au sens de Bass \cite{B}) fini, par exemple si $A$ est une algèbre commutative de type fini sur un corps, un anneau principal, un anneau local ou un anneau fini.
\item La catégorie $\mathbf{S}(A)$ vérifie $(P_4)$ si $A$ est un corps, un anneau principal ou un anneau fini.
\end{enumerate}
\end{pr}

\begin{proof}
 Il est évident que $\mathbf{S}(A)$ est une pseudo-théorie ; elle est régulière si et seulement si tout isomorphisme $A^{ n}\simeq A^{m}$ de $A$-modules à droite, où $n, m\in\mathbb{N}$, implique $n=m$. Lorsque $A$ est commutatif, cette propriété s'obtient de façon rapide et classique par l'utilisation du déterminant, pour $A$ artinien c'est une conséquence de la propriété de Krull-Schmidt.

Dans la catégorie $\mathbf{S}(A)$, la condition $(P_3)$ peut se reformuler comme suit : {\em pour tous $n,m\in\mathbb{N}$, il existe $r$ tel que, pour $t\geq r$, la condition $A^t\simeq A^n\oplus M$ pour un $A$-module à droite $M$ implique l'existence d'un isomorphisme $M\simeq A^m\oplus N$ pour un certain $A$-module à droite $N$}. Il suffit bien sûr de le montrer pour $n=m=1$, en raisonnant par récurrence. Dire que $A$ est de rang stable fini implique classiquement et facilement cette propriété.

Dans la catégorie $\mathbf{S}(A)$, la condition $(P_4)$ est équivalente à la suivante : {\em il existe un entier $r$ tel que pour tout entier $s$, toute application $A$-linéaire (à droite) $f : A\to A^s$ se factorise en
$$A\to A^i\xrightarrow{g} A^s$$
pour un $i\leq r$ et un monomorphisme scindable $g$.}

On montre seulement ici que cette propriété implique l'hypothèse $(P_4)$ ; la réciproque, que le lecteur pourra traiter facilement en exercice, ne nous servira pas.
Si la condition précédente est vérifiée, considérons un morphisme $\Phi : A\to A^n\oplus A^s$ de $\mathbf{S}(A)$, c'est-à-dire la donnée d'applications linéaires $A\to A^n\oplus A^s$ de coordonnées $(x,y)$ et $A^n\oplus A^s\to A$ de coordonnées $(l,L)$ avec $l(x)+L(y)=1$. On factorise $y$ comme précédemment : on trouve $z : A\to A^i$ avec $i\leq r$, $g : A^i\to A^n$ et $h : A^n\to A^i$ tels que $gz=y$ et $hg=1$ ($g$ et $h$ forment donc un élément $\Psi$ de $\mathbf{S}(A)(A^i,A^n)$). On obtient alors une factorisation dans $\mathbf{S}(A)$
$$\xymatrix{A\ar[rr]^\Phi\ar[rd]_-\Xi & & A^n\oplus A^s \\
& A^n\oplus A^i\ar[ru]_{A^n\oplus\Psi} &
}$$
où $\Xi$ est le morphisme donné par $(x,z) : A\to A^n\oplus A^i$ et $A^n\oplus A^i\to A$ de coordonnées $l$ et $Lg$. Cela montre que $U(A,A^n)=\{A^i\,|\,i\leq r\}$ est un choix convenable (avec les notations de la définition~\ref{df-fin2} pour $(P_4)$), d'où l'on déduit facilement que $(P_4)$ est vérifiée (raisonner par récurrence sur $m$ pour trouver un $U(A^m,A^n)$ convenable).

Lorsque $A$ est un corps ou un anneau principal, il est clair que la condition précédente est satisfaite pour $r=1$. Si $A$ est un anneau fini, on peut choisir évidemment le cardinal de l'anneau pour $r$. Cela termine la démonstration.
\end{proof}

\begin{rem}
\begin{enumerate}
\item Il n'est pas difficile de voir que la condition $(P_4)$ pour la catégorie $\mathbf{S}(A)$ entraîne que $A$ a un rang stable de Bass fini (mais la réciproque semble erronée). D'un point de vue culturel, il serait intéressant de clarifier les liens entre les conditions de Bass et nos hypothèses, qui pourraient permettre de revisiter à l'aune de méthodes fonctorielles des résultats connus sur la stabilité en $K$-théorie algébrique.
\item On pourrait généraliser la proposition précédente à des catégories d'espaces hermitiens hyperboliques sur un anneau à involution aux bonnes propriétés.
\end{enumerate}
\end{rem}

\subsection{Utilisation pour les propriétés de finitude des foncteurs stablement nuls}

Pour espérer disposer de bonnes propriétés de finitude des foncteurs faiblement polynomiaux, il est indispensable de commencer par se pencher sur les foncteurs stablement nuls : certaines des hypothèses introduites précédemment permettront d'en garantir un comportement agréable. Pour cela, on relie la notion de foncteur stablement nul à celle de foncteur presque nul.

Les deux propositions suivantes sont immédiates. 

\begin{pr}\label{lm-pt3}
 Soient $\M$ un objet de $\mi$, $\A$ une catégorie abélienne et $F : \M\to\A$ un foncteur.
\begin{enumerate}
 \item Si $F$ est presque nul et que $\M$ vérifie la propriété $(P_2)$, alors il existe un objet $t$ de $\M$ tel que $\tau_t(F)=0$.
\item S'il existe un objet $t$ de $\M$ tel que $\tau_t(F)=0$ et que $\M$ vérifie la propriété $(P_1)$, alors $F$ est presque nul.
\end{enumerate}
\end{pr}

\begin{pr}\label{prp2}
 Soient $\M$ un objet de $\mi$ vérifiant $(P_1)$, $\A$ une catégorie abélienne et $F : \M\to\A$ un foncteur. Supposons que $F$ est faiblement ponctuellement noethérien et qu'il existe un objet $t$ de $\M$ tel que $\tau_t(F)$ soit noethérien. Alors $F$ est noethérien.
\end{pr}

Nous utiliserons la propriété $(P_3)$ par l'intermédiaire de la proposition sui\-vante.

\begin{pr}\label{psn}
 Soient $\M$ un objet de $\mi$ vérifiant la propriété $(P_3)$, $\A$ une catégorie abélienne et $F : \M\to\A$ un foncteur stablement nul et de type fini. Alors $F$ est presque nul.
\end{pr}

\begin{proof}
 Le foncteur $F$, stablement nul, est réunion filtrante des sous-foncteurs $\kappa_x(F)$ (pour $x\in {\rm Ob}\,\M$) et de type fini, donc il existe $x$ tel que $\kappa_x(F)=F$, c'est-à-dire tel que $i_x(F) : F\to\tau_x(F)$ soit nul. Si $S$ est un support de $F$, cela entraîne que $F(y)=0$ pour tout objet $y$ tel que toute flèche de source $t$ dans $S$ et de but $y$ se factorise par $t\to x+t$. Comme $F$ possède un support fini (corollaire~\ref{pf-gal2}) et que $\M$ vérifie $(P_3)$, on en déduit la finitude souhaitée de l'ensemble des classes d'isomorphisme d'objets $y$ de $\M$ tels que $F(y)$ est non nul.
\end{proof}

En utilisant la proposition~\ref{pr-idiote} et le corollaire~\ref{pf-gal2}, on en déduit :

\begin{cor}\label{cor-snn}
 Soient $\M$ un objet de $\mi$ vérifiant $(P_3)$, $\A$ une catégorie abélienne et $F : \M\to\A$ un foncteur stablement nul, à support fini et faiblement ponctuellement noethérien. Alors $F$ est noethérien.
\end{cor}

Sans la propriété $(P_3)$ (ou de légers affaiblissements de cette condition), il y a généralement beaucoup trop de foncteurs stablement nuls pour qu'on puisse aborder leurs propriétés de finitude par des méthodes spécifiques.

\begin{ex}
 Voyons le monoïde commutatif ordonné usuel $\mathbb{N}$ comme un objet de $\mi$ et notons $\mathbb{Z}_0 : \mathbb{N}\to\mathbf{Ab}$ le foncteur égal à $\mathbb{Z}$ en $0$ et nul ailleurs. Si $\M$ est un objet de $\mi$, le foncteur $\fct(\M,\A)\to\fct(\M\times\mathbb{N},\A)$ (où $\A$ est une catégorie abélienne) produit tensoriel extérieur par $\mathbb{Z}_0$ est exact et pleinement fidèle ; il est à valeurs dans les foncteurs stablement nuls. Ainsi, tout foncteur $\M\to\A$ peut se voir comme un foncteur stablement nul, en changeant la catégorie de départ. On notera de surcroît que, si $\M$ vérifie $(P_0)$ ou $(P_2)$, alors il en est de même pour $\M\times\mathbb{N}$ (qui ne vérifie toutefois presque jamais $(P_1)$).
\end{ex}

La proposition suivante et son corollaire constituent un échauffement pour les considérations du §\,\ref{sect-psf}.

\begin{pr}\label{pr-psfs}
 Soient $\M$ un objet de $\mi$ vérifiant $(P_0)$ et $(P_3)$ et $\A$ une catégorie abélienne avec colimites. Tout foncteur presque nul $F : \M\to\A$ est à présentation de support fini.
\end{pr}

\begin{proof}
 Soit $S$ un ensemble fini d'objets de $\M$ tel que $F$ soit nul hors des objets isomorphes à un élément de $S$. On peut supposer qu'il existe une fonction $f : S\to {\rm Ob}\,\M$ telle que, pour tout $s\in S$, $s+f(s)$ ne soit isomorphe à aucun élément de $S$. En effet, si ce n'est pas le cas, on déduit facilement de la propriété $(P_0)$ qu'il existe un ensemble fini $E$ d'objets de $\M$ tel que tout objet de $\M$ soit facteur direct d'un élément de $E$, auquel cas $E$ est support de présentation de {\em tout} foncteur $\M\to\A$. Cette fonction $f$ étant fixée, par $(P_3)$, on dispose, pour tout $s\in S$, d'un ensemble fini $R(s)$ d'objets de $\M$ tel que, si $u$ est un objet de $\M$ tel qu'existe une flèche $s\to u$ ne se factorisant pas par le morphisme canonique $s\to s+f(s)$, alors $u$ est isomorphe à un élément de $R(s)$. La proposition~\ref{prelim-spf} montre alors que l'ensemble fini
 $$S\cup\{\,s+f(s)\,|\,s\in S\}\cup\underset{s\in S}{\bigcup} R(s)$$
 est un support de présentation de $F$.
\end{proof}

En combinant ce résultat à la proposition~\ref{psn}, on obtient l'énoncé suivant.

\begin{cor}\label{cor-psf}
 Soient $\M$ un objet de $\mi$ vérifiant $(P_0)$ et $(P_3)$ et $\A$ une catégorie abélienne avec colimites. Tout foncteur stablement nul et de type fini $\M\to\A$ est à présentation de support fini.
\end{cor}

\section{Support des foncteurs fortement polynomiaux}

Les propriétés de finitude des foncteurs polynomiaux depuis une catégorie dans $\mn$ sont beaucoup plus faciles à étudier que pour une catégorie source générale de $\mi$. Cela est dû à une propriété générale de support (proposition~\ref{pr-P11} ci-après) des foncteurs fortement polynomiaux (depuis un objet de $\mi$ n'appartenant pas forcément à $\mn$), dont les conséquences sont aisées à tirer quand la sous-catégorie des foncteurs fortement polynomiaux est stable par sous-objets, ce qui n'est plus le cas en général quand on sort du cadre de $\mn$.

\subsection{Résultats principaux}\label{prp}

\begin{pr}\label{pr-P11}
 Soient $\M$ un objet de $\mi$ vérifiant la propriété $(P_0)$ et $\A$ une catégorie abélienne. Tout foncteur fortement polynomial $\M\to\A$ est à support fini. Plus précisément, soit $T$ un ensemble générateur monoïdal faible fini de $\M$ et $d\in\mathbb{N}$. L'ensemble $S_d:=\{a_1+\dots+a_i\,|\,i\leq d,\,a_i\in T\}$ (qui est égal à $\{0\}$ pour $d=0$) est un support pour chaque foncteur de $\Pol_d^{{\rm fort}}(\M,\A)$.
\end{pr}

\begin{proof}
On établit la propriété par récurrence pour $d$. Pour $d=0$, elle découle de ce que $0$ est objet initial de $\M$ (cf. remarque~\ref{sup0}).

Supposons-la vraie pour les foncteurs de degré strictement inférieur à $d>0$ et considérons un foncteur $F$ fortement polynomial de degré fort $d$. Soit $G$ un sous-foncteur de $F$ tel que $G(t)=F(t)$ pour $t\in S_d$. Pour $x\in T$, notons $H_x$ l'image du morphisme $\delta_x(G)\to\delta_x(F)$ induit par l'inclusion. L'inclusion $H_x(t)\subset\delta_x(F)(t)$ est une égalité pour $t\in S_{d-1}$, de sorte que $H_x=\delta_x(F)$ par l'hypothèse de récurrence. Autrement dit, $\delta_x(F/G)=0$ pour tout $x\in T$, de sorte que $F/G$ est polynomial de degré fort au plus $0$, puisque l'ensemble $T$ est générateur monoïdal faible de $\M$ (utiliser la proposition~1.17 de \cite{DV3}). Mais comme $0$ appartient à $S_d$, $F/G$ est nul en $0$, donc $F/G$ est nul, ce qui établit la proposition.
\end{proof}

\begin{rem}\label{rq-moninf}
Sans hypothèse de finitude sur $\M$, la conclusion de la proposition~\ref{pr-P11} peut être en défaut, en voici un exemple.

 Soient $E$ un ensemble quelconque et $M$ le monoïde commutatif obtenu en adjoignant à $E$ deux éléments notés $0$ et $\infty$, la loi de composition étant caractérisée par le fait que $0$ est son unité et que $x+y=\infty$ dès que ni $x$ ni $y$ ne sont égaux à $0$. On munit $M$ d'un ordre partiel compatible à sa structure de monoïde en imposant que $0$ (resp. $\infty$) en soit le plus petit (resp. le plus grand) élément et que les éléments de $E$ soient deux à deux incomparables. Cette relation d'ordre permet de voir $M$ comme une petite catégorie monoïdale symétrique, qui appartient à $\mi$. Elle vérifie même la propriété de finitude suivante (mais pas $(P_0)$) :  il existe un objet, en l'occurrence $\infty$, tel que pour tout objet $x$ de $M$ existe un objet $y$ tel que $x+y=\infty$ (cette propriété impliquerait $(P_0)$ pour une catégorie de $\mn$).

Soient $\A$ une catégorie abélienne et $A$ un objet non nul de $A$. Alors le foncteur de $\fct(M,\A)$ constant en $A$ est fortement polynomial de degré fort $0$, il est à support fini ($\{0\}$ en étant un support), et ponctuellement noethérien (donc en particulier de type fini par le corollaire~\ref{pf-gal2}) si $A$ est noethérien, mais il n'est pas noethérien si $E$ est un ensemble infini. En effet, le sous-foncteur $A_{>0}$ de $A$ nul en $0$ et égal à $A$ ailleurs n'est pas de type fini (tout support de ce foncteur contient $E$). Pourtant, ce foncteur $A_{>0}$ est fortement polynomial de degré fort~$1$.
\end{rem}

\begin{rem}\label{rq-fortpsf}
 Même lorsque $\M$ vérifie toutes les propriétés de finitude qu'on a introduites et que $\A$ vérifie également toutes les propriétés raisonnables de finitude qu'on peut attendre d'une catégorie abélienne, un foncteur fortement polynomial $F : \M\to\A$ n'est pas forcément à {\em présentation} de support fini, ce dès le degré fort $0$. En voici un exemple dans $\fct(\Theta,\mathbf{Ab})$. Pour $n\in\mathbb{N}$, notons $\mathbb{Z}_{<n}$ (resp. $\mathbb{Z}_{\geq n}$) le foncteur associant $\mathbb{Z}$ à $\mathbf{i}$ pour $i<n$ (resp. $i\geq n$) et égal à l'identité sur les morphismes $\mathbf{i}\to\mathbf{j}$ pour $j<n$ (resp. $i\geq n$) et à $0$ sur les autres morphismes. On dispose ainsi d'une suite exacte $0\to\mathbb{Z}_{\geq n}\to\mathbb{Z}\to\mathbb{Z}_{<n}\to 0$. Prenons la somme directe sur $n\in\mathbb{N}$ de ces suites exactes : on obtient une suite exacte
 $$0\to\underset{n\in\mathbb{N}}{\bigoplus}\mathbb{Z}_{\geq n}\to\mathbb{Z}^{\oplus\mathbb{N}}\to\underset{n\in\mathbb{N}}{\bigoplus}\mathbb{Z}_{<n}\to 0$$
 où $\mathbb{Z}^{\oplus\mathbb{N}}$ et $\underset{n\in\mathbb{N}}{\bigoplus}\mathbb{Z}_{<n}$ sont fortement polynomiaux de degré nul, donc en particulier à support fini, alors que $\underset{n\in\mathbb{N}}{\bigoplus}\mathbb{Z}_{\geq n}$ n'est pas à support fini (car tout support de $\mathbb{Z}_{\geq n}$ contient $\mathbf{n}$). Par conséquent, $\underset{n\in\mathbb{N}}{\bigoplus}\mathbb{Z}_{<n}$ est un foncteur fortement polynomial ne possédant pas la propriété $PSF$.
 
 On peut néanmoins établir la propriété $PSF$ pour des classes assez vastes de foncteurs polynomiaux, comme on le verra au paragraphe~\ref{sect-psf}.
\end{rem}

Dans le cas de la catégorie source $\Theta$, il est remarquable qu'on dispose d'une réciproque au fait qu'un foncteur fortement polynomial est à support fini :

\begin{pr}\label{theta-supp}
 Soient $\A$ une catégorie abélienne et $F : \Theta\to\A$ un foncteur. Alors $F$ est fortement polynomial si et seulement s'il est à support fini. Plus précisément, $F$ est fortement polynomial de degré fort au plus $d$ si et seulement si $\{\mathbf{0},\dots,\mathbf{d}\}$ est un support de $F$.
\end{pr}

\begin{proof}
 Un calcul immédiat fournit un isomorphisme naturel
$$\delta_{\mathbf{1}}(A[\Theta(\mathbf{d},-)])\simeq A[\Theta(\mathbf{d-1},-)]^{\oplus d}$$
(où $A$ est un objet de $\A$ et $d$ un entier naturel), d'où l'on tire par récurrence que $A[\Theta(\mathbf{d},-)]$ est fortement polynomial de degré au plus $d$. Il s'en suit que tout foncteur dont $\{\mathbf{0},\dots,\mathbf{d}\}$ est un support est fortement polynomial de degré au plus $d$. La réciproque provient de la proposition~\ref{pr-P11}.
\end{proof}

\begin{rem}\label{rqsp}
 En général, si $\M$ est un objet de $\mi$ (ou même de $\mn$), il est tout à fait exceptionnel que les foncteurs représentables $U[\M(x,-)]$ (où $U$ est un objet non nul de $\A$ et $x$ un objet de $\M$) soient polynomiaux (même faiblement). Par exemple, si $A$ est un anneau non nul et $n>0$ un entier, ce n'est jamais le cas pour $\M=\mathbf{S}(A)$ ou $\M=\mathbf{P}(A)$ si $x=A^{\oplus n}$. 
\end{rem}

\subsection{Application : propriétés de finitude des foncteurs polynomiaux depuis un objet de $\mn$}

\begin{pr}\label{pnoeth1}
 Soient $\M$ une catégorie de $\mn$ possédant la propriété $(P_0)$ et $\A$ une catégorie abélienne.

Tout foncteur polynomial $\M\to\A$ faiblement ponctuellement de type fini (resp. faiblement ponctuellement noethérien) est de type fini (resp. noethérien).
\end{pr}

\begin{proof}
 Cela résulte de la proposition~\ref{pr-P11}, du corollaire~\ref{pf-gal2} et de ce que tout sous-foncteur d'un foncteur polynomial de source dans $\mn$ est encore polynomial. 
\end{proof}

\begin{rem}\label{rqptf}
 Si l'on n'impose aucune condition de finitude à la catégorie monoïdale $\M$, les conclusions de la proposition précédente peuvent tomber en défaut (nous ne prétendons pas que la condition $(P_0)$ soit optimale, mais elle correspond très bien aux applications usuelles). L'exemple~\ref{ex-notad} ci-après en donne une illustration.
%En voici un exemple.
%Soient $(\B_i)_{i\in E}$ une famille infinie de petites catégories additives non nulles (on note $X_i$ un objet non nul de $\B_i$) et $\C:=\underset{i\in E}{\bigoplus}\B_i$ ; on munit cette petite catégorie additive de la structure monoïdale symétrique (dont l'unité est objet nul) donnée par la somme directe. Le foncteur
%$$\underset{i\in E}{\bigoplus}\C(X_i,-) : \C\to\mathbf{Ab}$$
%est additif (donc polynomial de degré $1$) --- il est aussi ponctuellement noethérien si les $\B_i$ et $X_i$ sont bien choisis --- mais n'est pas à support fini.
\end{rem}

%\subsection{Une propriété de finitude des foncteurs polynomiaux depuis une catégorie additive}

À titre d'application de la proposition~\ref{pnoeth1}, nous donnons une condition suf\-fisante simple pour que les foncteurs polynomiaux d'une catégorie {\em additive} vers une catégorie abélienne localement noethérienne forment une catégorie localement noethérienne.

\begin{pr}\label{pr-adln}
 Soient $\C$ une petite catégorie additive vérifiant $(P_0)$ et telle que, pour tous objets $a$ et $b$ de $\C$, le groupe abélien $\C(a,b)$ soit de type fini, et $\A$ une catégorie abélienne localement noethérienne. Alors la catégorie des foncteurs polynomiaux de $\C$ dans $\A$ est localement noethérienne. De plus, les objets noethériens de cette catégorie sont des foncteurs ponctuellement noethériens. 
\end{pr}

\begin{proof}
 Notons $\mathbf{Ab}^{{\rm tf}}$ la sous-catégorie pleine de $\mathbf{Ab}$ des groupes abéliens de type fini, et $Q_n : \mathbf{Ab}^{{\rm tf}}\to\mathbf{Ab}$, pour $n\in\mathbb{N}$, le plus grand quotient du foncteur $\mathbb{Z}[-]$ appartenant à $\Pol_n(\mathbf{Ab}^{{\rm tf}},\mathbf{Ab})$. Une description explicite (comme conoyau d'une certaine transformation naturelle) de ce foncteur peut se trouver dans le §\,2.3 de \cite{HPV}. La catégorie $\Pol_n(\C,\A)$ est engendrée par les foncteurs $A\otimes Q_n\circ\C(t,-)$, où $A$ est un objet noethérien de $\A$ et $t$ un objet de $\C$. La conclusion résulte donc de la proposition~\ref{pnoeth1} et du lemme ci-après.
\end{proof}

\begin{lm}
 Pour tout $n\in\mathbb{N}$, le foncteur $Q_n$ est à valeurs dans $\mathbf{Ab}^{{\rm tf}}$.
\end{lm}

\begin{proof}
Par la proposition~2.17 de \cite{HPV}, pour tout groupe abélien de type fini $M$, $Q_n(M)$ est le quotient de $\mathbb{Z}[M]$ par le sous-groupe engendré par les éléments du type
$$\sum_{I\subset\{0,\dots,n\}}(-1)^{{\rm Card}(I)}[x_I]$$
où $x_0,\dots,x_n$ sont des éléments de $M$ et l'on note $x_I:=\sum_{i\in I} x_i$.

Soit $E$ une partie génératrice finie de $M$. On vérifie que l'image dans $Q_n(M)$ des éléments de $\mathbb{Z}[M]$ du type
$$\Big[\sum_{t\in E}\varepsilon(t) t\Big]$$
où $\varepsilon$ parcourt l'ensemble des fonctions de $E$ vers les entiers de valeur absolue au plus $n$ est une partie génératrice de $Q_n(M)$.
\end{proof}

Sans la propriété $(P_0)$, la conclusion de la proposition~\ref{pr-adln} peut tomber en défaut, comme l'illustre l'exemple suivant.

\begin{ex}\label{ex-notad}
 La catégorie $\mathbf{Ab}^{tf}$ est additive, essentiellement petite et tous ses groupes abéliens de morphismes sont de type fini. Pourtant, la catégorie des foncteurs additifs  $\mathbf{Ab}^{tf}\to\mathbf{Ab}$ n'est pas localement noethérienne. En effet, le foncteur d'inclusion est de type fini (et par ailleurs ponctuellement noethérien), puisque représentable par $\mathbb{Z}$, mais contient le foncteur associant à un groupe abélien de type fini son sous-groupe de torsion, qui est la somme directe infinie non triviale des foncteurs de torsion $p$-primaire, $p$ décrivant l'ensemble des nombres premiers, donc n'est pas de type fini.
\end{ex}

\begin{rem}
 Si les foncteurs polynomiaux d'une petite catégorie additive $\C$ vers les groupes abéliens forment une catégorie localement noethérienne, alors ils possèdent une autre propriété de finitude : ceux d'entre eux qui sont de type fini admettent, dans la catégorie de {\em tous} les foncteurs $\C\to\mathbf{Ab}$, une résolution projective de type fini (lorsque $\C=\mathbf{P}(A)$ pour un certain anneau $A$, cette propriété équivaut, pour les foncteurs additifs, à la commutation de l'homologie de Mac Lane de $A$ aux colimites filtrantes). Ce résultat classique est lié à la construction cubique de Mac Lane ; une esquisse de démonstration en est donnée dans les notes \cite{Dja-notes} (proposition~7, page~8).
\end{rem}

\section{Propriétés de finitude des foncteurs fortement polynomiaux}\label{sec-eta}

\subsection{Propriétés de finitude des foncteurs $\eta(F)$}\label{sec-etav}

\begin{pr}\label{noeth-eta}
 Soient $\M$ un objet de $\mi$ vérifiant les propriétés $(P_0)$ et $(P_1)$, $\A$ une catégorie de Grothendieck et $F : \widetilde{\M}\to\A$ un foncteur polynomial ponctuellement noethérien. Alors $\eta(F)$ est un objet noethérien de $\fct(\M,\A)$.
\end{pr}

(La proposition~\ref{polfor-noeth2} ci-après montrera que l'hypothèse $(P_1)$ est en fait superflue.)

\begin{proof}
 On raisonne par récurrence sur le degré $d$ de $F$ : on suppose le résultat établi pour les foncteurs polynomiaux de degré strictement inférieur à $d$. Si $x$ est un objet de $\M$ et $A$ un sous-foncteur de $\eta(F)$, notons $A[x]$ l'image du morphisme
$$\tau_x(A)\hookrightarrow\tau_x\eta(F)\simeq\eta\tau_x(F)\twoheadrightarrow\eta(F)$$
composé du monomorphisme induit par l'inclusion et de l'épimorphisme canonique déduit de ce que $0$ est objet final dans $\widetilde{\M}$. Tout morphisme $x\to y$ de $\M$ induit un diagramme commutatif
$$\xymatrix{\tau_x(A)\ar@{^{(}->}[r]\ar[d] & \tau_x\eta(F)\ar[r]^\simeq\ar[d] & \eta\tau_x(F)\ar@{->>}[r]\ar[d] & \eta(F)\ar@{=}[d] \\
\tau_y(A)\ar@{^{(}->}[r] & \tau_y\eta(F)\ar[r]^\simeq & \eta\tau_y(F)\ar@{->>}[r] & \eta(F)
}$$
de $\fct(\M,\A)$ qui montre l'inclusion $A[x]\subset A[y]$ pour $\M(x,y)\neq\emptyset$. En particulier, $A=A[0]\subset A[x]$, et la famille de sous-objets $(A[x])_{x\in {\rm Ob}\,\M}$ de $\eta(F)$ est filtrante à droite. Notons $\tilde{A}$ sa réunion. La suite de la démonstration repose sur le lemme suivant.
\end{proof}

\begin{lm}\label{lm-noeta}
 \begin{enumerate}
  \item Le sous-objet $\tilde{A}$ de $\eta(F)$ s'identifie à l'image par le foncteur $\eta$ de l'image $\hat{A}$ du morphisme $\alpha(A)\to F$ de $\fct(\widetilde{\M},\A)$ adjoint à l'inclusion $A\hookrightarrow\eta(F)$.
\item Il existe un objet $x$ de $\M$ (dépendant de $A$) tel que $\tilde{A}=A[x]$.
\item Pour tout objet $x$ de $\M$, on dispose d'une suite exacte
$$0\to\tau_x(A)\cap\eta\delta_x(F)\to\tau_x(A)\to A[x]\to 0$$
(on plonge $\delta_x(F)$ dans $\tau_x(F)$ comme le noyau de la projection $\tau_x(F)\twoheadrightarrow F$) naturelle en le sous-foncteur $A$ de $\eta(F)$.
 \end{enumerate}

\end{lm}

\begin{proof}
Le premier point provient de la formule~(\ref{eqalph}) décrivant l'adjoint à gauche $\alpha$ de $\eta$. On en déduit que $\tilde{A}$ est fortement polynomial (de degré au plus $d$), ce qui entraîne que ce foncteur est à support fini par la proposition~\ref{pr-P11}. Comme ce foncteur est également ponctuellement noethérien, comme $F$, il est de type fini (corollaire~\ref{pf-gal2}). Du fait qu'il est réunion filtrante des sous-foncteurs $A[x]$, cela implique la deuxième propriété.

 La dernière assertion est évidente.
\end{proof}

\begin{proof}[Fin de la démonstration de la proposition~\ref{noeth-eta}]
 Soit $(A_n)_{n\in\mathbb{N}}$ une suite croissante de sous-objets de $\eta(F)$. La suite $(\hat{A}_n)$ de sous-foncteurs de $F$ stationne --- disons $\hat{A}_n=\hat{A}_N$ pour $n\geq N$. En effet, la proposition~\ref{pnoeth1} montre que $F$ est noethérien ($\widetilde{\M}$ vérifie $(P_0)$ comme $\M$). Le lemme~\ref{lm-noeta} procure un objet $x$ de $\M$ tel que $\eta(\hat{A}_N)=\tilde{A}_N=A_N[x]$. Pour $n\geq N$, l'inclusion $A_N[x]\subset A_n[x]$ est une égalité, puisque sa composée avec l'inclusion $A_n[x]\subset \tilde{A}_n=\eta(\hat{A}_n)$ coïncide avec la composée des inclusions $A_N[x]\subset\tilde{A}_N=\eta(\hat{A}_N)\subset\eta(\hat{A}_n)$ qui sont des égalités par choix de $N$ et $x$.

Par ailleurs, la suite $(\tau_x(A_n)\cap\eta\delta_x(F))_n$ de sous-foncteurs de $\eta\delta_x(F)$ stationne, par l'hypothèse de récurrence. Le dernier point du lemme précédent montre donc que la suite $(\tau_x(A_n))_n$ de sous-foncteurs de $\tau_x\eta(F)$ stationne. Grâce à la propriété $(P_1)$, qui garantit que tous les objets de $\M$, sauf un nombre fini à isomorphisme près, appartiennent à l'image essentielle de l'endofoncteur $x+ -$ et au fait que $F$ est ponctuellement noethérien, on en déduit que la suite $(A_n)_n$ elle-même stationne. Cela termine la démonstration.
\end{proof}

\begin{rem}\label{amelior}
 En reprenant la même démonstration en tenant compte de l'action naturelle du monoïde ${\rm End}_\M(x)$ sur $\tau_x(A)$ et $\delta_x(A)$ pour tous $x\in {\rm Ob}\,\M$ et $A : \M\to\A$, on voit qu'on peut affaiblir l'hypothèse que $F$ est ponctuellement noethérien en l'hypothèse suivante : pour tous objets $a_1,\dots,a_r$ de $\M$, $F(a_1+\dots+a_r)$ est un objet noethérien dans la catégorie des représentations (dans $\A$) du monoïde ${\rm End}_\M(a_1)\times\dots\times {\rm End}_\M(a_r)$ (on peut se contenter de cette propriété pour $r$ égal au degré polynomial de $F$). Nous n'insistons toutefois pas sur cette légère amélioration (qui se transmet à de nombreux résultats ultérieurs de cet article) car nous n'en connaissons pas d'application.
\end{rem}

\subsection{Application : cas de la catégorie source $\Theta$ (cf. \cite{CEFN})}

La proposition suivante (dans le cas où $\A$ est une catégorie de modules) constitue l'un des résultats principaux ({\em Theorem~1.1}) de l'article \cite{CEFN}.

\begin{pr}[Church-Ellenberg-Farb-Nagpal]\label{FInoeth}
 Soit $\A$ une catégorie de Grothendieck localement noethérienne. La catégorie $\fct(\Theta,\A)$ est localement noethérienne.
\end{pr}

\begin{proof}
 Pour tout objet $A$ de $\A$ et tout entier $n\in\mathbb{N}$, le foncteur $A[\widetilde{\Theta}(\mathbf{n},-)]$ de $\fct(\widetilde{\Theta},\A)$ est polynomial de degré $n$ --- on vérifie en effet que $\delta_\mathbf{1}\big(A[\widetilde{\Theta}(\mathbf{n},-)]\big)\simeq A[\widetilde{\Theta}(\mathbf{n-1},-)]^{\oplus n}$ (cf. la démonstration de la proposition~\ref{theta-supp}). Comme la catégorie $\Theta$ vérifie les propriétés $(P_0)$ et $(P_1)$, la proposition~\ref{noeth-eta} montre que $\eta\big(A[\widetilde{\Theta}(\mathbf{n},-)]\big)$, donc a fortiori $A[\Theta(\mathbf{n},-)]$ qui en est un sous-foncteur, est un objet noethérien de $\fct(\Theta,\A)$ si $A$ est un objet noethérien de $\A$. La catégorie $\fct(\Theta,\A)$ est engendrée par les $A[\widetilde{\Theta}(\mathbf{n},-)]$, où $A$ appartient à un ensemble de générateur noethériens de $\A$ et $n$ à $\mathbb{N}$, d'où le résultat.
\end{proof}

\begin{rem}
 La méthode que nous avons suivie pour établir la proposition~\ref{noeth-eta} généralise celle de l'article \cite{CEFN} ; elle suit également un principe analogue à celui utilisé pour démontrer un résultat de finitude dans des catégories de foncteurs (nettement plus difficile) appelé {\em théorème de simplicité généralisé} dans  \cite{Dja} (théorème~16.2.7).
\end{rem}

La situation de la catégorie $\Theta$ est particulièrement favorable : le fait que des foncteurs du type $A[\C(c,-)]$, où $\C$ est un objet de $\mi$ (ou de $\mn$) et $c$ un objet de $\C$, soient polynomiaux est tout à fait exceptionnel (cf. remarque~\ref{rqsp}). En particulier, les propriétés de finitude des foncteurs polynomiaux ne suffisent pas à montrer des propriétés de finitude pour tous les foncteurs de $\fct(\C,\A)$ (qui sont notoirement difficiles en général --- cf. \cite{Dja}, \cite{SSn} ou \cite{PSam}) ; de plus, même pour les foncteurs polynomiaux, la proposition~\ref{noeth-eta} ne suffit pas à elle seule à obtenir facilement tous les résultats de finitude qu'on peut espérer. Elle constitue toutefois le point de départ de tous nos résultats de noethérianité ultérieurs.

\subsection{Résultat principal}

En utilisant la fonctorialité en $\M$ des constructions, on voit que la catégorie $\Theta$ joue un rôle fondamental pour l'étude des foncteurs fortement polynomiaux.

\begin{lm}\label{lmevf}
 Soit $\M$ un objet de $\mi$ vérifiant $(P_0)$. Il existe un entier $n\in\mathbb{N}$ et un morphisme $\Phi : \Theta^n\to\M$ de $\mi$ qui est faiblement essentiellement surjectif au sens où tout objet de $\M$ est facteur direct d'un objet de l'image de $\Phi$.
\end{lm}

\begin{proof}
 Soit $E=\{a_1,\dots,a_n\}$ un ensemble générateur monoïdal faible fini de $\M$. La propriété universelle de $\Theta$ (cf. \cite{DV3}, exemple~1.13.1) procure, pour chaque $i$, un foncteur monoïdal $\phi_i : \Theta\to\M$ envoyant $\mathbf{1}$ sur $a_i$. Le foncteur $\Phi$ composé du produit des $\phi_i$ et du foncteur de somme itérée $\M^n\to\M$ convient.
\end{proof}

\begin{pr}\label{polfor-noeth2}
Soient $\M$ un objet de $\mi$ vérifiant $(P_0)$, $\A$ une catégorie de Grothendieck et $F : \M\to\A$ un foncteur fortement polynomial et ponctuellement noethérien. Alors $F$ est noethérien. 
\end{pr}

\begin{proof}
 La précomposition de $F$ par un foncteur $\Phi$ comme dans le lemme~\ref{lmevf} est encore fortement polynomiale et ponctuellement noethérienne. Du fait que $\Phi$ est faiblement essentiellement surjectif, le foncteur de précomposition $\Phi^*$ est exact et fidèle, donc la noethérianité de $\Phi^* F$ entraîne celle de $F$. Autrement dit, il suffit de vérifier la proposition pour $\M=\Theta^n$. La propriété s'obtient alors par récurrence sur $n$ à partir de la proposition~\ref{FInoeth}  en notant qu'un foncteur fortement polynomial et ponctuellement noethérien est de type fini (par la proposition~\ref{pr-P11} et le corollaire~\ref{pf-gal2}) et que, via les isomorphismes de catégories 
$$\fct(\M\times\N,\A)\simeq\fct(\M,\fct(\N,\A))$$
où $\M$ et $\N$ sont deux objets de $\mi$, un foncteur fortement polynomial de la catégorie de gauche est envoyé sur un foncteur fortement polynomial et à valeurs fortement polynomiales dans la catégorie de droite.
\end{proof}

\begin{rem}
 \begin{enumerate}
\item L'auteur ignore si le résultat de la proposition~\ref{polfor-noeth2} (ou même seulement de la proposition~\ref{noeth-eta}) persiste en remplaçant l'hypothèse d'un foncteur ponctuellement noethérien par celle d'un foncteur faiblement ponctuellement noethérien. C'est une difficulté générale due au traitement de foncteurs non fortement noethériens (cf. proposition~\ref{pr-pni}). 
\item Contrairement à ce qui arrive dans la proposition~\ref{polfor-noeth2}, pour les résultats de la suite de cet article, il semble illusoire de penser tout ramener au cas de la catégorie source $\Theta$, ne serait-ce que parce que le foncteur section $s$ n'a nulle raison de se comporter agréablement relativement aux foncteurs de précomposition par un morphisme de $\mi$.
 \end{enumerate}
\end{rem}

\section{Propriétés de finitude des foncteurs faiblement polynomiaux}\label{spffp}

Nous nous intéressons maintenant aux propriétés de finitude des foncteurs de $\Pol_d^{{\rm faible}}(\M,\A)$. L'idée générale de la démonstration est toujours de raisonner par  récurrence sur le degré en utilisant les propriétés des différents foncteurs disponibles ($\tau$, $\delta$, $\kappa$...), mais une difficulté apparaît du fait que la proposition~\ref{pr-P11} n'a pas d'analogue pour les foncteurs faiblement polynomiaux, même avec une <<~très bonne~>> catégorie source $\M$, comme l'illustre l'exemple suivant.

\begin{ex}\label{expn-th}
Dans la catégorie $\fct(\Theta,\mathbf{Ab})$ (qui est localement noethérienne --- cf. proposition~\ref{FInoeth}), pour $n\in\mathbb{N}$, notons $\mathbb{Z}_n$ le foncteur égal à $\mathbb{Z}$ sur les ensembles de cardinal $n$ et nul ailleurs. Le foncteur $\underset{n\in\mathbb{N}}{\bigoplus}\mathbb{Z}_n$ de $\fct(\Theta,\mathbf{Ab})$ est stablement nul, donc faiblement polynomial de degré $-\infty$, et ponctuellement noethérien, mais il n'est pas de type fini.

Pourtant, la catégorie $\Theta$ vérifie toutes nos hypothèses de finitude.
\end{ex}

Nous nous limiterons donc, dans un premier temps (corollaire~\ref{cornp}), aux foncteurs faiblement polynomiaux et $\s n(\M,\A)$-fermés de $\fct(\M,\A)$ (i.e. ceux sur lesquels l'unité ${\rm Id}\to s\pi$ de l'adjonction est un isomorphisme).

Pour cela, nous aurons besoin d'un résultat crucial sur les foncteurs dérivés à droite $\mathbf{R}^i(s)$ du foncteur section $s : \st(\M,\A)\to\fct(\M,\A)$. Comme la composée de $s$ avec le foncteur exact $\pi$ est isomorphe à l'identité, donc exacte, pour $i>0$, $\mathbf{R}^i(s)$ prend ses valeurs dans $\s n(\M,\A)$. De plus, on rappelle que les résultats de \cite{DV3} sur le foncteur $\eta$ impliquent que  $\mathbf{R}^i(s)(\pi\eta A)=0$ pour tout $i>0$ et tout foncteur $A$ de $\fct(\widetilde{\M},\A)$ (tandis que $s\pi\eta(A)\simeq\eta(A)$).

\smallskip

Nous commençons par deux lemmes qui donnent des propriétés qui ressemblent à des propriétés de présentation finie pour l'image dans $\st(\M,\A)$ des foncteurs constants.

\begin{lm}\label{lm-tck}
 Soient $\M$ un objet de $\mi$, $\A$ une catégorie de Grothendieck, $F : \M\to\A$ un foncteur, $C$ un objet de $\A$ (vu aussi comme foncteur constant $\M\to\A$) et $f : F\to C$ un morphisme stablement surjectif (i.e. dont l'image dans $\st(\M,\A)$ est un épimorphisme). Si $S$ est un support de $F$ et $x$ un objet de $\M$ tel que $\M(s,x)\neq\emptyset$ pour tout $s\in S$, alors $f(x) : F(x)\to C$ est surjectif.
\end{lm}

\begin{proof}
 Il suffit de vérifier que $\varphi : \underset{s\in S}{\bigoplus}F(s)\to C$ (de composantes les $f(s)$) est surjectif. Pour cela, on note que, pour tout objet $t$ de $\M$, le diagramme commutatif
$$\xymatrix{\bigoplus_{s\in S}F(s)[\M(s,t)]\ar@{>>}[r]\ar[d] & F(t)\ar[d]^-{f(t)} \\
\underset{s\in S}{\bigoplus}F(s)\ar[r]^\varphi & C
}$$
dont la flèche verticale de gauche est la somme directe des augmentations $F(s)[\M(s,t)]\to F(s)$ et la flèche horizontale supérieure le morphisme canonique, dont le fait que $S$ est un support de $F$ garantit la surjectivité, montre que l'image de $\varphi$ contient l'image de $f(t)$. Comme $C$ est la réunion de ces images, puisque $F$ est stablement surjectif, on en déduit ${\rm Im}\,\varphi=C$ comme souhaité.
\end{proof}

\begin{lm}\label{lm-coker}
 Soient $\M$ un objet de $\mi$ vérifiant $(P_0)$ et $\A$ une catégorie de Grothendieck. Il existe un objet $a$ de $\M$ tel que pour tous $d\in\mathbb{N}$, $F\in {\rm Ob}\,\Pol_d^{{\rm fort}}(\M,\A)$, $C\in {\rm Ob}\,\A$ et $f : F\to C$ stablement surjectif, $\tau_{a^{+d}}(f)$ soit surjectif.
\end{lm}

\begin{proof}
Le foncteur fortement polynomial $F$ possède un support fini donné par la proposition~\ref{pr-P11}. On déduit alors le résultat du lemme~\ref{lm-tck} en prenant pour $a$ la somme des éléments d'un ensemble générateur monoïdal faible fini de $\M$.
\end{proof}

\subsection{Propriété fondamentale du foncteur section}\label{par-sect}

La proposition suivante constitue l'un des principaux résultats de ce travail. Elle est de nature très générale, ne requérant que peu d'hypothèses de finitude sur la catégorie source et aucune hypothèse de noethérianité ; sa conclusion n'est d'ailleurs pas vraiment une propriété de finitude. Elle jouera un rôle complémentaire de la proposition~\ref{noeth-eta} pour établir des propriétés de finitude sur les foncteurs faiblement polynomiaux.

\begin{pr}\label{prn-princ}
  Soient $\M$ un objet de $\mi$ vérifiant $(P_0)$ et $\A$ une catégorie de Grothendieck. Pour tout $d\in\mathbb{N}$, il existe des objets $r(d)$ et $q(d)$ de $\M$ tels que, pour tout objet $X$ de $\Pol_d(\M,\A)$, on ait  :
\begin{enumerate}
 \item $\tau_{r(d)}(\mathbf{R}^i s)(X)=0$ pour $1\leq i\leq d$ et $(\mathbf{R}^i s)(X)=0$ pour $i>d$ ;
\item $\tau_{q(d)}s(X)\in {\rm Ob}\,\Pol^{{\rm fort}}_d(\M,\A)$.
\end{enumerate}
\end{pr}

\begin{proof}
 On procède par une récurrence imbriquée sur trois hypothèses, où $d\geq d'\geq 0$ sont des entiers et $r(d)$, $q(d)$ et $t(d,d')$ sont des objets de $\M$ :
\begin{itemize}
 \item $(HR_1(d))$ pour tout $X\in {\rm Ob}\,\Pol_d(\M,\A)$, $\tau_{r(d)}(\mathbf{R}^i s)(X)=0$ pour $1\leq i\leq d$ et $(\mathbf{R}^i s)(X)=0$ pour $i>d$ ;
\item $(HR_2(d))$ pour tout $X\in {\rm Ob}\,\Pol_d(\M,\A)$, $\tau_{q(d)}s(X)\in {\rm Ob}\,\Pol^{{\rm fort}}_d(\M,\A)$ ;
\item $(HR_3(d,d'))$ pour tous $A\in {\rm Ob}\,\Pol_d(\widetilde{\M},\A)$, $X\in {\rm Ob}\,\Pol_{d'}(\M,\A)$ et tout épimorphisme $f : \pi\eta(A)\twoheadrightarrow X$ de $\st(\M,\A)$, $\tau_{t(d,d')}({\rm Coker}\,s(f))=0$.
\end{itemize}

Pour l'initialisation, on note que $HR_1(0)$ et $HR_2(0)$ sont vérifiées avec $r(0)=q(0)=0$ (puisque $\Pol_0(\M,\A)$ est équivalente à $\A$ via l'inclusion des foncteurs constants) ; le lemme~\ref{lm-coker} montre que $HR_3(d,0)$ est vérifiée avec $t(d,0)=a^{+d}$ pour un objet convenable $a$ de $\M$.

Le résultat découle alors de la conjonction des trois lemmes qui suivent, où l'on suppose $d\geq d'>0$.
\end{proof}

\begin{lm}\label{lma1}
Si $HR_1(d-1)$ et $HR_3(d,d-1)$ sont vérifiées, alors $HR_1(d)$ l'est également, et tout $r(d)$ tel qu'existent des objets $a'$ et $a''$ de $\M$ tels que  $r(d)\simeq r(d-1)+a'$ et $r(d)\simeq t(d,d-1)+a''$ (par exemple, $r(d)=r(d-1)+t(d,d-1)$) convient.
\end{lm}

\begin{proof}
On sait qu'il existe dans $\st(\M,\A)$ une suite exacte
$$0\to Y\to X\xrightarrow{u}\pi\eta(A)\xrightarrow{v} Z\to 0$$
avec $Y$ et $Z$ dans $\Pol_{d-1}(\M,\A)$ et $A$ dans $\Pol_d(\widetilde{\M},\A)$ (on rappelle que cela provient de ce que le foncteur $\eta :\fct(\widetilde{\M},\A)\to\fct(\M,\A)$ induit, d'après \cite{DV3}, un foncteur essentiellement surjectif --- et même une équivalence --- $$\Pol_d(\widetilde{\M},\A)/\Pol_{d-1}(\widetilde{\M},\A)\to\Pol_d(\M,\A)/\Pol_d(\M,\A)).$$
Notons $L$ l'image de $u$ : la suite exacte courte $0\to L\to\pi\eta(A)\to Z\xrightarrow{v} 0$ et les relations $(\mathbf{R}^i s)(\pi\eta(A))=0$ pour $i>0$ montrent que $(\mathbf{R}^i s)(L)\simeq (\mathbf{R}^{i-1}s)(Z)$ pour $i\geq 2$, d'où $\tau_{r(d-1)}(\mathbf{R}^i s)(X)=0$ pour $2\leq i\leq d$ et $(\mathbf{R}^i s)(X)=0$ pour $i>d$ par $HR_1(d-1)$, tandis que $(\mathbf{R}^1 s)(L)$ est le conoyau de $s(v)$. Ainsi $HR_3(d,d-1)$ fournit $\tau_{t(d,d-1)}\mathbf{R}^1(s)(L)=0$.

Maintenant, la suite exacte courte $0\to Y\to X\to L\to 0$ fournit des suites exactes $(\mathbf{R}^i s)(Y)\to (\mathbf{R}^i s)(X)\to (\mathbf{R}^i s)(L)$ pour tout $i\in\mathbb{N}^*$ : ce qu'on vient de montrer sur les $(\mathbf{R}^i s)(L)$ et ce que $HR_1(d-1)$ implique pour $Y$ entraîne le résultat souhaité pour $X$.
\end{proof}

\begin{lm}\label{lma2}
Si $HR_2(d-1)$ et $HR_1(d)$ sont vérifiées, alors $HR_2(d)$ l'est également avec $q(d)=q(d-1)+r(d)$.
\end{lm}

\begin{proof}
 On procède selon une méthode analogue à celle du lemme~\ref{lma1}, dont on conserve les notations de la démonstration. La suite exacte
$$0\to s(L)\to\eta(A)\to s(Z)\to(\mathbf{R}^1 s)(L)\to 0$$
de $\fct(\M,\A)$ induit, grâce à $HR_1(d)$, une suite exacte courte
$$0\to\tau_{r(d)} s(L)\to\tau_{r(d)}\eta(A)\to\tau_{r(d)}s(Z)\to 0\;;$$
autrement dit, en utilisant la commutation du foncteur section $s$ et de $\tau_{r(d)}$, la suite exacte courte $0\to L\to\pi\eta(A)\to Z\to 0$ induit une suite exacte courte
$$0\to s(\tau_{r(d)} L)\to\eta(\tau_{r(d)}A)\to s(\tau_{r(d)}Z)\to 0.$$
Si $x$ est un objet de $\M$, la suite exacte courte $0\to\delta_x(L)\to\pi\eta\delta_x(A)\to\delta_x(Z)\to 0$ induit une suite exacte courte
$$0\to s(\delta_x(\tau_{r(d)} L))\to\eta(\delta_x(\tau_{r(d)}A))\to s(\delta_x(\tau_{r(d)}Z))\to 0$$
pour la même raison que précédemment : $\delta_x(L)$ appartient à $\Pol_d(\M,\A)$, de sorte que $\tau_{r(d)}(\mathbf{R}^1 s)(\delta_x(L))=0$ par l'hypothèse $HR_1(d)$ (on utilise également sans cesse la commutation de $\tau_{r(d)}$ à tous les foncteurs utilisés).

Considérons le diagramme commutatif aux lignes exactes
$$\xymatrix{0\ar[r] & \delta_x s(\tau_{r(d)} L)\ar[r]\ar[d] & \eta(\delta_x\tau_{r(d)} A)\ar[r]\ar[d]^\simeq &\delta_x s(\tau_{r(d)} Z)\ar[r]\ar[d] & 0\\
0\ar[r] & s(\delta_x(\tau_{r(d)} L))\ar[r] & \eta(\delta_x(\tau_{r(d)}A))\ar[r] & s(\delta_x(\tau_{r(d)}Z))\ar[r] & 0
}$$
de $\fct(\M,\A)$ (la suite supérieure est exacte à gauche car $\kappa_x(s(\tau_{r(d)} Z))=0$) dont les flèches verticales sont les {\em monomorphismes} naturels : il montre que le morphisme $\delta_x s(\tau_{r(d)} L)\to s(\delta_x(\tau_{r(d)} L))$ est un isomorphisme. Comme $\tau_{q(d-1)}s(\delta_x(\tau_{r(d)} L))$ appartient à $\Pol^{{\rm fort}}_{d-1}(\M,\A)$ par $HR_2(d-1)$, on en tire que $\tau_{q(d-1)}s(\tau_{r(d)} L)\simeq\tau_{q(d-1)+r(d)}s(L)$ appartient à $\Pol^{{\rm fort}}_{d}(\M,\A)$.

Revenons à $X$ : la suite exacte courte $0\to Y\to X\to L\to 0$ et $HR_1(d)$ fournissent une suite exacte courte
$$0\to\tau_{r(d)+q(d-1)}s(Y)\to\tau_{r(d)+q(d-1)}s(X)\to\tau_{r(d)+q(d-1)}s(L)\to 0\;;$$
comme $\tau_{r(d)+q(d-1)}s(L)$ appartient à $\Pol^{{\rm fort}}_{d}(\M,\A)$ comme on vient de le voir, et $\tau_{r(d)+q(d-1)}s(Y)$ aussi (par $HR_2(d-1)$), on voit qu'il en est de même pour  $\tau_{r(d)+q(d-1)}s(X)$, d'où le lemme.
\end{proof}

\begin{lm}\label{lma3}
Si $HR_3(d-1,d'-1)$ et $HR_2(d-1)$ sont vérifiées, alors $HR_3(d,d')$ l'est également avec $t(d,d')=t(d-1,d'-1)+q(d-1)+a^{+d}$.
\end{lm}

\begin{proof}
 Soit $f : \pi\eta(A)\twoheadrightarrow X$ un épimorphisme de $\st(\M,\A)$, avec $A\in {\rm Ob}\,\Pol_d(\widetilde{\M},\A)$ et $X\in {\rm Ob}\,\Pol_{d'}(\M,\A)$. Pour tout objet $x$ de $\M$, $\tau_{t(d-1,d'-1)}s(\delta_x (f))$ est un épimorphisme (par l'hypothèse $HR_3(d-1,d'-1)$). En utilisant le monomorphisme naturel $\delta_x s\hookrightarrow s\delta_x$ (comme dans la démonstration du lemme~\ref{lma2}), on voit que le diagramme commutatif
$$\xymatrix{\eta\delta_x(A)\ar@{=}[d]\ar[r]^-\simeq & \delta_x s\big(\pi\eta(A)\big)\ar[rr]^-{\delta_x (s(f))} & & \delta_x(s X)\ar[d]\\
\eta\delta_x(A)\ar[r]^-\simeq & s\delta_x\big(\pi\eta(A)\big)\ar[rr]^-{s\delta_x(f)} & & s\delta_x(X)
}$$
où la flèche verticale de droite est un monomorphisme procure un monomorphisme ${\rm Coker}\,\delta_x(sf)\hookrightarrow{\rm Coker}\,s\delta_x(f)$. Par conséquent, $\tau_{t(d-1,d'-1)}\delta_x(sf)\simeq\delta_x s(f')$ est un épimorphisme, où l'on note $f':=\tau_{t(d-1,d'-1)}(f)$ --- nous désignerons de manière générale, dans la suite, par un prime l'application du foncteur $\tau_{t(d-1,d'-1)}$, pour alléger.

Le conoyau $N$ de $s(f') : \eta(A')(=\eta\tau_{t(d-1,d'-1)}(A))\to s(X')(=s\tau_{t(d-1,d'-1)}(X))$ appartient donc à $\Pol_0^{{\rm fort}}(\M,\A)$ ; autrement dit, le morphisme canonique $N(0)\to N$ est un épimorphisme. Notons $F$ le produit fibré du monomorphisme canonique $sX'(0)\hookrightarrow s(X')$ et de $s(f')$, de sorte qu'on dispose d'une suite exacte courte
$$0\to F\to\eta(A')\oplus s(X')(0)\to s(X')\to 0.$$
En appliquant $s\pi$, on voit que l'unité $F\to s\pi(F)$ est un isomorphisme.

Raisonnant encore comme précédemment, on obtient pour tout $x\in {\rm Ob}\,\M$ un diagramme commutatif aux lignes exactes
$$\xymatrix{0\ar[r] & \delta_x(F)\ar[r]\ar[d] & \eta\delta_x(A')\ar@{=}[d]\ar[r] & \delta_x s(X')\ar[r]\ar[d] & 0\\
0\ar[r] & s\pi\delta_x(F)\ar[r] & \eta\delta_x(A')\ar[r] & s\delta_x(X') & 
}$$
dont les flèches verticales sont des monomorphismes, d'où l'on déduit que l'unité $\delta_x(F)\to s\pi\delta_x(F)$ est un isomorphisme. Comme $\pi\delta_x(F)$, qui est un sous-objet de $\pi\eta\delta_x(A')$, appartient à $\Pol_{d-1}(\M,\A)$, on en déduit, grâce à $HR_2(d-1)$, que $\tau_{q(d-1)}\delta_x(F)$ appartient à $\Pol_{d-1}^{{\rm fort}}(\M,\A)$, donc que $\tau_{q(d-1)}(F)$ appartient à $\Pol_{d}^{{\rm fort}}(\M,\A)$.

On utilise maintenant la suite exacte
$$F\xrightarrow{u} s(X')(0)\to N\to 0$$
déduite des définitions de $F$ et de $N$ et de ce que la composée $s(X')(0)\hookrightarrow s(X')\twoheadrightarrow N$ est un épimorphisme : en lui appliquant $\tau_{q(d-1)}$ et en appliquant le lemme~\ref{lm-coker} à $u$ (on rappelle que $N$ est stablement nul car $f$, donc $f'$, est un épimorphisme de $\st(\M,\A)$), on obtient $\tau_{q(d-1)+a^{+d}}(N)\simeq\tau_{a^{+d}}(\tau_{q(d-1)} N)=0$. Comme $N={\rm Coker}\,s\tau_{t(d-1,d'-1)}(f)$ par définition, on en déduit (utilisant toujours la commutation entre $s$ et les $\tau_b$)
$$\tau_{t(d-1,d'-1)+q(d-1)+ a^{+d}}{\rm Coker}\,s(f)=0,$$
ce qui achève la démonstration.
\end{proof}

\begin{cor}\label{cornp}
 Soient $\M$ un objet de $\mi$ vérifiant $(P_0)$ et $(P_1)$, $\A$ une catégorie de Grothendieck, $d\in\mathbb{N}$, $X\in {\rm Ob}\,\Pol_d(\M,\A)$ et $F:=s(X)$. Si le foncteur $F : \M\to\A$ est ponctuellement noethérien, alors il est noethérien.
\end{cor}

\begin{proof}
 Par la proposition~\ref{prn-princ}, il existe un objet $x$ de $\M$ tel que $\tau_x(F)$ soit fortement polynomial. Comme ce foncteur est ponctuellement noethérien, la proposition~\ref{polfor-noeth2} montre qu'il est noethérien. En utilisant la proposition~\ref{prp2}, on en déduit le résultat.
\end{proof}

Notre but consiste maintenant à montrer que, sous certaines hypothèses sur $\M$, la catégorie des foncteurs faiblement polynomiaux de $\M$ vers une catégorie de Grothendieck localement noethérienne est localement noethérienne (ou presque). Cela exigera, outre la proposition~\ref{prn-princ}, l'utilisation de l'hypothèse de finitude $(P_4)$ et la démonstration de quelques résultats intermédiaires de présentation à support fini, qui possèdent également un intérêt intrinsèque.

\subsection{Propriétés autour de la présentation finie}\label{sect-psf}

On commence par un résultat général sur les foncteurs à support fini qui explique l'utilité de la propriété $(P_4)$.

\begin{lm}\label{lmtf}
 Soient $\M$ un objet de $\mi$, $\A$ une catégorie abélienne et $F$ un foncteur de $\fct(\M,\A)$. Considérons les propriétés suivantes :
\begin{enumerate}
 \item\label{tf1} $F$ est à support fini ;
\item\label{tf2} pour tout objet $x$ de $\M$, $\tau_x(F)$ est à support fini ;
\item\label{tf3} pour tout objet $x$ de $\M$, $\delta_x(F)$ est à support fini.
\end{enumerate}
Alors \ref{tf1}.$\Rightarrow$\ref{tf2}. si $\M$ vérifie $(P_4)$, \ref{tf2}.$\Rightarrow$\ref{tf3}., et \ref{tf3}.$\Rightarrow$\ref{tf1}. si $\M$ vérifie $(P_0)$.

De plus, il suffit que $\tau_x(F)$ soit à support fini pour {\em un} objet $x$ de $\M$ pour que $F$ soit à support fini, si $\M$ vérifie $(P_1)$.
\end{lm}

\begin{proof}
 {\em \ref{tf1}.$\Rightarrow$\ref{tf2}.} : cela provient de ce que pour tous objets $a$, $t$ de $\M$ et $A$ de $\A$, $U(a,t)$ est un support fini du foncteur $\tau_t(A[\M(a,-)])$.

{\em \ref{tf2}.$\Rightarrow$\ref{tf3}.} découle de ce que $\delta_x(F)$ est un quotient de $\tau_x(F)$.

 {\em \ref{tf3}.$\Rightarrow$\ref{tf1}.} : soient $T$ un ensemble générateur monoïdal faible fini de $\M$, $S_x$ un support fini de $\delta_x(F)$ pour $x\in T$ et
$$S':=\{0\}\cup\underset{x\in T}{\bigcup}(x+S_x).$$

Montrons que l'ensemble fini $S'$ est un support de $F$. Pour cela, on montre que, pour toute famille d'entiers naturels $\mathbf{n}=(n_x)_{x\in T}$, si l'on pose $a_\mathbf{n}:=\underset{x\in S}{\sum}x^{+n_x}$, $F(a_\mathbf{n})$ est somme d'éléments venant de $F(s)$ pour $s\in S'$, par récurrence sur $m:=\underset{x\in T}{\sum}{n_x}$ (cela suffit puisque tout objet de $\M$ est facteur direct d'un tel $a_\mathbf{n}$).

Pour $m=0$, cela résulte de ce que $0\in S'$. Si $m>0$, on choisit un $x\in T$ tel que $n_x>0$, de sorte que $a_\mathbf{n}$ s'écrit $x+b$, où $F(b)$ est somme d'éléments venant de $F(s)$ pour $s\in S'$, par l'hypothèse de récurrence. Comme $S_x$ est un support de $\delta_x(F)$ et que $x+S_x\subset S'$, $F(a_\mathbf{n})=F(x+b)$ est somme d'éléments venant de $F(s)$ pour $s\in S'$ et d'un élément de $F(b)$, qui lui-même est de ce type.

Pour la dernière assertion : si $S$ est un support fini de $\tau_x(F)$, la réunion de $x+S$ et d'un ensemble (fini par $(P_1)$) de représentants des classes d'isomorphisme d'objets de $\M$ n'appartenant pas à l'image essentielle de $x+-$ est un support fini de $F$.
\end{proof}

\begin{rem}
 L'implication \ref{tf3}.$\Rightarrow$\ref{tf1}. (sous $(P_0)$) généralise la proposition~\ref{pr-P11}.
\end{rem}

Une première conséquence, peu surprenante mais ne semblant pas facile à établir sans employer les résultats du paragraphe~\ref{par-sect}, est la suivante.

\begin{pr}\label{pr-sff}
 Soient $\M$ un objet de $\mi$ vérifiant $(P_0)$ et $(P_1)$, $\A$ une catégorie de Grothendieck et $X$ un objet polynomial de $\st(\M,\A)$. Alors $s(X)$ est à support fini.
\end{pr}

\begin{rem}
 On peut trouver un tel support fini qui ne dépende que du degré de $X$.
\end{rem}

\begin{proof}
 C'est une conséquence directe des propositions~\ref{prn-princ} et~\ref{pr-P11} et du lemme~\ref{lmtf}. 
\end{proof}

On s'intéresse maintenant à la propriété $PSF$.

\begin{lm}\label{lm-psfd}
 Soient $\M$ un objet de $\mi$ vérifiant $(P_0)$ et $(P_4)$, $\A$ une catégorie de Grothendieck et $F : \M\to\A$ un foncteur. On suppose que, pour objet $x$ de $\M$, $\delta_x(F)$ est à présentation de support fini et que $\kappa_x(F)$ est à support fini. Alors $F$ est à présentation de support fini.
\end{lm}

\begin{proof}
 Le lemme~\ref{lmtf} montre déjà que $F$ est à support fini. Supposons que $0\to H\to G\to F\to 0$ est une suite exacte de $\fct(\M,\A)$ et que $G$ est à support fini. Pour tout objet $x$ de $\M$, la suite exacte
 $$\kappa_x(F)\to\delta_x(H)\to\delta_x(G)\to\delta_x(F)\to 0$$
 et les hypothèses sur $\kappa_x(F)$ et $\delta_x(F)$ impliquent, comme $\delta_x(G)$ est à support fini d'après le lemme~\ref{lmtf}, que $\delta_x(H)$ est à support fini (utiliser le corollaire~\ref{pf-gal2}). La conclusion découle donc du lemme~\ref{lmtf}.
\end{proof}

On démontre de façon simple et entièrement analogue, en utilisant le foncteur exact $\tau_x$, la propriété suivante, à partir du lemme~\ref{lmtf}.

\begin{lm}\label{lm-psfd2}
 Soient $\M$ un objet de $\mi$ vérifiant $(P_1)$ et $(P_4)$, $\A$ une catégorie de Grothendieck, $F : \M\to\A$ un foncteur et $x$ un objet de $\M$. Si $\tau_x(F)$ est $PSF$, alors $F$ l'est également.
\end{lm}

La propriété suivante constitue le résultat principal de ce paragraphe.

\begin{pr}\label{pr-pff}
 Soient $\M$ un objet de $\mi$ vérifiant $(P_0)$, $(P_1)$ et $(P_4)$, $\A$ une catégorie de Grothendieck, $d\in\mathbb{N}$ et $X\in {\rm Ob}\,\Pol_d(\M,\A)$. Alors $s(X)$ est à présentation de support fini.
\end{pr}

\begin{proof}
On procède par récurrence sur $d$. Soit $x$ un objet de $\M$. Le foncteur $\kappa_x(sX)$ est nul, et le conoyau du monomorphisme naturel $\delta_x(s X)\hookrightarrow s (\delta_x X)$ s'injecte dans $(\mathbf{R}^1 s)(X)$, comme le montre le diagramme commutatif aux lignes exactes suivant.
$$\xymatrix{0\ar[r] & sX\ar[r]\ar@{=}[d] & \tau_x (s X)\ar[r]\ar[d]^\simeq & \delta_x (s X)\ar[r]\ar[d] & 0\\
0\ar[r] & sX\ar[r] &  s(\tau_x X)\ar[r] & s(\delta_x X)\ar[r] & (\mathbf{R}^1 s)(X)
}$$
La proposition~\ref{prn-princ} fournit un objet $r$ de $\M$ tel que $\tau_r (\mathbf{R}^1 s)(X)=0$, de sorte que le morphisme canonique $\tau_r(\delta_x(s X))\to\tau_r(s (\delta_x X))\simeq s(\tau_r\delta_x X)$ est un isomorphisme. L'hypothèse de récurrence montre que $s(\tau_r\delta_x X)$ est $PSF$, puisque $\tau_r\delta_x X$ est polynomial de degré strictement inférieur à $d$ (nul si $d=0$). Le lemme~\ref{lm-psfd2} implique que $\delta_x(sX)$ est $PSF$. Le lemme~\ref{lm-psfd} donne alors la conclusion.
\end{proof}

\begin{cor}\label{crnpsf}
 Soient $\M$ un objet de $\mi$ vérifiant $(P_0)$, $(P_1)$, $(P_2)$, $(P_3)$ et $(P_4)$, $\A$ une catégorie de Grothendieck et $F : \M\to\A$ un foncteur faiblement polynomial. On suppose que $F$ est noethérien et $s\pi(F)$ de type fini. Alors $F$ est $PSF$.
\end{cor}

\begin{proof}
 Notons $N$ et $C$ les noyau et conoyau, respectivement, de l'unité $F\to s\pi(F)$. Alors $N$ et $C$ sont de type fini et stablement nuls. Par les propositions~\ref{psn} et~\ref{lm-pt3}, on en déduit qu'il existe $x\in {\rm Ob}\,\M$ tel que $\tau_x(N)$ et $\tau_x(C)$ soient nuls. Ainsi, $\tau_x(F)\simeq\tau_x(s\pi(F))\simeq s\pi\tau_x(F)$ est $PSF$ par la proposition~\ref{pr-pff}. La conclusion découle donc du lemme~\ref{lm-psfd2}.
\end{proof}

%Nous utiliserons également la variation suivante de la proposition~\ref{pr-pff}.
%
%\begin{cor}\label{cor-pff}
%Soient $\M$ un objet de $\mi$ vérifiant $(P_0)$, $(P_1)$ et $(P_4)$, $(FM)$ et $(FE)$, $\A$ une catégorie de Grothendieck localement noethérienne, $d\in\mathbb{N}$ et $X\in {\rm Ob}\,\Pol_d(\M,\A)$. Si
%$$0\to F\to G\to s(X)\to 0$$
%est une suite exacte de $\fct(\M,\A)$ avec $G$ de type fini, alors $F$ est de type fini.
%\end{cor}
%
%\begin{proof}
% La proposition~\ref{pf-gal2} montre que $G$ est à support fini et faiblement ponctuellement noethérien. Le sous-foncteur $F$ de $G$ est donc également faiblement ponctuellement noethérien. Mais il est à support fini par la proposition~\ref{pr-pff}, il est donc de type fini, toujours d'après la proposition~\ref{pf-gal2}.
%\end{proof}

Le lemme facile qui suit n'a pas de rapport direct avec les catégories de foncteurs mais interviendra un peu plus loin.

\begin{lm}\label{lm-tec}
 Soient $\A$ une catégorie abélienne, $G$ un groupe de type fini et $Ou : \A_G\to\A$ le foncteur d'oubli (cf. notation~\ref{nca}, page~\pageref{nca}, pour la définition de $\A_G$). Si $M$ est un objet de $\A_G$ tel que $Ou(M)$ est de présentation finie dans $\A$, alors $M$ est de présentation finie dans la catégorie $\A_G$.
\end{lm}

\begin{proof}
 On commence par noter que $\mathbb{Z}$ (représentation triviale de $G$, dans les groupes abéliens) est une représentation de présentation finie de $G$ : si $E$ est un ensemble fini de générateurs de $G$, on dispose d'une suite exacte
$$\mathbb{Z}[G\times E]\simeq\mathbb{Z}[G]^{\oplus E}\to\mathbb{Z}[G]\to\mathbb{Z}\to 0$$
où la première flèche est donnée par $[g,e]\mapsto [ge]-[g]$ et la deuxième est l'augmentation. Le résultat s'en déduit en tensorisant cette suite exacte par $M$ et en utilisant le fait que $M\otimes\mathbb{Z}[G]\simeq M_{triv}\otimes\mathbb{Z}[G]$ (où $M_{triv}$ est la représentation de $G$ dans $\A$ dont l'objet de $\A$ sous-jacent est $M$, mais avec une action triviale de $G$), puisque $G$ est un groupe, et l'isomorphisme naturel
$$\A_G(M_{triv}\otimes\mathbb{Z}[G],N)\simeq\A(Ou(M),Ou(N))$$
qui montre que $M_{triv}\otimes\mathbb{Z}[G]$ est de présentation finie dans la catégorie $\A_G$ si $M$ est de présentation finie dans~$\A$.
\end{proof}

\subsection{Résultats fondamentaux sur les foncteurs faiblement polynomiaux}

\begin{lm}\label{lm-ninter}
 Soient $\M$ un objet de $\mi$ vérifiant $(P_0)$, $(P_1)$, $(P_2)$, $(P_3)$ et $(P_4)$, $\A$ une catégorie de Grothendieck et $F : \M\to\A$ un foncteur vérifiant les propriétés suivantes :
\begin{enumerate}
 \item $F$ est faiblement polynomial ;
\item $\kappa(F)=0$ ;
\item $F$ est ponctuellement noethérien ;
\item $F$ est à support fini. 
\end{enumerate}

Alors :
\begin{enumerate}
 \item $\alpha(F) : \widetilde{\M}\to\A$ est ponctuellement noethérien ;
\item $s\pi(F) : \M\to\A$ est ponctuellement noethérien ;
\item $F$ est noethérien.
\end{enumerate}
\end{lm}

\begin{proof}
 On procède par récurrence sur le degré faible $d$ de $F$.

On commence par remarquer que, comme tous les $\tau_t(F)$ sont à support fini d'après le lemme~\ref{lmtf}, le fait que $F$ soit ponctuellement noethérien entraîne que $\alpha(F)$ l'est également (par la proposition~\ref{lmalph}). Notons $G$ et $H$ les noyau et image respectivement de l'unité $F\to\eta\alpha(F)$. Comme l'unité de l'adjonction ${\rm Id}\to s\pi$ est un isomorphisme sur $\eta\alpha(F)$, l'inclusion $H\hookrightarrow\eta\alpha(F)$ se factorise en
$$H\hookrightarrow s\pi(H)\hookrightarrow\eta\alpha(F).$$
Ainsi, $s\pi(H)$ est noethérien et ponctuellement noethérien, comme $\eta\alpha(F)$ (utiliser la proposition~\ref{noeth-eta}), d'où l'on déduit que le conoyau de $H\hookrightarrow s\pi(H)$ est noethérien, ponctuellement noethérien et stablement nul. Par conséquent, en utilisant les propositions~\ref{psn} et~\ref{lm-pt3}, on voit qu'existe un objet $x$ de $\M$ tel que $\tau_x(H)\hookrightarrow\tau_x(s\pi(H))\simeq s\tau_x\pi(H)$ soit un isomorphisme. On dispose donc d'une suite exacte
$$0\to\tau_x(G)\to\tau_x(F)\to s\tau_x\pi(H)\to 0.$$

Comme $\tau_x\pi(H)$ appartient à $\Pol_d(\M,\A)$ et que $\tau_x(F)$ est comme $F$ à support fini par le lemme~\ref{lmtf}, on voit que $\tau_x(G)$, puis $G$ (par le même lemme), est également à support fini, en utilisant la proposition~\ref{pr-pff}. Mais $G$, sous-foncteur de $F$, est également ponctuellement noethérien et tel que $\kappa(G)=0$, et $G$ appartient à $\Pol_{d-1}^{{\rm faible}}(\M,\A)$. L'hypothèse de récurrence montre que $G$ est noethérien et $s\pi(G)$ ponctuellement noethérien. Les suites exactes
$$0\to G\to F\to H\to 0$$
et
$$0\to s\pi(G)\to s\pi(F)\to s\pi(H)$$
montrent donc que $F$ est noethérien et $s\pi(F)$ ponctuellement noethérien (on a vu plus haut que $H$ est noethérien et $s\pi(H)$ ponctuellement noethérien), d'où le lemme.
\end{proof}

Nous utiliserons aussi la variante suivante du lemme~\ref{lm-ninter}.

\begin{lm}\label{lm-ninter2}
 Soient $\M$ un objet de $\mi$ vérifiant $(P_0)$, $(P_1)$, $(P_2)$, $(P_3)$, $(P_4)$, $(FM)$, $(FE)$ et $(EI)$ et $\A$ une catégorie de Grothendieck localement noethérienne. On suppose également que les foncteurs polynomiaux et de type fini $\widetilde{\M}\to\A$ sont ponctuellement noethériens. Soit $F : \M\to\A$ un foncteur vérifiant les propriétés suivantes :
\begin{enumerate}
 \item $F$ est faiblement polynomial ;
\item $\kappa(F)=0$ ;
\item $F$ est de type fini. 
\end{enumerate}

Alors :
\begin{enumerate}
 \item $\alpha(F) : \widetilde{\M}\to\A$ est ponctuellement noethérien ;
\item  $F$ et $s\pi(F) : \M\to\A$ sont ponctuellement noethériens ;
\item $F$ est noethérien.
\end{enumerate}
\end{lm}

\begin{proof}
 Le schéma de la démonstration est exactement le même que dans le lemme~\ref{lm-ninter}, nous nous bornerons donc à indiquer les différences avec la démonstration de ce lemme, dont on conserve les notations.
\begin{enumerate}
 \item Le foncteur $\alpha(F) : \widetilde{\M}\to\A$ est de type fini comme $F$ grâce à la proposition~\ref{lmatf}. Il est également polynomial comme $F$, de sorte que l'hypothèse implique qu'il est ponctuellement noethérien.
\item Le foncteur $G\subset F$ est à support fini (pour la même raison que dans la démonstration du lemme~\ref{lm-ninter}), mais il est aussi faiblement ponctuellement de type fini. Cela provient de ce que $F$ est faiblement ponctuellement de type fini, puisque ce foncteur est de type fini, que $\M$
 vérifie l'hypothèse $(FM)$ et que $\A$ est localement noethérienne (utiliser la proposition~\ref{pf-gal}), de ce que $H$ prend des valeurs qui sont de présentation finie dans la catégorie des représentations des monoïdes d'endomorphismes des objets (utiliser les hypothèses $(FE)$, $(EI)$, le fait que $H$ est ponctuellement noethérien comme $\alpha(F)$, le caractère localement noethérien de $\A$ et le lemme~\ref{lm-tec}) et de la suite exacte $0\to G\to F\to H\to 0$. Le corollaire~\ref{pf-gal2} implique donc que $G$ est de type fini, ce qui permet de lui appliquer l'hypothèse de récurrence.
\end{enumerate}

\end{proof}

\begin{pr}\label{pr-nf2}
 Soient 
$\M$ un objet de $\mi$ vérifiant $(P_0)$, $(P_1)$, $(P_2)$, $(P_3)$ et $(P_4)$, $\A$ une catégorie de Grothendieck et $F : \M\to\A$ un foncteur vérifiant les propriétés suivantes :
\begin{enumerate}
 \item $F$ est faiblement polynomial ;
\item $F$ est ponctuellement noethérien ;
\item $F$ est à support fini. 
\end{enumerate}

Alors :
\begin{enumerate}
 \item $\alpha(F) : \widetilde{\M}\to\A$ est ponctuellement noethérien ;
\item $s\pi(F) : \M\to\A$ est ponctuellement noethérien ;
\item $F$ est noethérien ;
\item $F$ est $PSF$.
\end{enumerate}
\end{pr}

\begin{proof}
 Posons $G:=F/\kappa(F)$. Alors $G$ vérifie les hypothèses du lemme~\ref{lm-ninter}, de sorte que $G$ est noethérien et que $\alpha(G)$ et $s\pi(G)$ sont ponctuellement noethériens. Comme la projection $F\twoheadrightarrow G$ induit des isomorphismes quand on applique les foncteurs $\alpha$ ou $s\pi$, cela démontre déjà les deux premières assertions.

On note par ailleurs que l'inclusion $G\hookrightarrow s\pi(G)$ induit un isomorphisme $\tau_x(G)\hookrightarrow\tau_x(s\pi(G))\simeq s\tau_x\pi(G)$ pour un certain objet $x$ de $\M$ (même raisonnement que dans la démonstration du lemme~\ref{lm-ninter} pour $H$), de sorte qu'on a une suite exacte courte
$$0\to\tau_x\kappa(F)\to\tau_x(F)\to s\tau_x\pi(G)\to 0$$
dont on déduit (comme pour le lemme~\ref{lm-ninter}) que $\kappa(F)$ est à support fini. Ce foncteur est donc de type fini (il est ponctuellement noethérien comme $F$), de sorte que le corollaire~\ref{cor-snn} montre que $\kappa(F)$ est noethérien. Cela implique que $F$ est noethérien, comme souhaité.

La dernière assertion se déduit des précédentes et des corollaires~\ref{crnpsf} et~\ref{cornp}.
\end{proof}

On a également la variante suivante :

\begin{pr}\label{pr-nf3}
 Soient 
$\M$ un objet de $\mi$ vérifiant $(P_0)$, $(P_1)$, $(P_2)$, $(P_3)$, $(P_4)$, $(FM)$, $(FE)$ et $(EI)$ et $\A$ une catégorie de Grothendieck localement noethérienne. On suppose également que les foncteurs polynomiaux et de type fini $\widetilde{\M}\to\A$ sont ponctuellement noethériens. Soit $F : \M\to\A$ un foncteur faiblement polynomial et de type fini.

Alors :
\begin{enumerate}
 \item $\alpha(F) : \widetilde{\M}\to\A$ est ponctuellement noethérien ;
\item $s\pi(F) : \M\to\A$ est ponctuellement noethérien ;
\item $F$ est presque noethérien (en particulier, si $F$ est faiblement ponctuellement noethérien, alors $F$ est noethérien) ;
\item $F$ est $PSF$.
\end{enumerate}
\end{pr}

\begin{proof}
 La démonstration de cette proposition à partir du lemme~\ref{lm-ninter2} est la même que celle de la proposition~\ref{pr-nf2} à partir du lemme~\ref{lm-ninter} (à la fin, utiliser que $F$ est faiblement ponctuellement de type fini, donc $\kappa(F)$ aussi, en utilisant le lemme~\ref{lm-tec}
comme dans la démonstration du lemme~\ref{lm-ninter2}).
\end{proof}

\begin{cor}\label{cor-ln2}
 Supposons que $\M$ est un objet de $\mi$ vérifiant $(P_0)$, $(P_1)$, $(P_2)$, $(P_3)$ et $(P_4)$ et dont les ensembles de morphismes sont tous finis. Pour toute catégorie de Grothendieck localement noethérienne $\A$, la sous-catégorie abélienne des foncteurs faiblement polynomiaux de $\fct(\M,\A)$ est localement noethérienne.
\end{cor}

\begin{proof}
 Du fait que $\A$ est localement noethérienne et que les ensembles de morphismes sont finis dans $\M$, les foncteurs à support fini et ponctuellement noethériens engendrent $\fct(\M,\A)$, et les foncteurs faiblement polynomiaux à support fini et ponctuellement noethériens engendrent la catégorie des foncteurs faiblement polynomiaux. La proposition~\ref{pr-nf2} montrant que de tels foncteurs sont noethériens, cela établit le corollaire.
\end{proof}

À titre d'application, on obtient l'un des résultats annoncés en introduction (et surpassé, au moins si l'anneau $A$ est commutatif, par \cite{PSam}) :

\begin{thm}\label{th-cf}
 Soient $A$ un anneau fini et $\A$ une catégorie de Grothendieck localement noethérienne. La catégorie des foncteurs faiblement polynomiaux de $\fct(\mathbf{S}(A),\A)$ est localement noethérienne.
\end{thm}

\begin{proof}
 La catégorie $\mathbf{S}(A)$ vérifie les hypothèses $(P_i)$ pour $i\leq 4$ d'après les propositions~\ref{pfpa} et~\ref{pr-ev}, et ses ensembles d'enfomorphismes sont finis puisque $A$ est fini. On peut donc appliquer le corollaire~\ref{cor-ln2}.
\end{proof}

En s'appuyant sur la proposition~\ref{pr-nf3}, on obtient la variante suivante du corollaire~\ref{cor-ln2}.

\begin{cor}\label{cor-ln3}
 Supposons que $\M$ est un objet de $\mi$ vérifiant $(P_0)$, $(P_1)$, $(P_2)$, $(P_3)$, $(P_4)$, $(FM)$, $(FE)$ et $(EI)$ et $\A$ une catégorie de Grothendieck localement noethérienne. On suppose également que les foncteurs polynomiaux et de type fini $\widetilde{\M}\to\A$ sont ponctuellement noethériens. Alors la sous-catégorie abélienne des foncteurs faiblement polynomiaux de $\fct(\M,\A)$ est localement presque noethérienne (i.e. est engendrée par des foncteurs presque noethériens). La sous-catégorie de $\st(\M,\A)$ des objets polynomiaux est localement noethérienne.
\end{cor}

Ce corollaire entraîne aussitôt l'un des résultats principaux de cet article.

\begin{thm}\label{thsab}
 Si $\A$ est une catégorie de Grothendieck localement noethérienne, alors la catégorie des foncteurs faiblement polynomiaux de $\mathbf{S}(\mathbb{Z})$ vers $\A$ est localement presque noethérienne, et la sous-catégorie de $\st(\mathbf{S}(\mathbb{Z}),\A)$ des objets polynomiaux est localement noethérienne.
\end{thm}

\begin{proof}
 Il est clair que $\mathbf{S}(\mathbb{Z})$ vérifie les propriétés $(EI)$, $(P_0)$, $(P_1)$, $(P_2)$ et $(P_3)$. Comme $\mathbb{Z}$ est un anneau principal, cette catégorie vérifie également $(P_4)$ et $(FM)$ (les $GL_n(\mathbb{Z})$-ensembles $\mathbf{S}(\mathbb{Z})(\mathbb{Z}^i,\mathbb{Z}^n)$ sont transitifs). La propriété $(FE)$ provient de ce que les groupes $GL_n(\mathbb{Z})$ sont de type fini.

Par ailleurs, l'un des principaux résultats de $\cite{DV3}$ montre que le foncteur évident $\mathbf{S}(\mathbb{Z})\to\mathbf{P}(\mathbb{Z})^{op}\times\mathbf{P}(\mathbb{Z})$ induit une équivalence de catégories des foncteurs polynomiaux $\mathbf{P}(\mathbb{Z})^{op}\times\mathbf{P}(\mathbb{Z})\to\A$ vers les foncteurs polynomiaux $\widetilde{\mathbf{S}(\mathbb{Z})}\to\A$. La proposition~\ref{pr-adln} montrant que les foncteurs polynomiaux de type fini $\widetilde{\mathbf{S}(\mathbb{Z})}\to\A$ sont ponctuellement noethériens, on peut appliquer le corollaire~\ref{cor-ln3} pour obtenir la conclusion.
\end{proof}

\begin{rem}
 On ne peut pas s'affranchir de l'adverbe {\em presque} dans le thé\-orème~\ref{thsab} (et bien sûr pas non plus dans la proposition~\ref{pr-nf3} ni dans le corollaire~\ref{cor-ln3}). En effet, la catégorie des foncteurs presque nuls $\mathbf{S}(\mathbb{Z})\to\mathbf{Ab}$ n'est déjà pas localement noethérienne, car, pour $n>1$, la catégorie des $\mathbb{Z}[GL_n(\mathbb{Z})]$-modules n'est pas localement noethérienne. Cela provient de ce que $GL_n(\mathbb{Z})$ n'est pas un groupe noethérien pour $n>1$, car c'est déjà le cas de son sous-quotient $PSL_2(\mathbb{Z})$. Il est en effet classique que ce groupe est isomorphe à $\mathbb{Z}/2*\mathbb{Z}/3$, qui contient un groupe libre non commutatif (le noyau de la projection $\mathbb{Z}/2*\mathbb{Z}/3\twoheadrightarrow\mathbb{Z}/2\times\mathbb{Z}/3$ par exemple), donc un groupe libre de rang infini.
\end{rem}

\section{Applications autour des groupes d'automorphismes des groupes libres}\label{sappl}

Soient $V$ un groupe abélien et $n\in\mathbb{N}$. On note $\G r^s_{\leq n}(V)$ (pour {\em grassmannienne scindée}) l'ensemble des couples $(A,B)$ de sous-groupes de $V$ tels que $V=A\oplus B$, où $A$ est un groupe abélien libre de rang au plus $n$ ; on munit cet ensemble de la relation d'ordre $\leq$ définie par $(A,B)\leq (C,D)$ si $A\subset C$ et $B\supset D$. On obtient ainsi un foncteur $\G r^s_{\leq n}$ de la catégorie $\mathbf{S}(\mathbb{Z})$ vers la catégorie des ensembles ordonnés.

Avec les notations du début du §\,\ref{sect-sk}, $\G r^s_{\leq n}(V)$, vu comme une petite catégorie, s'identifie à $S_n[\varphi_n,V]$, où $\varphi_n$ désigne l'inclusion de la sous-catégorie pleine $S_n$ de $\mathbf{S}(\mathbb{Z})$ constituée des $\mathbb{Z}^i$ pour $i\leq n$. Par conséquent, la proposition suivante n'est qu'une reformulation d'une partie de la proposition~\ref{pr-suppFini} et de la proposition-définition~\ref{pres-suppFini}.

\begin{pr}\label{prp-id}
 Soient $F : \mathbf{S}(\mathbb{Z})\to\mathbf{Ab}$ un foncteur et $n\in\mathbb{N}$.
 \begin{enumerate}
  \item L'ensemble des groupes abéliens $\mathbb{Z}^i$, pour $i\leq n$, est un support de $F$ si et seulement si le morphisme canonique
  $$\underset{W=(A,B)\in\G r^s_{\leq n}(V)}{\col}F(A)\to F(V)$$
  est un épimorphisme pour tout objet $V$ de $\mathbf{S}(\mathbb{Z})$.
  \item L'ensemble des groupes abéliens $\mathbb{Z}^i$, pour $i\leq n$, est un support de présentation de $F$ si et seulement si le morphisme précédent est un isomorphisme pour tout objet $V$ de $\mathbf{S}(\mathbb{Z})$.
 \end{enumerate}
\end{pr}

Nous donnons maintenant quelques applications des résultats de la section~\ref{spffp} à des foncteurs $\mathbf{S}(\mathbb{Z})\to\mathbf{Ab}$ liés aux groupes d'automorphismes des groupes libres, en reprenant les notations de la section~4.2 de \cite{DV3}, à laquelle on se réfère pour les détails. Il existe un foncteur {\em groupe d'automorphismes} Aut de la catégorie des groupes libres de rang fini, les morphismes étant les monomorphismes de groupes dont l'image possède un complément pour le produit libre, ce complément faisant partie de la structure (catégorie notée $\G$ dans \cite{DV3}), vers la catégorie $\mathbf{Grp}$ des groupes. Un sous-foncteur particulièrement intéressant et difficile à étudier est le sous-foncteur noté $IA$, où $IA(G)$ est constitué des automorphismes de $G$ qui induisant l'identité sur l'abélianisation de $G$. La composée du foncteur $IA : \G\to\mathbf{Grp}$ avec le foncteur $\gamma_n/\gamma_{n+1} : \mathbf{Grp}\to\mathbf{Ab}$, où $n$ est un entier naturel et $\gamma_n(T)$ désigne le $n$-ème terme de la suite centrale descendante d'un groupe $T$, se factorise par le foncteur d'abélianisation $\G\to\mathbf{S}(\mathbb{Z})$ ; par abus de notation, on désignera encore par $\gamma_n(IA)/\gamma_{n+1}(IA) : \mathbf{S}(\mathbb{Z})\to\mathbf{Ab}$ la factorisation en question. Il n'est pas difficile de voir que ce foncteur est fortement polynomial (de degré au plus $3n$ --- cf. \cite{DV3}, proposition~4.12) et ponctuellement noethérien. Par conséquent, il est noethérien, par la proposition~\ref{polfor-noeth2}. La proposition~\ref{pr-nf2} permet de voir qu'il est également $PSF$ (ce que, contrairement à la propriété noethérienne, on ne semble pas pouvoir déduire aisément de la simple considération de foncteurs depuis $\Theta$, puisque le comportement de la propriété $PSF$ par précomposition par le foncteur canonique $\Theta\to\mathbf{S}(\mathbb{Z})$ n'a rien d'évident), ce qui se traduit, compte-tenu de la proposition~\ref{prp-id}, de la manière suivante.

\begin{pr}\label{th-ia}
 Soit $n\in\mathbb{N}$. Il existe $N\in\mathbb{N}$ tel que, pour tout groupe abélien libre de rang fini $V$, le morphisme canonique
$$\underset{W=(R,S)\in\G r^s_{\leq N}(V)}{\col}(\gamma_n(IA)/\gamma_{n+1}(IA))(R)\to (\gamma_n(IA)/\gamma_{n+1}(IA))(V)$$ 
 est un isomorphisme.
\end{pr}

Une autre filtration importante du foncteur $IA : \G\to\mathbf{Grp}$ est la {\em filtration d'Andreadakis} $(\A_n)$ : $\A_n(G)$ est le groupe des automorphismes de $G$ qui induisent l'identité sur $G/\gamma_{n+1}(G)$. Comme précédemment, $\A_n/\A_{n+1}$ induit un foncteur $\mathbf{S}(\mathbb{Z})\to\mathbf{Ab}$ (encore noté de la même façon) ; on vérifie qu'il est ponctuellement noethérien, faiblement polynomial (de degré au plus $n+2$ --- cf. \cite{DV3}, proposition~4.13). C'est en fait un sous-foncteur d'un foncteur fortement polynomial et ponctuellement noethérien, donc un foncteur noethérien. Par conséquent :

\begin{pr}\label{th-and}
 Soit $n\in\mathbb{N}$. Il existe $N\in\mathbb{N}$ tel que, pour tout groupe abélien libre de rang fini $V$, le morphisme canonique
$$\underset{W=(R,S)\in\G r^s_{\leq N}(V)}{\col}(\A_n/\A_{n+1})(R)\to (\A_n/\A_{n+1})(V)$$ 
 est un isomorphisme.
\end{pr}

Ces applications sont inspirées du théorème~C de l'article \cite{CEFN} de Church, Ellenberg, Farb et Nagpal, et des applications qu'il en tire, comme pour l'homologie de certains groupes de congruences (théorème~1.6 du même article). D'ailleurs, comme l'homologie des groupes de congruences définit non seulement des foncteurs sur $\Theta$, mais aussi sur $\mathbf{S}(\mathbb{Z})$ (voir la fin du §\,4.2 de \cite{DV3} pour une discussion plus détaillée à ce sujet), les mêmes arguments que ceux qu'on a présentés dans le contexte de foncteurs liés aux automorphismes des groupes libres peut s'appliquer aux foncteurs d'homologie de groupes de congruences considérés dans l'article \cite{CEFN}, de sorte qu'on peut obtenir une variation de son théorème~1.6 en remplaçant l'ensemble ordonné des sous-objets de cardinal au plus $N$ d'un ensemble fini par la grassmannienne scindée $\G r^s_{\leq N}$.

\section{Propriétés de finitude de $\Pol_d(\M,\A)$}\label{seckrul}

\begin{conv}
 Dans toute cette section, $\M$ désigne un objet de $\mi$ vérifiant la propriété $(P_0)$ et $\A$ une catégorie de Grothendieck. 
 
 On note $\{x_1,\dots,x_m\}$ un ensemble générateur monoïdal faible fini de $\M$.
\end{conv}

\begin{pr}\label{pr-plf}
 Pour tout $d\in\mathbb{N}$, il existe un objet $r(d)$ de $\M$ tel que, pour tout objet $t$ de $\M$, le foncteur
 $$\Pol_d(\M,\A)\to\Pol_{d-1}(\M,\A)^m\times\A\quad X\mapsto (\delta_{x_1}(X),\dots,\delta_{x_m}(X),sX(t+r(d)))$$
 est exact et fidèle. Il commute de plus aux colimites. Il commute également aux limites si les produits sont exacts dans $\A$.
\end{pr}

\begin{proof}
 On choisit l'objet $r(d)$ de $\M$ donné par la proposition~\ref{prn-princ}. Celle-ci montre l'exactitude de la restriction à $\Pol_d(\M,\A)$ de la composée du foncteur $s : \st(\M,\A)\to\fct(\M,\A)$ et de l'évaluation en $t+r(d)$, d'où l'exactitude du foncteur en question. Celui-ci est fidèle car le fait que $\{x_1,\dots,x_m\}$ est un ensemble générateur monoïdal faible montre que le noyau du foncteur
 $$\st(\M,\A)\to\st(\M,\A)^m\qquad X\mapsto (\delta_{x_1}(X),\dots,\delta_{x_m}(X))$$
 est constitué des objets polynomiaux de degré $0$, c'est-à-dire ceux dont l'image par le foncteur section est constante.
 
 Les foncteurs $\delta_{x_i} : \st(\M,\A)\to\st(\M,\A)$ commutent aux colimites ; le foncteur section commute aux colimites filtrantes par le corollaire~\ref{corsec}, donc notre foncteur commute aux colimites filtrantes. Comme il est exact, il commute donc à toutes les colimites.
 
 Le foncteur section commute aux limites car c'est un adjoint à gauche. Pour conclure, il suffit donc de voir que les foncteurs $\delta_x : \Pol_d(\M,\A)\to\st(\M,\A)$ commutent aux produits (car cela implique la commutation aux limites pour un foncteur exact) lorsque les produits sont exacts dans $\A$. On note d'abord que c'est le cas pour l'endofoncteur $\delta_x$ de $\fct(\M,\A)$, puisque c'est le conoyau d'une transformation naturelle ${\rm Id}\to\tau_x$ entre foncteurs commutant aux produits. Vérifions maintenant que c'est le cas pour $\Pol_d(\M,\A)\to\st(\M,\A)$, ce qui terminera la démonstration (on rappelle que les inclusions $\Pol_i(\M,\A)\to\st(\M,\A)$ commutent aux produits --- voir \cite{DV3}, proposition~1.22). Soit $(X_i)_{i\in I}$ un ensemble d'objets de $\Pol_d(\M,\A)$. Son produit dans $\st(\M,\A)$ (et donc dans $\Pol_d(\M,\A)$) est
 $$\pi\Big(\prod_{i\in I}s(X_i)\Big)$$
(puisque $s$ commute aux produits), et le morphisme canonique
\begin{equation}\label{mcds}
 \delta_x\Big(\prod_{i\in I}X_i\Big)\to\prod_{i\in I}\delta_x(X_i)
\end{equation}
s'identifie, en utilisant la commutation de $\delta_x$ à $\pi$ et aux produits dans $\fct(\M,\A)$, à la flèche
$$\delta_x\pi\Big(\prod_{i\in I}s(X_i)\Big)\simeq\pi\delta_x\Big(\prod_{i\in I}s(X_i)\Big)\simeq\pi\Big(\prod_{i\in I}\delta_xs(X_i)\Big)\to\pi\Big(\prod_{i\in I}s\delta_x(X_i)\Big)$$
%$$\pi\Big(\prod_{i\in I}s\delta_x(X_i)\Big)\simeq\prod_{i\in I}\delta_x(X_i)$$
induite par les morphismes canoniques $\delta_x s(X_i)\to s\delta_x(X_i)$. Or ceux-ci sont des monomorphismes de conoyaux inclus dans $\mathbf{R}^1s(X_i)$, donc annihilés par $\tau_{r(d)}$ (voir la démonstration de la proposition~\ref{pr-pff}). La commutation de $\tau_{r(d)}$ aux produits et leur exactitude dans $\fct(\M,\A)$ permettent  d'en déduire que le morphisme canonique
$$\prod_{i\in I}\delta_x s(X_i)\to\prod_{i\in I}s\delta_x(X_i)$$
est un monomorphisme de conoyau annihilé par $\tau_{r(d)}$, donc en particulier stablement nul, ce qui montre que le morphisme~(\ref{mcds}) est un isomorphisme, comme souhaité.
 \end{proof}
 
\begin{rem}
Le même raisonnement permet de voir que les produits sont exacts dans $\Pol_d(\M,\A)$ s'ils le sont dans $\A$. Ce n'est généralement pas le cas dans $\st(\M,\A)$, même si $\M$ et $\A$ possèdent de très fortes propriétés de finitude et de régularité. En voici un exemple dans $\fct(\Theta,\mathbf{Ab})$. 
 
Pour tout entier $n\in\mathbb{N}$, posons $P_n:=\mathbb{Z}[\Theta(\mathbf{n},-)]$ et notons $F_n$ le noyau du morphisme $P_{n+1}\to P_0=\mathbb{Z}$ induit par l'unique morphisme $\mathbf{0}\to\mathbf{n+1}$. On note que le conoyau de ce morphisme est le foncteur $\mathbb{Z}_{\leq n}$ quotient du foncteur constant $\mathbb{Z}$ dont les valeurs sont $\mathbb{Z}$ sur $\mathbf{i}$ pour $i\leq n$ et $0$ pour $i>n$. Comme $\mathbb{Z}_{\leq n}$ est stablement nul, on dispose dans $\st(\Theta,\mathbf{Ab})$ d'une suite exacte
$$0\to\pi(F_n)\to\pi(P_{n+1})\to\pi(\mathbb{Z})\to 0$$
dont on va voir que le produit sur tous les $n\in\mathbb{N}$ n'est pas exact (alors que tous les objets en jeu sont noethériens et polynomiaux --- mais pas de degré borné).

Les $P_i$ sont $\s n(\Theta,\mathbf{Ab})$-fermés, il en est donc de même pour les foncteurs $F_n$. Par conséquent, le produit de nos suites exactes est l'image par le foncteur $\pi$ du produit des suites exactes
$$0\to F_n\to P_{n+1}\to\mathbb{Z}$$
de $\fct(\Theta,\mathbf{Ab})$, donc le conoyau du produit des morphismes $\pi(P_{n+1})\to\pi(\mathbb{Z})$ est isomorphe à l'image par $\pi$ de
$$\prod_{n\in\mathbb{N}}\mathbb{Z}_{\leq n}.$$
Or ce foncteur n'est {\em pas} stablement nul (ses valeurs ont la puissance du continu, mais son image par le foncteur $\kappa$ prend des valeurs dénombrables), comme annoncé.
\end{rem}

En raisonnant par récurrence sur le degré $d$ et en utilisant derechef que le morphisme canonique $\delta_xs(X)\to s\delta_x(X)$ devient un isomorphisme lorsqu'on applique $\tau_{r(d)}$, pour $X$ dans $\Pol_d(\M,\A)$, on déduit de la proposition~\ref{pr-plf} le corollaire suivant.

\begin{cor}\label{cor-psta}
 Soit $d\in\mathbb{N}$. On note $E(d,m)$ l'ensemble des suites finies {\em non ordonnées} (notées, par abus, de la même façon qu'une suite ordonnée les représentant) de longueur au plus $d$ d'éléments de $\mathbf{m}$ ; pour $a=(a_1,\dots,a_i)\in E(d,m)$ ($i\leq d$), on note $\Delta_a$ l'endofoncteur $\delta_{x_{a_1}}\dots\delta_{x_{a_i}}$ de $\fct(\M,\A)$ (qui, à isomorphisme canonique près, ne dépend pas de l'ordre de la suite).
  
 Alors il existe un objet $r(d)$ de $\M$ tel que, pour toute famille $(t_a)_{a\in E(d,m)}$ d'objets de $\M$, le foncteur
 $$\Pol_d(\M,\A)\to\A^{E(d,m)}\qquad X\mapsto\big(\Delta_a sX(t_a+r(d))\big)_{a\in E(d,m)}$$
 soit exact et fidèle. Il commute de plus aux colimites. Il commute également aux limites si les produits sont exacts dans $\A$.
\end{cor}

Avant de tirer une série de conséquences de ce résultat, on rappelle quelques définitions complémentaires de propriétés de finitude dans une catégorie abélienne $\C$. Un objet de $\C$ est dit {\em artinien} si toute suite décroissante de sous-objets d'icelui stationne (i.e. si l'objet correspondant de $\C^{op}$ est noethérien), {\em de longueur finie} s'il possède une filtration finie dont les sous-quotients sont simples (i.e. non nuls mais sans sous-objet non trivial). Il est classique qu'un objet est de longueur finie si et seulement s'il est à la fois noethérien et artinien. Tout objet noethérien $X$ d'une catégorie abélienne possède une {\em dimension de Krull}, qui est un ordinal noté $K\dim(X)$, qu'on peut définir comme la {\em déviation} de l'ensemble ordonné de ses sous-objets --- on renvoie le lecteur au chapitre~6 de \cite{MCR} pour la définition de la déviation d'un ensemble ordonné (1.2) et le fait qu'un ensemble ordonné dont toutes les suites croissantes stationnent possède une déviation (1.8). S'il existe une injection croissante d'une ensemble ordonné $A$ dans un autre $B$, la déviation de $A$ est au plus celle de $B$ ; un objet noethérien non nul d'une catégorie abélienne est de dimension de Krull nulle si et seulement s'il est de longueur finie. La dimension de Krull d'une catégorie localement noethérienne est le suprémum des dimensions de Krull de ses objets noethériens. Il existe aussi une notion de dimension de Krull plus générale dans les catégories abéliennes --- voir \cite{Gab}, chapitre~IV, §\,1 ---, qui n'est d'ailleurs pas définie pour toutes les catégories abéliennes, mais la précédente n'est appropriée que pour les catégories localement noethériennes.

Les deux premières assertions du lemme classique suivant sont immédiates. La dernière résulte de ce que tout foncteur d'une catégorie de Grothendieck vers $\mathbf{Ab}$
qui commute aux limites est représentable et de ce qu'une catégorie de Grothendieck est équivalente à une catégorie de modules si et seulement si elle possède un générateur projectif de type fini (théorie de Morita) : ces deux faits impliquent qu'une catégorie de Grothendieck est équivalente à une catégorie de modules si et seulement s'il existe un foncteur fidèle commutant aux limites et colimites de cette catégorie vers $\mathbf{Ab}$.

\begin{lm}\label{lm-classique}
 Soient $\Phi : \B\to\C$ un foncteur exact et fidèle entre catégories abéliennes et $X$ un objet de $\B$.
 \begin{enumerate}
  \item Si $\Phi(X)$ est un objet noethérien (resp. artinien, de longueur finie) de $\C$, alors $X$ est un objet noethérien (resp. artinien, de longueur finie) de $\B$. Si $\Phi(X)$ est noethérien, on a de plus $K\dim(X)\leq K\dim\Phi(X)$.
  \item Si $\Phi$ commute aux colimites et que $\Phi(X)$ est un objet de type fini de $\C$, alors $X$ est un objet de type fini de $\B$.
  \item Supposons que $\Phi$ commute aux limites et colimites, que $\B$ est une catégorie de Grothendieck et que $\C$ est équivalente à une catégorie de modules, alors $\B$ est équivalente à une catégorie de modules.
 \end{enumerate}
\end{lm}

En combinant les deux premières assertions de ce lemme au corollaire~\ref{cor-psta}, on obtient :

\begin{cor}\label{cor-dks}
 Soient $X$ un objet polynomial de $\st(\M,\A)$ et $\omega$ un ordinal. Supposons qu'il existe un objet $t$ de $\M$ tel que, pour tout objet $x$ de $\M$, $sX(t+x)$ soit un objet noethérien (resp. artinien, de longueur finie, de type fini, noethérien et de dimension de Krull au plus $\omega$) de $\A$. Alors $X$ est un objet noethérien (resp. artinien, de longueur finie, de type fini, noethérien et de dimension de Krull au plus $\omega$) de $\st(\M,\A)$.
\end{cor}

En particulier :
\begin{cor}\label{cor-keta}
  Soient $F$ un foncteur polynomial de $\widetilde{\M}$ dans $\A$ et $\omega$ un ordinal. Supposons que, pour tout objet $x$ de $\widetilde{\M}$, $F(x)$ soit un objet noethérien (resp. artinien, de longueur finie, de type fini, noethérien et de dimension de Krull au plus $\omega$) de $\A$. Alors $\pi\eta(F)$ est un objet noethérien (resp. artinien, de longueur finie, de type fini, noethérien et de dimension de Krull au plus $\omega$) de $\st(\M,\A)$.
\end{cor}

\begin{rem}
 L'assertion relative à la propriété noethérienne est aussi un corollaire (au moins si $\M$ vérifie $(P_1)$) de la proposition~\ref{noeth-eta}, parce que le foncteur $\pi$ préserve les objets noethériens. En revanche, dans $\fct(\Theta,\A)$, par exemple, l'image d'un foncteur non nul de $\fct(\widetilde{\Theta},\A)$ par le foncteur $\eta$ n'est {\em jamais} un foncteur artinien : les assertions des corollaires précédents relatives à la propriété artinienne ou à la dimension de Krull sont spécifiques à la catégorie $\st(\M,\A)$.
\end{rem}

Dans le résultat suivant, on considère la dimension de Krull de catégories de Grothendieck pas nécessairement localement noethériennes (cf. remarque~\ref{rq-ln}). On rappelle que cette dimension de Krull (lorsqu'elle est définie) est le suprémum des dimensions de Krull des objets de la catégorie en question. La sous-catégorie pleine des objets de dimension de Krull au plus $\omega$ (ordinal fixé) d'une catégorie de Grothendieck en est, par construction, une sous-catégorie localisante (voir \cite{Gab}, chap.~IV).

\begin{thm}\label{th_k}
 Soit $d\in\mathbb{N}$. Supposons que $\Pol_d(\widetilde{\M},\A)$ est engendrée par des foncteurs ponctuellement noethériens (ce qui implique que $\A$ est localement noethérienne, comme on le voit en considérant les foncteurs constants). Alors la catégorie $\Pol_d(\M,\A)$ a la même dimension de Krull que $\A$.
\end{thm}

\begin{proof}
 La considération de la sous-catégorie $\Pol_0(\M,\A)\simeq\A$ de $\Pol_d(\M,\A)$ montre que la dimension de Krull de cette catégorie (si elle est définie) est au moins égale à celle de $\A$, que nous noterons $\omega$. Le corollaire~\ref{cor-keta} montre que $\pi\eta(F)$ est de dimension de Krull au plus $\omega$ si $F\in {\rm Ob}\,\Pol_d(\widetilde{\M},\A)$ est ponctuellement noethérien ; l'hypothèse faite sur cette catégorie, ainsi que la commutation de $\pi\eta$ aux colimites, montre que cette propriété est vraie pour {\em tout} $F\in {\rm Ob}\,\Pol_d(\widetilde{\M},\A)$. La conclusion provient donc de ce que $\Pol_d(\M,\A)$ est la plus petite sous-catégorie épaisse de $\st(\M,\A)$ contenant l'image par $\pi\eta$ de $\Pol_d(\widetilde{\M},\A)$ (grâce à la proposition~3.9 de \cite{DV3}).
\end{proof}

\begin{rem}\label{rq-ln}
 Sous les hypothèses du théorème, la catégorie $\Pol_d(\widetilde{\M},\A)$ est localement noethérienne, mais pas nécessairement $\Pol_d(\M,\A)$ (les corollaires~\ref{cor-ln2} et~\ref{cor-ln3} donnent toutefois des conditions suffisantes pour qu'il en soit ainsi). Cela est illustré par l'exemple suivant.
 \end{rem}
 
 \begin{ex}\label{ex-lnb}
Soient $\C$ un objet de $\mn$ et $M : \C^{op}\to\mathbf{Ab}$ un foncteur additif. Notons $\C_M$ la catégorie ayant les mêmes objets que $\A$ et dont les morphismes sont donnés par
 $$\C_M(x,y):=\C(x,y)\times M(x),$$
 la composition étant donnée par
 $$\C_M(y,z)\times\C_M(x,y)\to\C_M(x,z)\qquad \big((f,u),(g,v)\big)\mapsto (f\circ_\C g,v+g^*u).$$
 
 Comme $M$ est monoïdal (au sens fort), la structure monoïdale symétrique sur $\C$ induit une structure monoïdale symétrique sur $\C_M$ et en fait un objet de $\mi$. Le foncteur d'oubli $\C_M\to\C$ est une flèche de $\mi$. En particulier, sa composée $\tilde{A}$ avec un foncteur additif $A : \C\to\A$ est un foncteur polynomial de degré fort $1$, qui appartient à l'image du foncteur $\eta$ (en effet, $\C$ appartenant à $\mn$, le foncteur d'oubli $\C_M\to\C$ se factorise par le foncteur canonique  $\C_M\to\widetilde{\C_M}$). Notons $K$ l'objet $M\underset{\C}{\otimes}A$ de $\A$ (par définition, c'est la cofin du bifoncteur produit tensoriel extérieur $M\boxtimes A : \C^{op}\times\C\to\A$). On définit un foncteur $B : \C_M\to\A$ par $B(x)=A(x)\oplus K$ sur les objets, $B$ envoyant un morphisme $(f,u)\in\C_M(x,y)$ sur le morphisme $A(x)\oplus K\to A(y)\oplus K$ dont :
 \begin{itemize}
  \item la composante $A(x)\to A(y)$ est $A(f)$ ;
  \item la composante $A(x)\to K$ est la composée de $A(x)\xrightarrow{u\otimes -}M(x)\otimes A(x)$ et du morphisme canonique $M(x)\otimes A(x)\to M\underset{\C}{\otimes}A$ ;
  \item la composante $K\to A(y)$ est nulle ;
  \item la composante $K\to K$ est l'identité.
\end{itemize}
On dispose d'une suite exacte
$$0\to K\to B\to\tilde{A}\to 0$$
dans $\fct(\C_M,\A)$ (où $K$ désigne le foncteur constant en $K$).

On vérifie aisément que cette extension est essentielle à droite au sens où tout sous-foncteur $F$ de $B$ tel que le morphisme $F\hookrightarrow B\twoheadrightarrow\tilde{A}$ soit surjectif est égal à $B$. Cela implique que $B$, et donc $\pi(B)$, est de type fini si $\tilde{A}$ l'est, condition qui est vérifiée si $A$ est ponctuellement noethérien. En revanche, si $K$ n'est pas un objet noethérien de $\A$, alors ni $B$ ni $\pi(B)$ ne sont noethériens.

Par ailleurs, le foncteur canonique $\widetilde{\C_M}\to\C$ (induit par le foncteur d'oubli $\C_M\to\C$) est \guillemotleft~souvent~\guillemotright\ une équivalence de catégories --- on vérifie notamment sans peine que c'est le cas si $\C=\widetilde{\Theta}$ ou si $\C$ est une catégorie additive.

Si l'on choisit $\C$ de ce type et de sorte que les catégories $\Pol_d(\C,\A)$ soient engendrées par des foncteurs ponctuellement noethériens lorsque $\A$ est une catégorie localement noethérienne (c'est vrai pour $\C=\widetilde{\Theta}$ ou $\C=\mathbf{P}(A)$ où $A$ est un anneau dont le groupe additif sous-jacent de type fini, par exemple), la considération d'un foncteur additif non nul et de type fini $A : \C\to\A$ et d'un foncteur additif $M : \C^{op}\to\mathbf{Ab}$ tels que $M\underset{\C}{\otimes}A$ soit non noethérien (par, exemple, une somme directe infinie de copies de foncteurs tels que ce produit tensoriel soit non nul, ce qui est toujours possible lorsque $A$ est non nul) procure un objet de type fini et non noethérien $\tilde{A}$ dans $\Pol_1(\C_M,\A)$.

On notera que, dans cet exemple, la catégorie $\s n(\C_M,\A)$ est réduite à $0$ (bien que $\C_M$ n'appartienne à $\mn$ que si $M$ est nul), car $\C$ appartient à $\mn$. Cet exemple est susceptible de variations en partant d'une catégorie $\C$ appartenant seulement à $\mi$ (les foncteurs monoïdaux covariant et contravariant $A$ et $M$ doivent alors être définis sur $\widetilde{\C}$).
\end{ex}

Voici quelques corollaires utiles du théorème~\ref{th_k}.

\begin{cor}\label{cor-k1}
 Supposons que $\A$ est localement noethérienne et que $\M$ est la catégorie $\Theta$, ou la catégorie $\mathbf{S}(A)$, où $A$ est un anneau dont le groupe abélien sous-jacent est de type fini. Alors, pour tout entier $d\geq 0$, $\Pol_d(\M,\A)$ a la même dimension de Krull que $\A$.
\end{cor}

Afin d'en déduire des résultats de dimension de Krull sur des {\em foncteurs} (et non des objets de $\st(\M,\A)$) polynomiaux, on donne un lemme aisé pour contrôler, dans certains cas, la dimension de Krull de $\s n(\M,\A)$.

\begin{lm}\label{lmks}
 \begin{enumerate}
  \item Si $\M$ vérifie la propriété $(EI)$ et que pour tout objet $x$ de $\M$ existe un objet $y$ de $\M$ tel que $x+y$ ne soit pas isomorphe à $x$, alors tous les objets de $\fct(\M,\A)$ de dimension de Krull nulle sont stablement nuls.
  \item Si $\M$ vérifie la propriété $(P_3)$, que tous les ensembles de morphismes de $\M$ sont finis et que $\A$ est localement finie, alors $\s n(\M,\A)$ est localement finie.
 \end{enumerate}
\end{lm}

\begin{proof}\begin{enumerate}
              \item Sur une catégorie $(EI)$, les foncteurs simples sont {\em atomiques}, c'est-à-dire nuls sur toutes les classes d'isomorphisme d'objets sauf une. L'autre hypothèse faite sur $\M$ assure qu'un foncteur atomique est stablement nul, d'où la première assertion, la sous-catégorie $\s n(\M,\A)$ de $\fct(\M,\A)$ étant localisante.
              \item La proposition~\ref{psn} montre qu'il suffit de vérifier qu'un foncteur {\em presque nul} $\M\to\A$ est localement fini. Cela découle de ce que $\A$ est localement finie et que les ensembles de morphismes de $\M$ sont finis (car cette condition implique qu'un foncteur de type fini sur $\M$ prend des valeurs de type fini).
             \end{enumerate}
\end{proof}

\begin{cor}\label{cor-k2}
 Supposons que la catégorie $\A$ est localement finie. Alors $\fct(\Theta,\A)$ est une catégorie localement noethérienne de dimension de Krull $1$, dont les objets de dimension de Krull nulle sont exactement ceux de $\s n(\Theta,\A)$.
\end{cor}

\begin{cor}\label{cor-k3}
 Supposons que la catégorie $\A$ est localement finie. Soit $A$ un anneau fini. Alors la catégorie des foncteurs faiblement analytiques (i.e. qui sont colimites de foncteurs faiblement polynomiaux) de $\fct(\mathbf{S}(A),\A)$ est une catégorie localement noethérienne de dimension de Krull $1$, dont les objets de dimension de Krull nulle sont exactement ceux de $\s n(\mathbf{S}(A),\A)$.
\end{cor}

Ce corollaire mérite quelques commentaires : le caractère localement noethérien (qui est le théorème~\ref{th-cf}) est surpassé par les travaux de Putman et Sam \cite{PSam} (au moins lorsque $A$ est commutatif), qui établissent que $\fct(\mathbf{S}(A),\A)$ est localement noethérienne si $\A$ l'est ($A$ étant un anneau fini). En revanche, l'assertion sur la dimension de Krull semble difficile à relier aux méthodes issues des bases de Gröbner utilisées dans \cite{PSam}. De plus, la question de la détermination de la dimension de Krull et des sous-quotients de la filtration de Krull de $\fct(\mathbf{S}(A),\A)$ constitue un problème difficile, toujours ouvert. Seule l'égalité $\K_0(\mathbf{S}(A),\A)=\s n(\mathbf{S}(A),\A)$ (on suppose ici que la catégorie $\A$ est localement finie) est aisée. Il est naturel de conjecturer que $\K_1(\mathbf{S}(A),\A)$ est exactement constituée des foncteurs faiblement analytiques, mais même cette question semble ardue. Pour les étages supérieurs de la filtration de Krull, au moins lorsque $A$ est un corps fini et $\A$ la catégorie des espaces vectoriels sur $A$, les conjectures de \cite{Dja} (§\,13.3 et 12.2), qui traite du cas analogue de la catégorie source des $A$-espaces vectoriels de dimension finie avec injections (mais aucun scindement donné dans la structure), donnent une idée de ce à quoi on peut s'attendre.

\medskip

Utilisant la dernière assertion du lemme~\ref{lm-classique} et le corollaire~\ref{cor-psta}, on obtient la propriété suivante :

\begin{pr}
 Supposons que $\A$ est équivalente à une catégorie de modules. Alors il en est de même pour $\Pol_d(\M,\A)$ pour tout $d\in\mathbb{N}$.
\end{pr}

\begin{ex}\label{ex-ann}
 Soient $R$ et $S$ deux anneaux et $M$ un $R$-module à droite, qu'on identifie au foncteur additif ${\rm Hom}_R(-,M) : \mathbf{P}(R)^{op}\to\mathbf{Ab}$. On vérifie que la catégorie $\Pol_1(\mathbf{P}(R)_{M},\mathbf{Mod}-S)$ (les notations sont celles de l'exemple~\ref{ex-lnb}) est équivalente à la catégorie des modules à droite sur l'anneau donné matriciellement par $\left(\begin{array}{cc}R^{op}\otimes S & M\otimes S\\                                                                                                                                                                                     0 & S                                                                                                                                                                                  \end{array}
\right)$.
\end{ex}

\begin{rem}\label{rq-rec}
 L'exemple précédent montre que, rien qu'en considérant le cas $d=1$, le diagramme de recollement liant $\Pol_0(\M,\A)\simeq\A$, $\Pol_d(\M,\A)$ et $\Pol_d(\M,\A)/\Pol_0(\M,\A)$ peut être assez général. Voici néanmoins une restriction à ce sujet : cette situation de recollement est toujours isomorphe à un recollement venant de la construction de MacPherson-Vilonen, au moins si $\A$ est une catégorie de modules, comme on le voit en appliquant le théorème du §\,1 de \cite{FP} (dont le §\,3 rappelle la construction en question). En effet, on peut appliquer le critère de ce théorème parce que :
 \begin{enumerate}
  \item notre proposition~\ref{pr-plf} montre que, si $a$ est un objet convenable de $\M$, alors le foncteur $\Pol_1(\M,\A)\to\A\qquad X\mapsto sX(a)$, qui est toujours une rétraction du foncteur d'inclusion $\A\simeq\Pol_0(\M,\A)\to\Pol_1(\M,\A)$, est {\em exact} ;
  \item l'adjoint à gauche de l'inclusion $\A\to\st(\M,\A)$ est donné par la composée $\st(\M,\A)\to\fct(\M,\A)\to\A$ du foncteur section et du foncteur colimite, et les objets de $\A\simeq\Pol_0(\M,\A)$ sont {\em acycliques} pour ce foncteur ($\M$ ayant un objet initial, son homologie à coefficients constants est nulle en degré strictement positif).
 \end{enumerate}
 
 Illustrons ce phénomène en revenant sur l'exemple~\ref{ex-lnb} (dont on conserve les notations), en supposant pour simplifier que la catégorie $\C$ est additive. La catégorie $\Pol_1(\C_M,\A)$ est équivalente à la catégorie de MacPherson-Vilonen $\D$ dont les objets sont les triplets $(A,V,\xi)$ constitués d'un foncteur additif $A : \C\to\A$, d'un objet $V$ de $A$ et d'un morphisme $\xi : M\underset{\C}{\otimes}A\to V$ de $\A$, les morphismes $(A,V,\xi)\to (A',V',\xi')$ étant les morphismes $(A\to A',V\to V')$ de $\fct(\C,\A)\times\A$ vérifiant la condition de compatibilité évidente à $\xi$ et $\xi'$. On vérifie que les foncteurs suivants sont quasi-inverses l'un de l'autre :
 \begin{enumerate}
  \item $\D\to\Pol_1(\C_M,\A)$ associant à un objet $(A,V,\xi)$ l'extension de $\tilde{A}$ par $V$ image par $\xi$ de l'extension $0\to M\underset{\C}{\otimes}A\to B\to\tilde{A}\to 0$ construite dans l'exemple~\ref{ex-lnb} ;
  \item $\Pol_1(\C_M,\A)\to\D$ associant à un objet $X$ le triplet constitué de son image $\alpha(X)$ dans $\Pol_1(\widetilde{\C_M},\A)\simeq\Pol_1(\C,\A)$, sa valeur en $0$ (on rappelle que $\st(\C_M,\M)\simeq\fct(\C_M,\A)$) et un morphisme canonique $M\underset{\C}{\otimes}\alpha(X)\to X(0)$ qu'on laisse au lecteur le soin de définir.
 \end{enumerate}
\end{rem}

\appendix

\section{Foncteurs polynomiaux $\s n(\M,\A)$-parfaits}

\begin{conv} Dans ce qui suit, $\M$ désigne un objet de $\mi$ vérifiant $(P_0)$ et $\A$ une catégorie de Grothendieck.
\end{conv}

La propriété suivante est classique et formelle (voir par exemple \cite{Gab}, chap.~III).

\begin{prdef}
 Soit $F$ un foncteur de $\fct(\M,\A)$. Les assertions suivantes sont équivalentes.
 \begin{enumerate}
  \item Pour tout foncteur $N$ de $\s n(\M,\A)$, on a ${\rm Ext}^*_{\fct(\M,\A)}(N,F)=0$ ;
  \item pour tout foncteur $G$ de $\fct(\M,\A)$, le morphisme naturel
  $${\rm Ext}^*_{\fct(\M,\A)}(G,F)\to {\rm Ext}^*_{\st(\M,\A)}(\pi G,\pi F)$$
  induit par le foncteur exact $\pi$ est un isomorphisme ;
  \item l'unité $F\to s\pi F$ de l'adjonction entre $s$ et $\pi$ est un isomorphisme, et $(\mathbf{R}^i s)(\pi F)=0$ pour tout entier $i>0$.
 \end{enumerate}

 Si ces conditions sont vérifiées, on dira que $F$ est {\em $\s n(\M,\A)$-parfait}, ou, simplement, parfait. On notera $\mathbf{Pft}(\M,\A)$ la sous-catégorie pleine de $\fct(\M,\A)$ formée de ces foncteurs.
\end{prdef}

Commençons par donner une propriété générale qui sera d'un usage courant dans la suite.

\begin{pr}\label{prdsk}
Pour tout objet $x$ de $\M$, l'isomorphisme canonique $\tau_x s\xrightarrow{\simeq}s\tau_x$ de foncteurs $\st(\M,\A)\to\fct(\M,\A)$ induit un isomorphisme $\tau_x (\mathbf{R}^n s)\xrightarrow{\simeq}(\mathbf{R}^n s)\tau_x$ pour tout $n\in\mathbb{N}$.
\end{pr}

\begin{proof}
On montre l'assertion par récurrence sur $n$, qu'on peut supposer non nul. Comme le foncteur $\tau_x$ est exact, un argument formel de décalage montre que l'isomorphisme canonique induit un {\em monomorphisme} $\tau_x (\mathbf{R}^n s)\hookrightarrow(\mathbf{R}^n s)\tau_x$, sous l'hypothèse de récurrence que le morphisme $\tau_x (\mathbf{R}^j s)\to(\mathbf{R}^j s)\tau_x$ est un isomorphisme pour $j<n$. Ce monomorphisme s'insère un triangle commutatif
$$\xymatrix{\mathbf{R}^n s\ar[rr]^-{i_x(\mathbf{R}^n s)}\ar[rrd]_-{(\mathbf{R}^n s)(i_x)} & & \tau_x(\mathbf{R}^n s)\ar@{^{(}->}[d]\\
& & (\mathbf{R}^n s)\tau_x
}$$
(car c'est vrai en degré nul, pour l'isomorphisme canonique $\tau_x s\xrightarrow{\simeq}s\tau_x$).
 
 Par conséquent, pour tout objet $X$ de $\st(\M,\A)$, la suite exacte longue de cohomologie associée à la suite exacte courte $0\to X\to\tau_x(X)\to\delta_x(X)\to 0$ fournit une suite exacte
 $$(\mathbf{R}^{n-1} s)(\tau_x X)\to(\mathbf{R}^{n-1} s)(\delta_x X)\to(\mathbf{R}^n s)(X)\xrightarrow{i_x}\tau_x(\mathbf{R}^n s)(X)$$
 qui montre que $\kappa_x(\mathbf{R}^n s)(X)$ est isomorphe au conoyau du morphisme naturel $(\mathbf{R}^{n-1} s)(\tau_x X)\to(\mathbf{R}^{n-1} s)(\delta_x X)$. Comme $\tau_x$, $\delta_x$, ainsi que $\mathbf{R}^{n-1} s$ par l'hypothèse de récurrence, commutent aux foncteurs de décalage, on en déduit que $\kappa_x(\mathbf{R}^n s)$ commute aux foncteurs de décalage. Mais $\mathbf{R}^n s$ est à valeurs dans $\s n(\M,\A)$, puisque $n>0$, de sorte que $\mathbf{R}^n s=\kappa(\mathbf{R}^n s)$ est la colimite de ses sous-foncteurs $\kappa_x(\mathbf{R}^n s)$. Comme les foncteurs de décalage commutent aux colimites, on en tire que $\mathbf{R}^n s$ commute aux décalages, ce qu'il fallait démontrer.
 \end{proof}

On en déduit aussitôt, par récurrence sur $n$ (en utilisant derechef l'argument de suite exacte longue et de stable nullité des valeurs de $\mathbf{R}^j s$ pour $j>0$ de la démonstration qui précède), l'énoncé suivant, qui étend la proposition~\ref{caract-sect}.

 \begin{cor}\label{cor-pft}
  Pour tout entier $n>0$ et tout foncteur $F : \M\to\A$, il y a équivalence entre :
 \begin{enumerate}
  \item l'unité $F\to s\pi(F)$ est un isomorphisme et $(\mathbf{R}^i s)(\pi F)=0$ pour $0<i<n$ ;
  \item pour toute suite finie $(x_1,\dots,x_r)$ d'objets de $\M$ avec $0\leq r\leq n$, on a $\kappa\delta_{x_1}\dots\delta_{x_r}(F)=0$.
 \end{enumerate}
 
 En particulier, $F$ est parfait si et seulement si $\kappa\delta_{x_1}\dots\delta_{x_r}(F)=0$ pour toute suite finie $(x_1,\dots,x_r)$ d'objets de $\M$.
 \end{cor}

\begin{rem}
 On déduit aussi aisément de la proposition~\ref{prdsk} que le corollaire~\ref{corsec} s'étend aux dérivés du foncteur section : ceux-ci commutent aux colimites filtrantes.
\end{rem}

La propriété suivante énumère les propriétés de base des foncteurs parfaits.

\begin{pr}\label{prfpe}
\begin{enumerate}
 \item Soit $0\to F\to G\to H\to 0$ une suite exacte de $\fct(\M,\A)$. Si deux des foncteurs $F$, $G$ et $H$ appartiennent à $\mathbf{Pft}(\M,\A)$, alors il en est de même du troisième.
 \item La sous-catégorie $\mathbf{Pft}(\M,\A)$ de $\fct(\M,\A)$ est stable par les endofoncteurs $\tau_x$ et $\delta_x$.
 \item Pour tout entier $d\in\mathbb{N}$, il existe un objet $r(d)$ de $\M$ tel que $\tau_{r(d)}F$ appartienne à $\mathbf{Pft}(\M,\A)$ pour tout foncteur $F$ faiblement polynomial de degré au plus $d$ et $\s n(\M,\A)$-fermé.
 \item Le degré fort et le degré faible d'un foncteur polynomial parfait $\M\to\A$ coïncident.
\end{enumerate}
\end{pr}

\begin{proof}
 Le premier point est immédiat. Le deuxième découle de la proposition~\ref{prdsk} (et du corollaire~\ref{cor-pft}). Le troisième se déduit des propositions~\ref{prn-princ} et~\ref{prdsk}. Le dernier provient du corollaire~\ref{cor-pft}.
\end{proof}

\begin{pr}\label{pr1-cp}
 Soit $\C$ une sous-catégorie pleine de $\fct(\M,\A)$ vérifiant les deux propriétés suivantes.
 \begin{enumerate}
  \item Pour tout foncteur polynomial $\s n(\M,\A)$-fermé $F$, il existe un objet $x$ de $\M$ tel que $\tau_x(F)$ appartienne à $\C$. 
  \item Si $0\to F\to G\to H\to 0$ est une suite exacte de $\fct(\M,\A)$ avec $G$ et $H$ dans $\C$, alors $F$ est dans $\C$.
 \end{enumerate}
Alors $\C$ contient tous les foncteurs polynomiaux de $\mathbf{Pft}(\M,\A)$.
\end{pr}

\begin{proof}
 On montre par récurrence sur l'entier $d$ que tout foncteur $F$ de $\mathbf{Pft}(\M,\A)$ polynomial de degré au plus $d$ appartient à $\C$. Pour $d<0$, il n'y a rien à démontrer, on suppose donc $d\geq 0$ et l'assertion établie pour le degré $d-1$. Choisissons un objet $x$ de $\M$ tel que $\tau_x(F)$ appartienne à $\C$. Comme $F$ est parfait, le morphisme canonique $F\to\tau_x(F)$ est un monomorphisme et son conoyau $\delta_x(F)$ est parfait. L'hypothèse de récurrence montre que $\delta_x(F)$ appartient à $\C$, de sorte que la deuxième hypothèse sur cette catégorie et la suite exacte $0\to F\to\tau_x(F)\to\delta_x(F)\to 0$ donnent la conclusion.
\end{proof}

\begin{pr}\label{pr2-cp}
 Soit $\C$ une sous-catégorie pleine de $\fct(\M,\A)$ vérifiant les trois propriétés suivantes.
 \begin{enumerate}
  \item L'image par le foncteur $\eta$ de tout foncteur polynomial $\widetilde{\M}\to\A$ appartient à $\C$. 
  \item Si $0\to F\to G\to H\to 0$ est une suite exacte de $\fct(\M,\A)$ avec $G$ et $H$ dans $\C$, alors $F$ est dans $\C$.
   \item Si $0\to F\to G\to H\to 0$ est une suite exacte de $\fct(\M,\A)$ avec $F$ et $H$ dans $\C$, alors $G$ est dans $\C$.
 \end{enumerate}
Alors $\C$ contient tous les foncteurs polynomiaux de $\mathbf{Pft}(\M,\A)$.
\end{pr}

\begin{proof}
 On montre par récurrence sur $d$ que, pour tout objet $X$ de $\Pol_d(\M,\A)$, il existe $x\in {\rm Ob}\,\M$ tel que $\tau_x(sX)$ appartienne à $\C$, ce qui permettra d'appliquer la proposition précédente pour conclure. On utilise la suite exacte canonique $0\to A\to X\to\pi\eta\alpha(X)\to B\to 0$, avec $A$ et $B$ dans $\Pol_{d-1}(\M,\A)$, qu'on scinde en deux suites exactes courtes :
 $$0\to A\to X\to Y\to 0\quad\text{et}\quad 0\to Y\to\pi\eta\alpha(X)\to B\to 0.$$
 En utilisant l'hypothèse de récurrence et la proposition~\ref{prn-princ}, on trouve un objet $x$ de $\M$ tel que $\tau_x(sA)$ et $\tau_x(sB)$ appartiennent à $\C$ et que les suites
  $$0\to\tau_x(sA)\to\tau_x(sX)\to\tau_x(sY)\to 0\;\text{et}\;0\to\tau_x(sY)\to\eta\alpha(\tau_x(X))\to\tau_x(sB)\to 0$$
 soient exactes. L'application successive des trois hypothèses permet de voir que $\tau_x(sY)$ et $\tau_x(sX)$ appartiennent à $\C$, comme souhaité.
\end{proof}

Dans le cas où $\M=\Theta$, on peut en fait s'abstenir de la deuxième condition dans la proposition précédente, comme le montre le résultat suivant, qui recoupe fortement les travaux de Nagpal \cite{Nag}, ainsi que \cite{Nag-prive}, \cite{NSS} et \cite{Ram}.

\begin{thm}\label{th-nag}
 Soit $F : \Theta\to\A$ un foncteur à support fini. Les assertions suivantes sont équivalentes :
 \begin{enumerate}
  \item $F$ est parfait ;
  \item $F$ possède une filtration finie dont les sous-quotients appartiennent à l'image essentielle du foncteur $\eta : \fct(\widetilde{\Theta},\A)\to\fct(\Theta,\A)$ ;
  \item il existe un entier $d$ tel que $F$ possède une résolution de la forme
  $$\cdots\to\eta(T_n)\to\eta(T_{n-1})\to\dots\to\eta(T_0)\to F\to 0$$
  où chaque foncteur $T_n : \widetilde{\Theta}\to\A$ est polynomial de degré au plus $d$.
 \end{enumerate}
 
 Lorsque $\A$ possède assez de projectifs, ces conditions équivalent encore à l'existence d'un entier $d$ tel que $F$ possède une résolution projective dont chaque terme a un support inclus dans $\{\mathbf{0},\dots,\mathbf{d}\}$.
\end{thm}

\begin{proof}
 Pour chaque entier $n$, notons $\rho_n$ l'inclusion dans $\Theta$ de la sous-catégorie pleine $\Theta_{\leq n}$ des $\mathbf{i}$ pour $i\leq n$ et $p_n$ l'endofoncteur de $\fct(\Theta,\A)$ image de la coünité $(\rho_n)_!\rho_n^*\to {\rm Id}$ (où $(\rho_n)_!$ désigne, comme dans le §\,\ref{sect-sk}, l'extension de Kan à gauche de $\rho_n$). On note que $p_n(F)$ est le plus grand sous-foncteur de $F$ appartenant à $\Pol_n^{{\rm fort}}(\Theta,\A)$, grâce aux propositions~\ref{pr-suppFini} et~\ref{theta-supp}. Le foncteur $p_n$ est additif et préserve les monomorphismes (comme l'identité dont c'est un sous-foncteur) et les épimorphismes (comme $(\rho_n)_!\rho_n^*$ dont c'est un quotient). En utilisant l'équivalence de catégories classique $\fct(\widetilde{\Theta},\A)\simeq\fct(\Sigma,\A)$, où $\Sigma$ désigne la catégorie des ensembles finis avec {\em bijections} (cette conséquence du théorème de Pirashvili à la Dold-Kan \cite{PDK} est rappelée dans le §\,4 de \cite{DV3} ; elle figure aussi, avec des notations différentes, dans \cite{CEF}, dont c'est le théorème~4.1.5), on voit qu'il existe un endofoncteur {\em exact} $\tilde{p}_n$ de $\fct(\widetilde{\Theta},\A)$ tel que $p_n\circ\eta\simeq\eta\circ\tilde{p}_n$. En utilisant cette même équivalence de catégories, on voit également que ${\rm Ext}^*_{\fct(\Theta,\A)}(\eta(A),\eta(B))$ est nul lorsque $A$ est de degré au plus $n$ et $B$ nul sur $\mathbf{n}$ (condition équivalente à $\tilde{p}_n(B)=0$), et naturellement isomorphe à ${\rm Ext}^*_{\fct(\widetilde{\Theta},\A)}(A,B)$ lorsque $A$ est de degré au plus $n$ et $B$ nul sur $\mathbf{n+1}$. On en déduit aisément que $F$ possède une filtration finie dont les sous-quotients sont dans l'image essentielle de $\eta$ si et seulement si, pour tout $n\in\mathbb{N}$, $p_n(F)/p_{n-1}(F)$ est isomorphe à l'image par $\eta$ d'un foncteur de support $\{\mathbf{n}\}$, c'est-à-dire à un foncteur du type $\Phi_n(M):=\mathbb{Z}[\Theta(\mathbf{n},-)]\underset{\Sigma_n}{\otimes}M$ où $M$ est un objet de $\A$ muni d'une action du groupe symétrique $\Sigma_n$.
 
 Cette observation permet de voir que la classe $\C$ des foncteurs à support fini $\Theta\to\A$ qui possèdent une telle filtration est stable par noyaux d'épimorphismes. Supposons en effet que $0\to X\to F\xrightarrow{f} G\to 0$ est une suite exacte de $\fct(\Theta,\A)$ avec $F$ et $G$ dans $\C$. Pour chaque entier $n$, $p_n(f) : p_n(F)\to p_n(G)$ est un épimorphisme dont le noyau $p_n(F)\cap X(\subset F)$ est {\em parfait}. En effet, $p_n(F)$ et $p_n(G)$, qui appartiennent à $\C$ d'après la caractérisation précédente, sont parfaits (puisque la classe des foncteurs parfaits contient l'image essentielle de $\eta$ et est stable par extensions, par la proposition~\ref{prfpe}), et la classe des foncteurs parfaits est stable par noyaux d'épimorphismes (par la même proposition). Comme $p_n(F)\cap X$ est faiblement polynomial de degré au plus $n$ (car c'est un sous-foncteur de $p_n(F)$), on en déduit qu'il est également polynomial de degré {\em fort} au plus $n$ (d'après la proposition~\ref{prfpe}), donc inclus dans $p_n(X)$. Par conséquent, on a $p_n(F)\cap X=p_n(X)$, ce qui montre que la suite $0\to p_n(X)\to p_n(F)\xrightarrow{p_n(f)} p_n(G)\to 0$ est {\em exacte} pour tout $n$, donc également $0\to p_n(X)/p_{n-1}(X)\to p_n(F)/p_{n-1}(F)\to p_n(G)/p_{n-1}(G)\to 0$. Mais par hypothèse il existe des objets $M$ et $N$ de $\A_{\Sigma_n}$ tels que $p_n(F)/p_{n-1}(F)\simeq\Phi_n(M)$ et $p_n(G)/p_{n-1}(G)\simeq\Phi_n(N)$ ; comme $\Phi_n$ est exact et pleinement fidèle, on en déduit que $p_n(X)/p_{n-1}(X)$ est isomorphe à l'image par $\Phi_n$ du noyau d'un certain épimorphisme $M\to N$ de $\A_{\Sigma_n}$, ce qui montre que $X$ appartient à $\C$. Comme cette classe est également stable par extensions et contient l'image par $\eta$ des foncteurs polynomiaux (qui, dans $\fct(\Theta,\A)$, sont exactement les foncteurs à support fini), la proposition~\ref{pr2-cp} montre que $\C$ contient tous les foncteurs parfaits. Ainsi, les deux premières conditions de l'énoncé sont équivalentes.
 
 Si $F$ appartient à $\Pol_d^{{\rm fort}}(\Theta,\A)$, on peut trouver un épimorphisme du type $G:=\underset{i\leq d}{\bigoplus}\mathbb{Z}[\Theta(\mathbf{i},-)]\otimes A_i\twoheadrightarrow F$ (en utilisant de nouveau les propositions~\ref{pr-suppFini} et~\ref{theta-supp}) ; si $\A$ a assez de projectifs, on peut même supposer que les $A_i$ sont projectifs. On note que $G$ appartient à l'image par $\eta$ d'un foncteur de $\Pol_d(\widetilde{\Theta},\A)$, et est projectif si les $A_i$ le sont. Si l'on suppose de plus que $F$ est parfait, comme $G$ l'est également, le noyau de cet épimorphisme $G\twoheadrightarrow F$ est encore un foncteur parfait, de degré fort au plus $d$, puisque le degré fort d'un foncteur parfait coïncide avec son degré faible et que $G$ est de degré au plus $d$. On en déduit par récurrence que tout foncteur parfait vérifie la troisième propriété de l'énoncé, et même le renforcement en termes de projectifs si $\A$ a assez de projectifs.
 
 Il reste à montrer qu'un foncteur $F$ vérifiant la troisième condition est parfait. En fait, il suffit de disposer d'une suite exacte du type
 $$\eta(T_{d+1})\to\cdots\to\eta(T_n)\to\eta(T_{n-1})\to\dots\to\eta(T_0)\to F\to 0$$
 avec tous les $T_i$ de degré au plus $i$. Cela se voit en considérant les suites spectrales d'hypercohomologie associées, via l'application de $(\mathbf{R}^*s)\pi$, et les deux faits suivants :
 \begin{enumerate}
  \item les foncteurs dans l'image de $\eta$ sont parfaits ;
  \item on a $(\mathbf{R}^i s)(X)=0$ pour $i>d$ si $X$ appartient à $\Pol_d(\Theta,\A)$ (cf. proposition~\ref{prn-princ}).
 \end{enumerate}
\end{proof}

Le théorème~\ref{th-nag} doit beaucoup à l'étude du travail \cite{GP-cow} de Powell (la deuxième propriété, qui traduit que $F$ est un foncteur {\em \#-filtré} au sens de \cite{Nag}, est similaire à celle de {\em foncteur $DJ$-bon} dans le contexte de \cite{GP-cow}, tandis que la dernière propriété, et surtout sa variante dans le cas où $\A$ possède assez de projectifs, constitue un analogue direct du critère homologique de \cite{GP-cow} pour la description des foncteurs $DJ$-bons), dont les considérations s'inspirent des {\em catégories de plus haut poids} introduites par Cline, Parshall et Scott dans \cite{CPS}.

\begin{rem}
 En général, saturer par extensions l'image essentielle du foncteur $\eta$ ne suffit nullement à obtenir tous les foncteurs polynomiaux parfaits $\M\to\A$ (même en leur imposant de fortes hypothèses de finitude). En voici un exemple sur des foncteurs $\mathbf{S}(\mathbb{Q})\to\mathbf{Mod}-\mathbb{Q}$ fortement polynomiaux de degré $3$. Les foncteurs $F : V\mapsto V^*\otimes V\otimes V$ (où l'étoile indique la dualité dans $\mathbb{Q}$-espaces vectoriels) et $A : V\mapsto V$ appartiennent à l'image du foncteur $\eta$, et l'on dispose d'un {\em épimorphisme} de foncteurs $F\twoheadrightarrow A$ donné sur $V$ par $V\mapsto V^*\otimes V\otimes V\xrightarrow{{\rm tr}\otimes V}V$, où tr désigne la trace. Son noyau est un foncteur polynomial parfait de degré $3$, de longueur finie, mais on vérifie sans peine, en utilisant la proposition~6.8 de \cite{DV3}, qu'il ne s'obtient pas par extensions successives de foncteurs appartenant à l'image essentielle de $\eta$.
\end{rem}

\bibliographystyle{plain}
\bibliography{b-npol.bib}
\end{document}